\documentclass[10pt,a4paper]{article}
\usepackage[utf8]{inputenc}
\usepackage[T1]{fontenc}
\usepackage{geometry}
\usepackage{amsthm}
\usepackage[french,english]{babel}
\usepackage{amssymb}
\usepackage{lmodern}
\usepackage{amsmath,amsfonts}
\usepackage{mathrsfs} 
\usepackage{dsfont}
\usepackage{stmaryrd}
\DeclareMathOperator{\dive}{div} 
\usepackage{graphicx}
\usepackage{fullpage}
\usepackage[nottoc,notlot,notlof]{tocbibind}
\usepackage[colorlinks=true,urlcolor=black,linkcolor=black,citecolor=black]{hyperref}
\numberwithin{equation}{section}
\usepackage{mathtools,array}
\usepackage{verbatimbox}
\usepackage{color}
\usepackage{bm}
\usepackage{tikz-cd}
\usepackage{pst-node}
\usetikzlibrary{matrix}
\usepackage{mathtools}
\usepackage{verbatimbox}
\usepackage{diagbox}
\usepackage{array}
\newcolumntype{C}{>{$\displaystyle} c <{$}}
\usepackage{makecell}
\usepackage{float}
\usepackage{tabu}
\usepackage{cancel}
\usepackage{lscape}
\usepackage[babel=true]{csquotes}
\usepackage{bm}
\usepackage{xcolor}
\makeatletter
\def\env@dmatrix{\hskip -\arraycolsep
	\let\@ifnextchar\new@ifnextchar
	\def\arraystretch{2}%
	\array{*{\c@MaxMatrixCols}{>{\displaystyle}c}}%
}
\usepackage{soul}

\makeatother

\DeclareFontShape{OMX}{cmex}{m}{n}{
	<-7.5> cmex7
	<7.5-8.5> cmex8
	<8.5-9.5> cmex9
	<9.5-> cmex10
}{}
\SetSymbolFont{largesymbols}{normal}{OMX}{cmex}{m}{n}
\SetSymbolFont{largesymbols}{bold}  {OMX}{cmex}{m}{n}

\newcommand{\ad}[1]{\overline{#1}}
\usepackage{enumitem}

\def\ii{\mathfrak{i}}

\begin{document}

	\renewcommand{\thefootnote}{\fnsymbol{footnote}}
	
	\title{The Loewner Energy via the Renormalised Energy of \\  Moving Frames}

	\author{Alexis Michelat\footnote{EPFL B, Station 8, CH-1015 Lausanne, Switzerland \hspace{.5em} \href{alexis.michelat@epfl.ch}{alexis.michelat@epfl.ch}/ \href{mailto:alexis.michelat@normalesup.org}{alexis.michelat@normalesup.org}}\and Yilin Wang\footnote{Institut des Hautes \'Etudes Scientifiques, Bures-sur-Yvette, France \hspace{.6em}\href{mailto:yilin@ihes.fr}{yilin@ihes.fr}}}
	\date{\today}
	
	\maketitle
	
	\vspace{-0.5em}
	
	\begin{abstract}
		We obtain a new formula for the Loewner energy of Jordan curves 
		on the sphere, which is a K\"ahler potential for the essentially unique K\"ahler metric on the Weil-Petersson universal Teichm\"uller space, as the renormalised energy of moving frames on the two domains of the sphere delimited by the given curve. 
	\end{abstract}

	\tableofcontents
	\vspace{0cm}
	\begin{center}
	\end{center}

	\theoremstyle{plain}
	\newtheorem*{theorem*}{Theorem}
	\newtheorem{theorem}{Theorem}[section]
	\newenvironment{theorembis}[1]
	{\renewcommand{\thetheorem}{\ref{#1}$'$}
		\addtocounter{theorem}{-1}%
		\begin{theorem}}
		{\end{theorem}}
	\renewcommand*{\thetheorem}{\Alph{theorem}}
	\newtheorem{lemme}[theorem]{Lemma}
	\newtheorem*{lemme*}{Lemma}
	\newtheorem{propdef}[theorem]{Definition-Proposition}
	\newtheorem*{propdef*}{Definition-Proposition}
	\newtheorem{prop}[theorem]{Proposition}
	\newtheorem{cor}[theorem]{Corollary}
	\theoremstyle{definition}
	\newtheorem*{definition}{Definition}
	\newtheorem{defi}[theorem]{Definition}
	\newtheorem{rem}[theorem]{Remark}
	\newtheorem*{rem*}{Remark}
	\newtheorem{rems}[theorem]{Remarks}
	\newtheorem{remimp}[theorem]{Important Remark}
	\newtheorem{exemple}[theorem]{Example}
	\newtheorem{defi2}{Definition}
	\newtheorem{propdef2}[defi2]{Proposition-Definition}
	\newtheorem{remintro}[defi2]{Remark}
	\newtheorem{remsintro}[defi2]{Remarks}
	\newtheorem{conj}{Conjecture}
	\newtheorem{question}{Open Question}
	\renewcommand\hat[1]{%
		\savestack{\tmpbox}{\stretchto{%
				\scaleto{%
					\scalerel*[\widthof{\ensuremath{#1}}]{\kern-.6pt\bigwedge\kern-.6pt}%
					{\rule[-\textheight/2]{1ex}{\textheight}}%WIDTH-LIMITED BIG WEDGE
				}{\textheight}% 
			}{0.5ex}}%
		\stackon[1pt]{#1}{\tmpbox}
	}
	\parskip 1ex
	\newcommand{\totimes}{\ensuremath{\,\dot{\otimes}\,}}
	\newcommand{\vc}[3]{\overset{#2}{\underset{#3}{#1}}}
	\newcommand{\conv}[1]{\ensuremath{\underset{#1}{\longrightarrow}}}
	\newcommand{\A}{\ensuremath{\vec{A}}}
	\newcommand{\B}{\ensuremath{\vec{B}}}
	\newcommand{\C}{\ensuremath{\mathbb{C}}}
	\newcommand{\D}{\ensuremath{\nabla}}
	\newcommand{\Disk}{\ensuremath{\mathbb{D}}}
	\newcommand{\E}{\ensuremath{\vec{E}}}
	\newcommand{\I}{\ensuremath{\mathbb{I}}}
	\newcommand{\Q}{\ensuremath{\vec{Q}}}
	\newcommand{\loc}{\ensuremath{\mathrm{loc}}}
	\newcommand{\z}{\ensuremath{\bar{z}}}
	\newcommand{\hh}{\ensuremath{\mathscr{H}}}
	\newcommand{\h}{\ensuremath{\vec{h}}}
	\newcommand{\vol}{\ensuremath{\mathrm{vol}}}
	\newcommand{\hs}[3]{\ensuremath{\left\Vert #1\right\Vert_{\mathrm{H}^{#2}(#3)}}}
	\newcommand{\R}{\ensuremath{\mathbb{R}}}
	\renewcommand{\P}{\ensuremath{\mathbb{P}}}
	\newcommand{\N}{\ensuremath{\mathbb{N}}}
	\newcommand{\Z}{\ensuremath{\mathbb{Z}}}
	\newcommand{\p}[1]{\ensuremath{\partial_{#1}}}
	\newcommand{\Res}{\ensuremath{\mathrm{Res}}}
	\newcommand{\lp}[2]{\ensuremath{\mathrm{L}^{#1}(#2)}}
	\renewcommand{\wp}[3]{\ensuremath{\left\Vert #1\right\Vert_{\mathrm{W}^{#2}(#3)}}}
	\newcommand{\wpn}[3]{\ensuremath{\Vert #1\Vert_{\mathrm{W}^{#2}(#3)}}}
	\newcommand{\np}[3]{\ensuremath{\left\Vert #1\right\Vert_{\mathrm{L}^{#2}(#3)}}}
	\newcommand{\hp}[3]{\ensuremath{\left\Vert #1\right\Vert_{\mathrm{H}^{#2}(#3)}}}
	\newcommand{\ck}[3]{\ensuremath{\left\Vert #1\right\Vert_{\mathrm{C}^{#2}(#3)}}}
	\newcommand{\hardy}[2]{\ensuremath{\left\Vert #1\right\Vert_{\mathscr{H}^{1}(#2)}}}
	\newcommand{\lnp}[3]{\ensuremath{\left| #1\right|_{\mathrm{L}^{#2}(#3)}}}
	\newcommand{\npn}[3]{\ensuremath{\Vert #1\Vert_{\mathrm{L}^{#2}(#3)}}}
	\newcommand{\nc}[3]{\ensuremath{\left\Vert #1\right\Vert_{C^{#2}(#3)}}}
	\renewcommand{\Re}{\ensuremath{\mathrm{Re}\,}}
	\renewcommand{\Im}{\ensuremath{\mathrm{Im}\,}}
	\newcommand{\dist}{\ensuremath{\mathrm{dist}}}
	\newcommand{\diam}{\ensuremath{\mathrm{diam}\,}}
	\newcommand{\leb}{\ensuremath{\mathscr{L}}}
	\newcommand{\supp}{\ensuremath{\mathrm{supp}\,}}
	\renewcommand{\phi}{\ensuremath{\vec{\Phi}}}
	\renewcommand{\H}{\ensuremath{\vec{H}}}
	\renewcommand{\L}{\ensuremath{\vec{L}}}
	\renewcommand{\lg}{\ensuremath{\mathscr{L}_g}}
	\renewcommand{\ker}{\ensuremath{\mathrm{Ker}}}
	\renewcommand{\epsilon}{\ensuremath{\varepsilon}}
	\renewcommand{\bar}{\ensuremath{\overline}}
	\newcommand{\s}[2]{\ensuremath{\langle #1,#2\rangle}}
	\newcommand{\pwedge}[2]{\ensuremath{\,#1\wedge#2\,}}
	\newcommand{\bs}[2]{\ensuremath{\left\langle #1,#2\right\rangle}}
	\newcommand{\scal}[2]{\ensuremath{\langle #1,#2\rangle}}
	\newcommand{\sg}[2]{\ensuremath{\left\langle #1,#2\right\rangle_{\mkern-3mu g}}}
	\newcommand{\n}{\ensuremath{\vec{n}}}
	\newcommand{\ens}[1]{\ensuremath{\left\{#1\right\}}}
	\newcommand{\lie}[2]{\ensuremath{\left[#1,#2\right]}}
	\newcommand{\g}{\ensuremath{g}}
	\newcommand{\dzeta}{\ensuremath{\det\hphantom{}_{\kern-0.5mm\zeta}}}
	\newcommand{\e}{\ensuremath{\vec{u}}}
	\newcommand{\f}{\ensuremath{\vec{v}}}
	\newcommand{\ig}{\ensuremath{|\vec{\mathbb{I}}_{\phi}|}}
	\newcommand{\ik}{\ensuremath{\left|\mathbb{I}_{\phi_k}\right|}}
	\newcommand{\w}{\ensuremath{\vec{w}}}
	\newcommand{\hooklongrightarrow}{\lhook\joinrel\longrightarrow}
	\renewcommand{\tilde}{\ensuremath{\widetilde}}
	\newcommand{\vg}{\ensuremath{\mathrm{vol}_g}}
	\newcommand{\im}{\ensuremath{\mathrm{W}^{2,2}_{\iota}(\Sigma,N^n)}}
	\newcommand{\imm}{\ensuremath{\mathrm{W}^{2,2}_{\iota}(\Sigma,\R^3)}}
	\newcommand{\timm}[1]{\ensuremath{\mathrm{W}^{2,2}_{#1}(\Sigma,T\R^3)}}
	\newcommand{\tim}[1]{\ensuremath{\mathrm{W}^{2,2}_{#1}(\Sigma,TN^n)}}
	\renewcommand{\d}[1]{\ensuremath{\partial_{x_{#1}}}}
	\newcommand{\dg}{\ensuremath{\mathrm{div}_{g}}}
	\renewcommand{\Res}{\ensuremath{\mathrm{Res}}}
	\newcommand{\un}[2]{\ensuremath{\bigcup\limits_{#1}^{#2}}}
	\newcommand{\res}{\mathbin{\vrule height 1.6ex depth 0pt width
			0.13ex\vrule height 0.13ex depth 0pt width 1.3ex}}
	\newcommand{\ala}[5]{\ensuremath{e^{-6\lambda}\left(e^{2\lambda_{#1}}\alpha_{#2}^{#3}-\mu\alpha_{#2}^{#1}\right)\left\langle \nabla_{\e_{#4}}\vec{w},\vec{\mathbb{I}}_{#5}\right\rangle}}
	\setlength\boxtopsep{1pt}
	\setlength\boxbottomsep{1pt}
	\newcommand\norm[1]{%
		\setbox1\hbox{$#1$}%
		\setbox2\hbox{\addvbuffer{\usebox1}}%
		\stretchrel{\lvert}{\usebox2}\stretchrel*{\lvert}{\usebox2}%
	}
	\allowdisplaybreaks
	\newcommand*\mcup{\mathbin{\mathpalette\mcapinn\relax}}
	\newcommand*\mcapinn[2]{\vcenter{\hbox{$\mathsurround=0pt
				\ifx\displaystyle#1\textstyle\else#1\fi\bigcup$}}}
	\def\Xint#1{\mathchoice
		{\XXint\displaystyle\textstyle{#1}}%
		{\XXint\textstyle\scriptstyle{#1}}%
		{\XXint\scriptstyle\scriptscriptstyle{#1}}%
		{\XXint\scriptscriptstyle\scriptscriptstyle{#1}}%
		\!\int}
	\def\XXint#1#2#3{{\setbox0=\hbox{$#1{#2#3}{\int}$ }
			\vcenter{\hbox{$#2#3$ }}\kern-.58\wd0}}
	\def\ddashint{\Xint=}
	\newcommand{\dashint}[1]{\ensuremath{{\Xint-}_{\mkern-10mu #1}}}
	\newcommand\ccancel[1]{\renewcommand\CancelColor{\color{red}}\cancel{#1}}
	\newcommand\colorcancel[2]{\renewcommand\CancelColor{\color{#2}}\cancel{#1}}
	\newcommand{\abs}[1]{\left\lvert #1 \right \rvert}
	
	\renewcommand{\thetheorem}{\thesection.\arabic{theorem}}
    	
    	\section{Introduction}\label{intro}
    	
    	\subsection{Background on Weil-Petersson quasicircles}
    	
    	In \cite{yilingem,yilinloop}, the second author and Steffen Rohde introduced the M\"obius-invariant Loewner energy  
    	to measure the roundness of Jordan curves on the Riemann sphere $\C \cup \{\infty\}$ using the Loewner transform~\cite{Loewner23}. 
    	The original motivation comes from the probabilistic theory of Schramm-Loewner evolutions, see, e.g., \cite{yilin_survey} for an overview.
     It was proved in \cite{yilinvention} that the Loewner energy is proportional to the \emph{universal Liouville action} introduced by Takhtajan and Teo \cite{takteo}. In particular, the class of {finite energy curves} corresponds exactly to the  \emph{Weil-Petersson} class of quasicircles which has already been studied extensively by both physicists and mathematicians since the eighties, see, \emph{e.g.}, \cite{BowickRajeev1987string,witten88,NagVerjovsky,Nag_Sullivan,STZ_KdV,cui,takteo,sharon20062d,Figalli_circle,shenweil,GGPPR,bishop3,yilinterplay,VW2,johansson2021strong}, 
and is still an active research area.  See the introduction of \cite{bishop3} (see also
the companion papers \cite{bishop_ft} and \cite{bishop_travel} for more on this topic) for a summary and a list of 
equivalent definitions of very different nature. 

    	In this article, we sometimes view Jordan curves as curves on $S^2 \subset \R^3$ and give new characterisations of the Loewner energy in terms of the moving frames on $S^2$. 
    	Note that in this article, $S^2$ refers to the sphere of radius $1$ centred at the origin in $\R^3$ equipped with the induced round metric $g_0$ from its embedding into $\R^3$. Therefore, $S^2$ is isometric to $\widehat{\C}=\C\cup\ens{\infty}$ endowed with the metric
    	\begin{align*}
    	g_{\widehat{\C}}=\frac{4|dz|^2}{(1+|z|^2)^2}
    	\end{align*}
    	by the stereographic projection.
    	To distinguish the two setups, we will let $\gamma$ denote a Jordan curve in 
    	$\widehat{\C}$ and let $\Gamma$ denote a Jordan curve in $S^2$.  Let us 
      first list a few equivalent definitions of Weil-Petersson quasicircles that are 
     relevant to this work.

    	\begin{theorem}[Cui, \cite{cui}, Tahktajan-Teo, \cite{takteo}, Shen, \cite{shenweil}, Bishop, \cite{bishop3}]\label{thm:equi_WP}
    	Let $\gamma \subset \mathbb C$ be a Jordan curve, $\Omega$ be the bounded connected component of $\C\setminus\gamma$, and let $f:\mathbb{D}\rightarrow \Omega$ and $g:\C\setminus\bar{\mathbb{D}}\rightarrow \C\setminus\bar{\Omega}$ be biholomorphic maps such that $g (\infty) = \infty$.  The following conditions are equivalent:
    	\begin{enumerate}[topsep=-.1em,itemsep=-.3em]
    	\item[\emph{(1)}] There exists a quasiconformal extension of $g$ to $\C$ such that the Beltrami coefficient $\mu = \dfrac{\partial_{\ad z} g}{\partial_z g} :\mathbb D \rightarrow \mathbb D$ of $g|_{\mathbb D}$ satisfies $\displaystyle\int_{\mathbb{D}}|\mu(z)|^2\frac{|dz|^2}{(1-|z|^2)^2}<\infty$.
    	\item[\emph{(2)}] $\displaystyle\int_{\mathbb{D}}|\D\log|f'(z)||^2|dz|^2=\int_{\mathbb{D}}\left|\frac{f''(z)}{f'(z)}\right|^2|dz|^2<\infty$.
    	 \item[\emph{(3)}] $\displaystyle \int_{\C\setminus \bar{\mathbb{D}}}\left|\frac{g''(z)}{g'(z)}\right|^2|dz|^2<\infty$.
    	\item[\emph{(4)}] The \emph{(conformal) welding function} $\varphi=g^{-1}\circ f|_{S^1}$ satisfies $\log \varphi'$ belongs to the Sobolev space $H^{1/2}(S^1)$.
        \item[\emph{(5)}] The curve $\gamma$ is chord-arc and the unit tangent $\tau:\gamma\rightarrow S^1$ belongs to $H^{1/2}(\gamma)$.
        \item[\emph{(6)}] Every minimal surface $\Sigma\subset\mathbb{H}^3 \simeq \mathbb C \times \mathbb R_{+}^{\ast}$ with asymptotic boundary $\gamma$
    				has finite renormalised area, \emph{i.e.},
    				\begin{align}\label{renormalisedarea}
    				{\mathcal{RA}(\Sigma) = \lim\limits_{\varepsilon \rightarrow 0}\left(\mathrm{Area}(\Sigma_\varepsilon)-\mathrm{Length}(\partial \Sigma_\varepsilon)\right) =-2\pi\chi(\Sigma)-\int_{\Sigma}|\mathring{A}|^2d\mathrm{vol}_{\Sigma}>-\infty, }
    				\end{align}
    				where	for all {$\varepsilon>0$,
		$	\Sigma_\varepsilon=\Sigma\cap\ens{(z,t): t>\varepsilon}$ and $\partial \Sigma_\varepsilon=\Sigma\cap\ens{(z,t): t=\varepsilon}$.}
    	\end{enumerate}
    If $\gamma$ satisfies any of those conditions, $\gamma$ is called a \emph{Weil-Petersson quasicircle}.
    	\end{theorem}   	
    	The equivalences $(1)$ and $(2)$ are due to Cui, and independently to Takhtajan and Teo who proved the equivalences $(1),(2),(3)$. 
    	 In $(4)$, the continuous extension of $f, g$ to $S^1$ is well-defined by a classical theorem of Carathéodory \cite{caratheodory}. 
    	The equivalence between $(1)$ and $(4)$ 
     is proved by Shen.
    The second condition is perhaps the simplest one since it corresponds to the condition $\log|f'|\in W^{1,2}(\mathbb{D})$,  the Sobolev space of functions with squared-integrable weak derivatives.

For $(5)$, we recall that a Jordan curve is \emph{chord-arc} if there exists $K<\infty$ such that for all $x,y\in  \gamma$, we have $\ell(x,y)\leq K|x-y|$, where $\ell(x,y)$ is the length of the shortest arc joining $x$ to $y$. We mention that Weil-Petersson quasicircles are not only chord-arc but even \emph{asymptotically smooth}, namely, the ratio $\ell(x,y)/|x-y|$ tends to $1$ as $x$ tends to $y$. 
These curves are not necessarily $C^1$ for they allow certain types of infinite spirals. See Section~\ref{geodesicprop} for an explicit construction of such spirals.
    	    	Recall that for any Jordan chord-arc curve $\gamma$, a function $u:\gamma\rightarrow \C$ belongs to the Sobolev space $H^{1/2}(\gamma)$ if and only if 
    	\begin{align}\label{h12norm}
    	    \int_{\gamma}\int_{\gamma}\left|\frac{u(z)-u(w)}{z-w}\right|^2|dz||dw|<\infty,
    	\end{align}
    	where $|dz|$ is the arc-length measure.

    		The equivalence between ($1$) and  ($5$) was proven by Y. Shen and L. Wu (\cite{shenweilIII}; see also \cite{shenflow,shenweil,shenweilII,shenweilIIcor}), and also by Christopher Bishop \cite{bishop3}.
    		The last characterisation ($6$) due to Bishop \cite{bishop3} using the notion of renormalised area was first investigated for Willmore surfaces by S. Alexakis and R. Mazzeo (\cite{alexakismazzeo1}, \cite{alexakismazzeo2}) which has strong motivations arising from string theory \cite{GW99}.
    			The integral of the squared trace-free second fundamental form $\mathring A$ in $(6)$ is {
    			 {\em the Willmore energy of $\Sigma$} which is of particular interest for being conformally invariant}. 
    Amongst the important previous contribution that inspired this work, we should mention Epstein's work (\cite{epstein1}, \cite{epstein2}).

Not only we can characterize this class of curves qualitatively, as listed above, there is an important quantity associated with each element of the class.
Indeed, after appropriate normalisation, the class of Weil-Petersson quasicircles can be identified with the Weil-Petersson universal Teichm\"uller space $T_0(1)$ via conformal welding. Takhtajan and Teo \cite{takteo} showed that $T_0(1)$ carries an essentially unique homogeneous K\"ahler metric and introduced the \emph{universal Liouville action} $S_1$.  They showed that $S_1$ is a Kähler potential on $T_0(1)$ which is of critical importance for the Kähler geometry. 
    	We take an analytic instead of a Teichmüller theoretic viewpoint, so we will consider $S_1$ as defined for Weil-Petersson quasicircles instead of their welding functions. 
    	Explicitly, for a Weil-Petersson quasicircle $\gamma$, 
    	\begin{align}\label{defs1}
    	S_1 (\gamma)  =\int_{\mathbb{D}}\left|\frac{f''(z)}{f'(z)}\right|^2|dz|^2+\int_{\C\setminus\bar{\mathbb{D}}}\left|\frac{g''(z)}{g'(z)}\right|^2|dz|^2+4\pi\log|f'(0)|-4\pi\log|g'(\infty)|.
    	\end{align}

    	\begin{theorem}[Y. Wang, \cite{yilinvention}]\label{loewner_weil}
    	A Jordan curve $\gamma$ has finite Loewner energy $I^L(\gamma)$ if and only if $\gamma$ is a Weil-Petersson quasicircle. Furthermore, we have 
    	\begin{align}\label{universal}
    	     I^L(\gamma)=\frac{1}{\pi}S_1(\gamma). 
    	\end{align}
    	\end{theorem}
    	    	We will therefore use interchangeably the terms \enquote{Jordan curve of finite Loewner energy,} \enquote{Weil-Petersson quasicircle,} or simply \enquote{Weil-Petersson curve.}
    	As we did not define explicitly the Loewner energy $I^L(\gamma)$, readers may consider \eqref{universal} as its definition. It may not be obvious from the expression of $S_1$ that it is invariant under M\"obius transformations, such as the inversion ${\ii}: z\mapsto 1/z$.
    	However, it would follow directly  from the 
    	definition using Loewner transform in \cite{yilinloop}.
     Provided that $\gamma$ separates $0$ from $\infty$, we may choose the biholomorphic functions $f$ and $g$ as in Theorem~\ref{thm:equi_WP} and assume further that $f(0)=0$. Applying the invariance of the Loewner energy under ${\ii}$, we get
    	\begin{align}\label{s2}
    	     &I^L(\gamma)= I^L({\ii}(\gamma))\\
    	     &=\frac{1}{\pi}\int_{\mathbb{D}}\left|\frac{f''(z)}{f'(z)}-2\frac{f'(z)}{f(z)}+\frac{2}{z}\right|^2|dz|^2+\frac{1}{\pi}\int_{\mathbb{C}\setminus\bar{\mathbb{D}}}\left|\frac{g''(z)}{g'(z)}-2\frac{g'(z)}{g(z)}+\frac{2}{z}\right|^2|dz|^2+4\log|f'(0)|-4\log|g'(\infty)|\nonumber.
    	\end{align}

    		\subsection{Moving Frames and the Ginzburg-Landau Equations}

        Moving frames, first introduced by Darboux in the late $19$th century to study curves and surfaces, were later generalised by Élie Cartan and permit one to reformulate astutely a wide class of differential-geometric problems. One of the rather recent such use of this theory is found in the work of Frédéric Hélein on harmonic maps (\cite{helein}), where the moving frame pave the way towards new regularity results. 
    	
        In \cite{lauromain}, Paul Laurain and Romain Petrides suggest a new approach to relate 
        the Loewner energy  
        to the renormalised energy of moving frames using the Ginzburg-Landau energy in a minimal regularity setting (which is of independent interest). Although the Ginzburg-Landau is normally used to construct harmonic maps with values into $S^1$ under topological constraints where no smooth solutions exist (\cite{BBH}), it should be seen—although we will not use this functional here—more generally as a way to construct (singular) moving frames on surfaces. Through this approach, one may hope to link quantatively the Loewner energy and the Willmore energy that can also be written in terms of moving frames (\cite{framemondinoriviere}). 
        
        Let $\Omega\subset \C$ be a simply connected domain, and $\gamma=\partial \Omega$. In \cite{lauromain}, they show that the Bethuel-Brezis-Hélein (\cite{BBH}) analysis carries on for general chorc-arc curves and $H^{1/2}$ boundary data. Using this delicate analysis, they obtain the following result, which is the most relevant one in this article.

          \begin{theorem}[Laurain-Petrides, \cite{lauromain}, Theorem $0.2$, Theorem $0.3$]\label{gl1}
        Let $\Omega\subset \C$ be a bounded simply connected domain such that $\gamma=\partial \Omega$ is a Weil-Petersson quasicircle. Then, there exists a harmonic map $\e:\Omega\setminus\ens{p}\rightarrow S^1$ with boundary data $\tau:\Gamma\rightarrow S^1$ which is the unit tangent vector of $\partial\Omega=\Gamma$. {Let}  $\f=-i\,\e$ and $\omega=\s{\e}{d\f\,}$, then there exists a harmonic function $\mu:\Omega\rightarrow\R$ such that $\omega=\ast\,d\left(G_{\Omega}+\mu\right)$, and a conformal map $f:\mathbb{D}\rightarrow \Omega$ such {that $f(0) = p$, and }
        \begin{align}\label{uniform}
        \left\{\begin{alignedat}{1}
        \frac{1}{r}\p{\theta}f&=e^{\mu\circ f}\e\circ f\\
        \p{r}f&=e^{\mu\circ f}\f\circ f.
        \end{alignedat}\right.
        \end{align}
        Furthermore, we have
        \begin{align}\label{uniform2}
        \int_{\Omega}|\omega-\ast\,dG_{\Omega}|^2dx=\int_{\Omega}|\D\mu|^2dx=\int_{\mathbb{D}}\left|\frac{f''(z)}{f'(z)}\right|^2|dz|^2,
        \end{align}
        where $G_{\Omega}$ is a Green's function with Dirichlet boundary condition on $\partial\Omega$.
        \end{theorem}
        The other main result of \cite{lauromain} is to identify the renormalised energy in the sense of Bethuel-Brezis-Hélein as an explicit term involving \eqref{uniform2}. 
      \begin{rem}
   {The harmonic function} $\mu$ is explicitly given by $\mu=\log|\p{r}f|\circ f^{-1}=\log\left|\frac{1}{r}\p{\theta}f\right|\circ f^{-1}=\log|f'|\circ f^{-1}$. The last identities follow
    from the conformality of $f$. {We note that in \cite{lauromain}, the point $p$ is a special point such that any biholomorphic map $f$ with $f(0) = p$ maximizes $|f'(0)|$ amongst all biholomorphic maps $\mathbb D \to \Omega$.}
        \end{rem}
    We see that the frame energy \eqref{uniform2} coincides with the first term in \eqref{defs1}.    To obtain the second half of the Loewner energy involving
        \begin{align}\label{g_non_compact}
        \int_{\C\setminus\bar{\mathbb{D}}}\left|\frac{g''(z)}{g'(z)}\right|^2|dz|^2,
        \end{align}
         we cannot easily use the Ginzburg-Landau equation to construct the moving frames since that would force us to work on the non-compact domain $\C\setminus\bar{\Omega}$. 
         Using the inversion ${\ii}$ will not suffice either. 
         If we choose the biholomorphic map $\tilde{g}:\mathbb{D}\rightarrow {\ii}(\C\setminus\bar{\Omega})$ so that $\tilde{g}={\ii}\circ g\circ {\ii}$, we have 
         \begin{align*}
         	\int_{\mathbb{D}}\left|\frac{\tilde{g}''(z)}{\tilde{g}'(z)}\right|^2|dz|^2=\int_{\mathbb{C}\setminus\bar{\mathbb{D}}}\left|\frac{g''(z)}{g'(z)}-2\frac{g'(z)}{g(z)}+\frac{2}{z}\right|^2|dz|^2
         \end{align*}
         which is in general different from \eqref{g_non_compact}. 
        To overcome this {technicality}, 
        we work directly on $S^2$ to obtain a formula of the Loewner energy in terms of moving frames.

         \subsection{Main Results}
         
         \setcounter{theorem}{0}
         
         \renewcommand*{\thetheorem}{\Alph{theorem}}

         \begin{theorem}\label{wp1}
         	 	Let $\Gamma\subset S^2 \subset \mathbb R^3$ be a Weil-Petersson quasicircle,  $\Omega_1,\Omega_2\subset S^2$ be the two disjoint open connected components of $S^2\setminus\Gamma$.
         	 	 Fix some $j=1,2$. Then, for {any} $p_j\in \Omega_j$, there exists harmonic moving frames $(\e_j,\f_j):\Omega_j\setminus\ens{p_j}\rightarrow U\Omega_j\times U\Omega_j$ such that the Cartan form $\omega_j=\s{\e_j}{d\f_j}$ admits the decomposition
         	 	\begin{align}\label{thA1}
         	 		\omega_j=\ast\,d\left(G_{\Omega_j}+\mu_j\right),
         	 	\end{align}
         	 	where $G_{\Omega_j}:\Omega_j\setminus\ens{p_j}\rightarrow \R$ is the Green's function of the Laplacian $\Delta_{g_0}$ on $\Omega_j$ with Dirichlet boundary condition, and $\mu_j\in C^{\infty}(\Omega_j)$ satisfies
         	 	\begin{align}\label{thA2}
         	 		\left\{\begin{alignedat}{2}
         	 			-\Delta_{g_0}\mu_j&=1\qquad&&\text{in}\;\,\Omega_j\\
         	 			\partial_{\nu}\mu_j&=k_{g_0}-\partial_{\nu}G_{\Omega_j}\qquad&&\text{on}\;\,\partial\Omega_j,
         	 		\end{alignedat} \right.
         	 	\end{align}
         	 	where $k_{g_0}$ is the geodesic curvature on $\Gamma=\partial\Omega_j$. Define the functional {$\mathscr{E}$} \emph{(}that we call the renormalised energy associated to the frames $(\e_1,\f_1)$ and $(\e_2,\f_2)$\emph{)}  by
         	 	\begin{align}\label{thA3}
         	 		\mathscr{E}(\Gamma)=\int_{\Omega_1}|d\mu_1|^2_{g_0}d\mathrm{vol}_{g_0}+\int_{\Omega_2}|d\mu_2|^2_{g_0}d\mathrm{vol}_{g_0}+2\int_{\Omega_1}G_{\Omega_1}K_{g_0}d\mathrm{vol}_{g_0}+2\int_{\Omega_2}G_{\Omega_2}K_{g_0}d\mathrm{vol}_{g_0}+4\pi.
         	 	\end{align}
         	 	Then there exists conformal maps $f_1:\mathbb{D}\rightarrow \Omega_1$ and $f_2:\mathbb{D}\rightarrow \Omega_2$ such that $f_1(0)=p_1$, $f_2(0)=p_2$ and 
         	 	\begin{align}\label{thA4}
         	 		I^L(\Gamma)=\frac{1}{\pi}\mathscr{E}(\Gamma)+4\log|\D f_1(0)|+4\log|\D f_2(0)|-12\log(2)=\frac{1}{\pi}\mathscr{E}_0(\Gamma).
         	 	\end{align}
         	 \end{theorem}
         	 
         	 	 \begin{figure}[H]
         	 	\centering
         	 	\includegraphics[width=0.6\textwidth]{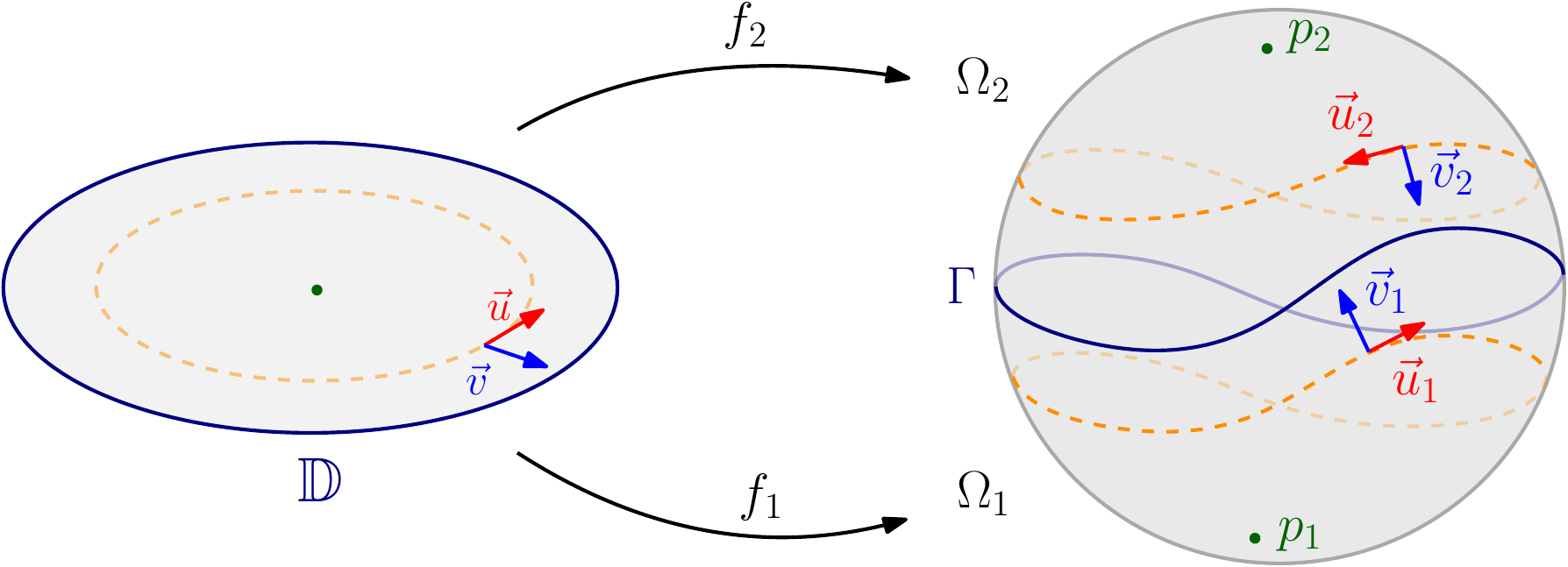}
         	
         	 	\caption{
         	 	Harmonic moving frames on the sphere associated to a Weil-Petersson quasicircle.
         	 	\label{fig:spherical} } 
         	 \end{figure}
      	 \setcounter{theorem}{9}
         	 \renewcommand{\thetheorem}{\thesection.\arabic{theorem}}
         	 \begin{rem}
         	 \begin{enumerate}
         	 \item[(1)]	 In the theorem above, we wrote $U\Omega_j$ ($j=1,2$) for the unit tangent bundle.  
         	 {T}he function $\mu_j${,} explicitly given by 
         	       \begin{align}\label{thA5}
         	        	 \mu_j=\frac{1}{2}\log\left(\frac{|\D f_j|^2}{2}\right)=\log|\D f_j|-\frac{1}{2}\log(2){,}
         	       \end{align}
         	 correspond to the conformal parameter of the conformal maps $f_1,f_2:\mathbb{D}\rightarrow S^2\subset \R^3$. 
         	 \item[(2)] The {constant term} 
         	 $4\pi$ in the definition of $\mathscr{E}$ is arranged so that $\mathscr{E}(S^1)=0$ (see Remark \ref{hemisphere}). Furthermore, the name renormalised energy is justified by the following identity
         	 \begin{align*}
         	 \mathscr{E}(\Gamma)=\int_{\Omega_1}\left(|d\e_1|^2_{g_0}+|d\f_1|_{g_0}^2-2|dG_{\Omega_1}|^2_{g_0}\right)d\mathrm{vol}_{g_0}+\int_{\Omega_2}\left(|d\e_2|^2_{g_0}+|d\f_2|_{g_0}^2-2|dG_{\Omega_2}|^2_{g_0}\right)d\mathrm{vol}_{g_0},
         	 \end{align*}
         	 where no constant term is involved. 
         	 \item[(3)] 
         	 The solution {\eqref{thA2}} to the Dirichlet problem is unique, and so {are} the moving frame{s} once the singularities $(p_1,p_2)\in \Omega_1\times \Omega_2$ are fixed. See Theorem \ref{key2} and \ref{5.2}. Notice that the geodesic curvature is understood in the distributional sense here (see Section \ref{geodesicprop} from the appendix for more details). 
         	 \end{enumerate}
         	 
         	 \end{rem} 
         \setcounter{theorem}{1}
                  \renewcommand*{\thetheorem}{\Alph{theorem}}
         This theorem corresponds to Theorem \ref{wp10} in the article {for smooth curves}
          and to Theorem~\ref{s3} {for general Weil-Petersson quasicircles}. 
         The general case follows essentially from the following result which can also be viewed as a restatement of Theorem~\ref{wp1} without any mention of moving frames. 
          \begin{theorem}[{See Theorem~\ref{s3}}]\label{s03}
                   	 	Let $\Gamma\subset S^2$ be 
                   	 	{a Weil-Petersson quasicircle and} $\Omega_1,\Omega_2\subset S^2\setminus\Gamma$ be the two connected components of $S^2\setminus\Gamma$. {F}or all conformal maps $f_1:\mathbb{D}\rightarrow \Omega_1$ and $f_2:\mathbb{D}\rightarrow \Omega_2$, we have
                   	 \begin{align}\label{conformula}
                   	 		I^L(\Gamma)&=\frac{1}{\pi}\sum_{j=1}^{2}\left(\int_{\mathbb{D}}|\nabla\log|\nabla f_j||^2|dz|^2+\int_{\mathbb{D}}\log|z||\nabla f_j|^2|dz|^2+\mathrm{Area}(\Omega_j)
                   	 		+4\pi\log|\nabla f_j(0)|\right)-12\log(2).
                   	 	\end{align}
                   	 \end{theorem}
       %  This theorem corresponds to Theorem \ref{s3} below. 

         \bigskip
         
         \textbf{Acknowledgements.} This paper is part of a common project between Paul Laurain and Romain Petrides and the two authors on the various characterisations of {Weil-Petersson quasicircles}. We thank Paul Laurain and Romain Petrides for useful discussions and kindly sharing their manuscript with us. We also thank Christopher Bishop for allowing us into his topic class which helped us understand his recent work [4]. A. M. is supported by the Early Postdoc.Mobility \emph{Variational Methods in Geometric Analysis} P$2$EZP$2$\_$191893$. Y. W. is partially supported by NSF grant DMS-$1953945$.
         
         \renewcommand{\thetheorem}{\thesection.\arabic{theorem}}

	   \section{Moving Frame Energy via Zeta-Regularised Determinants for Smooth Curves}\label{zeta}

    The following expression of the Loewner energy will prove crucial in this section.

    \begin{theorem}[Y. Wang \cite{yilinvention}]\label{yilindet}
   	 	Let $\alpha \in C^\infty(S^2, \R)$, $g=e^{2\alpha}g_0$ be a metric conformally equivalent to the spherical metric $g_0$ of $S^2$, and $\Gamma\subset S^2$ be a simple smooth curve. Let $\Omega_1,\Omega_2\subset S^2$ be the two disjoint open connected components of $S^2\setminus\Gamma$. Then we have
   	 	\begin{align}\label{spherical_loewner}
   	 		I^L(\Gamma)=12\,\log\frac{\det_{\zeta}(-\Delta_{S^2_-,g})\det_{\zeta}(-\Delta_{S^2_+,g})}{\det_{\zeta}(-\Delta_{\Omega_1,g})\det_{\zeta}(-\Delta_{\Omega_2,g})},
   	 	\end{align}
   	 	where $S_-^2$ \emph{(}resp. $S^2_+$\emph{)} is the southern hemisphere \emph{(}resp. the northern hemisphere\emph{)}.
   	    \end{theorem}
	 
	 We now use {the} formula {\eqref{spherical_loewner} expressing} the Loewner energy {in terms of} zeta-regularised determinants to link {the Loewner energy to} the renormalised energy of moving {frames}
	 on 
	 $S^2$.
	 First, let $g_0=g_{S^2}$ be the standard round metric on $S^2$. Let $\Gamma\subset S^2$ be a simple \emph{smooth}\footnote{It is necessary to assume that the curve {is smooth} 
	 for one will need to recurse to the Froebenius theorem below. Furthermore, the formula for the Loewner energy using the zeta-regularised determinants (\cite{yilinvention}) only works for smooth (or at least $C^3$) curves (\cite{takteo}, Corollary $3.12$).} curve, and let $\Omega_1,\Omega_2\subset S^2$ the two disjoint open connected components of $S^2\setminus\Gamma$. Since we are working on a curved manifold, we cannot directly use the result of \cite{lauromain} to construct moving frames with the Ginzburg-Landau method. However, 
	 {we will construct them directly in}
	 Section \ref{section7} (see Theorem \ref{key2}).
	 Therefore, let {us} assume that $(\e_1,\f_1):\Omega_1\setminus\ens{p_1}\rightarrow US^2\times US^2$ are harmonic vector fields such that $\e_1=\tau$ on $\partial\Omega_1=\Gamma$ (where $\tau$ is the unit tangent on $\Gamma$), 	 and the $1$-form $\omega=\s{\e_{1}}{d\f_{1}}$ satisfies
	 \begin{align}\label{cartan0}
	 	\omega=\ast\,d\,\left(G_{\Omega_1}+\mu_1\right)\qquad\text{in}\;\, \mathscr{D}'(\Omega_1)
	 \end{align}
	 where $G_{\Omega_{1}}:\Omega_1\setminus\ens{p_1}\rightarrow \R$ is the Green's function for the Laplacian on $\Omega_1\setminus\ens{p_1}$ with Dirichlet boundary condition. Namely, $G_{\Omega_1}$ satisfies 
	 \begin{align}\label{new_harmonic}
	 	\left\{\begin{alignedat}{2}
	 		\Delta_{g_0}G_{\Omega_1}&=2\pi\delta_{p_1}\qquad&&\text{in}\;\,\mathscr{D}'(\Omega_1)\\
	 		G_{\Omega_1}&=0\qquad&&\text{on}\;\,\partial\Omega_1,
	 	\end{alignedat} \right.
	 \end{align}
	 and $\mu_1:\Omega_1\rightarrow \R$ is a smooth function satisfying
	 \begin{align}\label{geodesic}
	 	\left\{\begin{alignedat}{2}
	 		-\Delta_{g_0}\mu_1&=1\qquad&&\text{in}\;\,\Omega_1\\
	 		\partial_{\nu}\mu_1&=k_{g_0}-\partial_{\nu}G_{\Omega_1}\qquad&&\text{on}\;\,\partial \Omega_1.
	 	\end{alignedat} \right.
	 \end{align}
	 where $k_{g_0}$ is the geodesic curvature with respect to the round metric $g_0$, and the normal derivative is taken with respect to the $g_0$. 

   To fix notations, we recall the following result. 
   \begin{theorem}[\cite{Jonsson-Wallin}, see also, %Jones 
   \cite{jones,wallin}]
         Let  $\Omega\subset \C$ \emph{(}resp. $\Omega\subset S^2$\emph{)} be a {bounded simply} connected domain {such that $\partial\Omega$ is chord-arc, and} let $g_0$ be the flat metric on $\Omega$ \emph{(}resp. $g_0$ be the round metric on $S^2$\emph{)}. Then for all $p\in \Omega$, there exists a unique Green's function $G_{\Omega,p}\in C^{\infty}(\Omega\setminus\ens{p},\R)$  with Dirichlet boundary condition.
        Furthermore, for 
        every $h\in H^{1/2}(\partial \Omega,\R)$, there exists a unique function $u\in W^{1,2}(\Omega,\R)$ such that
        \begin{align*}
        \left\{\begin{alignedat}{2}
        \Delta_{g_0}u&=0\qquad&&\text{in}\;\,\Omega\\
        u&=h\qquad&&\text{on}\;\,\partial\Omega.
        \end{alignedat}\right.
        \end{align*}
        \end{theorem}
        Whenever it is clear from context, we will write $G_{\Omega_1}$ for $G_{\Omega_1,p_1}$.
         \begin{rem}
       {The existence of a
        Green's function follows from its conformal invariance and the uniformisation theorem. Indeed, if $\Omega$ is a Jordan domain, and $f:\mathbb{D}\rightarrow\Omega$ is a biholomorphic map such that $f(0)=p$, and $G_{\mathbb{D},0}=\log|z|$, then $G_{\Omega,p}=G_{\mathbb{D},0}\circ f^{-1}$. We assume that $\partial \Omega$ is chord-arc so that the trace theorems apply as in \cite{Jonsson-Wallin,wallin}.} 
        The passage from $\C$ to $S^2$ is easy using a stereographic projection and the conformal invariance of Green's functions. 
        \end{rem}
 
	 Now, 
	 {following} Proposition $5.1$ of \cite{lauromain}, {it is not hard to} see that the{ir} proof using the Froebenius {theorem} also works for domains of the sphere,
	 and we get a conformal diffeomorphism $	 {\varphi}:(-\infty,0)\times \partial B(0,\rho)\rightarrow \Omega_1\setminus\ens{p_1}$ {for some $\rho > 0$} such that

	 \begin{align*}
	 \begin{alignedat}{1}
	 	\partial_{s}{\varphi}(s,\theta)&=e^{G_{\Omega_1}\circ {\varphi}+\mu_1\circ{\varphi}}\f_{1}\circ{\varphi}\\
	 	\partial_{\theta}{\varphi}(s,\theta)&=e^{G_{\Omega_1}\circ{\varphi}+\mu_1\circ{\varphi}}\e_{1}\circ{\varphi}.
	 \end{alignedat}
	 \end{align*}
	 Notice that the Proposition $5.1$ of \cite{lauromain} gives a privileged $p_1\in\Omega_1$, but we will show in Theorem \ref{key2} that $p_1$ can be taken arbitrarily (see also Theorem \ref{s3}). However, the proof works for an arbitrary harmonic moving frame whose Cartan form admits an expansion as in \eqref{cartan0} where $\mu_1$ solves \eqref{geodesic}. 
	 Since $\mu_1$ is defined up to {an additive} constant, we can assume that $\rho=1$ in the following. 
	 {We define the conformal map} $f_1:\mathbb{D}\rightarrow \Omega_1$ using the polar coordinates by
	 \begin{align*}
	 	f_1(r,\theta)={\varphi}(\log(r),\theta),
	 \end{align*}
	 we can continuously extend $f_1$ at $z=0$ such that $f_1(0)=p_1$.
	 
	 {Now we relate $\mu_1$ to $f_1$. S}ince $f_{1}$ is conformal, the function $G=G_{\Omega_1}\circ f_{1}:\mathbb{D}\setminus\ens{0}\rightarrow\R$ is harmonic on $\mathbb{D}\setminus\ens{0}$, satisfies $G=0$ on $\partial\mathbb{D}$, so by \eqref{new_harmonic}, we deduce that
	 \begin{align*}
	 	G=G_{\mathbb{D},{0}}=\log|z|.
	 \end{align*}
	 Therefore, we have
	 \begin{align}\label{moving_frames}
	 	\left\{\begin{alignedat}{1}
	 		\partial_r f_1&=\frac{1}{r}\partial_{s}{\varphi}(\log(r),\theta)=\frac{1}{r}e^{\log(r)+\mu_1\circ f_1}{\f}_1\circ f_1=e^{\mu_1\circ f_1}\f_1\circ f_1\\
	 		\frac{1}{r}\partial_{\theta}f_1&=\frac{1}{r}\partial_{\theta}{\varphi}(\log(r),\theta)=e^{\mu_1\circ f_1}\e_1\circ f_1.
	 	\end{alignedat}\right.
	 \end{align}
	 Since 
	 {$|\e_1|=|\f_1|=1$}, and $\s{\e_{1}}{\f_{1}}=0$, we deduce that  
	 \begin{align*}
	 	&|\partial_rf_1|^2=\frac{1}{r^2}\left|\partial_{\theta}f_1\right|^2=e^{2\mu_1\circ f_1}\\
	 	&\s{\partial_rf_1}{\partial_{\theta}f_1}=0,
	 \end{align*}
	 which shows that the conformal parameter of $f$ is
	 \begin{align*}
	 	\frac{1}{2}|\D f_1|^2 
	 	= {\frac{1}{2} \left(|\partial_rf_1|^2+ \frac{1}{r^2}\left|\partial_{\theta}f_1\right|^2\right)}=e^{2\mu_1\circ f_1},
	 \end{align*}
	 which implies that
	 \begin{align}\label{eq:change_factor_mu}
	 \mu_1=\log|\D f_1|\circ f_1^{-1}-\frac{1}{2}\log(2).
	 \end{align}
	 In particular, we have
	 \begin{align}\label{rho}
	     \mu_1(p_1)=\log|\D f_1(0)|-\frac{1}{2}\log(2),
	 \end{align}
	 where $p_1\in \Omega_1$ is the singularity of the moving frame $(\e_1,\f_1):\Omega_1\setminus\ens{p_1}\rightarrow US^2\times U S^2$.

	 {We can relate the change of metric by $f_1$ to $\mu_1$ as follows.}
    If $\iota:\Omega_1\subset S^2\hookrightarrow \R^3$ is the inclusion map, we have ${{g_{0}}_{|\Omega_1}}	 =\iota^{\ast}g_{\R^3}$. 

	 As $f_{1}$ is conformal, we have
	 \begin{align}\label{conformal_factor1}
	 	f_1^{\ast}{g_{0}}_{|\Omega_1} &=f_1^{\ast}\iota^{\ast}g_{\R^3}=(\iota\circ f_1)^{\ast}g_{\R^3}= {\frac{1}{2}}|\D f_1(z)|^2|dz|^2=e^{2\mu_1\circ f_1(z)}|dz|^2=e^{2\mu_1\circ f_1(z)-2\psi(z
	 		)}\frac{4|dz|^2}{(1+|z|^2)^2}\nonumber\\
	 	&=e^{2\mu_1\circ f_1-2\psi}(({\pi^{-1}})^{\ast}g_{0})|_{\mathbb{D}},
	 \end{align}
	 where
	 \begin{align*}
	 	\psi(z)=\log\left(\frac{2}{1+|z|^2}\right),
	 \end{align*}
	 and $\pi^{-1}:\C\rightarrow S^2\setminus\ens{N}$ is the inverse stereographic projection. 
  Writing for simplicity
	 \begin{align*}
	 	g_{\mathbb{D}}=\frac{4|dz|^2}{(1+|z|^2)^2}=e^{2\psi(z)}|dz|^2  = (\pi^{-1})^\ast g_0|_{\mathbb D},
	 \end{align*}
	 we deduce by \eqref{conformal_factor1} that
	 \begin{align*}
	 	{g_{0}}_{|\Omega_1}=(f_1\circ f_1^{-1})^{\ast}(g_{0|\Omega_1})=(f_1^{-1})^{\ast}f_1^{\ast}{g_{0}}_{|\Omega_1}=
	 	e^{2\mu_1-2\psi\circ f_1^{-1}}(f_1^{-1})^{\ast}(g_{\mathbb{D}}),
	 \end{align*}
	 so that (by an abuse of notation for the last identity) 
	 \begin{align}\label{conformal_factor2}
   (f_1^{-1})^{\ast}(g_{\mathbb{D}})=e^{-2\mu_{1}+2\psi\circ f_1^{-1}}{g_{0}}_{|\Omega_1}=e^{2\alpha_1}{g_{0}}_{|\Omega_{1}}
	 \end{align}
	 where 
	 \begin{align*}
	 	\alpha_1(z)=-\mu_{1}(z)+\psi {\circ} f_1^{-1}(z).
	 \end{align*}
	 
	 \begin{rem}\label{rem:existence_frame}
	 To summarize, the above discussion shows that the moving frame $(\e_1, \f_1)$ satisfying the boundary condition $\e_1 = \tau$ on $\Gamma$, \eqref{cartan0}, and \eqref{geodesic} is tightly related to a conformal map $f_1 : \mathbb D \to \Omega_1$ using Froebenius theorem as in \cite{lauromain}, in the way that the moving frame satisfies \eqref{moving_frames}. However, we can start directly with any conformal map $f_1$ and \eqref{moving_frames} gives a moving frame $(\e_1, \f_1)$ which satisfies \eqref{cartan0} and \eqref{geodesic}. This is the approach we take in Section~\ref{section7} which allows us to relax the regularity assumption of $\partial \Omega_1=\Gamma$.
	
	 \end{rem} 
	 \begin{defi}
	 	Define the open subsets $S_+^2,S^2_-\subset S^2$ by 
	 	\begin{align*}
	 		S_+^2&=S^2\cap\ens{(x,y,z) {\in \mathbb R^3}:z>0}\\
	 		S_-^2&=S^2\cap\ens{(x,y,z) {\in \mathbb R^3}:z<0}.
	 	\end{align*}
	 \end{defi}
	 
	 \begin{theorem}\label{wp10}
	 	Let $\Gamma\subset S^2$ be a smooth Jordan curve, and let $\Omega_1,\Omega_2\subset S^2$ the two disjoint open connected components of $S^2\setminus\Gamma$. Fix some $j=1,2$. Then, for all $p_j\in \Omega_j$ and for all harmonic moving frames $(\e_j,\f_j):\Omega_j\setminus\ens{p_j}\rightarrow U\Omega_j\times U\Omega_j$ such that the Cartan form $\omega_j=\s{\e_j}{d\f_j}$ admits the decomposition
	 	\begin{align*}
	 		\omega_j=\ast\,d\left(G_{\Omega_j}+\mu_j\right),
	 	\end{align*}
	 	where $G_{\Omega_j}:\Omega_j\setminus\ens{p_j}\rightarrow \R$ is the Green's function of the Laplacian $\Delta_{g_0}$ on $\Omega_j$ with Dirichlet boundary condition, and $\mu_j\in C^{\infty}(\Omega_j)$ satisfies
	 	\begin{align}\label{boundary}
	 		\left\{\begin{alignedat}{2}
	 			-\Delta_{g_0}\mu_j&=1\qquad&&\text{in}\;\,\Omega_j\\
	 			\partial_{\nu}\mu_j&=k_{g_0}-\partial_{\nu}G_{\Omega_j}\qquad&&\text{on}\;\,\partial\Omega_j,
	 		\end{alignedat} \right.
	 	\end{align}
	 	where $k_{g_0}$ is the geodesic curvature on $\Gamma=\partial\Omega_j$. Define the functional $\mathscr{E}$ \emph{(}that we call the renormalised energy associated to the frames $(\e_1,\f_1)$ and $(\e_2,\f_2)$\emph{)} by
	 	\begin{align*}
	 		\mathscr{E}(\Gamma)=\int_{\Omega_1}|d\mu_1|^2_{g_0}d\mathrm{vol}_{g_0}+\int_{\Omega_2}|d\mu_2|^2_{g_0}d\mathrm{vol}_{g_0}+2\int_{\Omega_1}G_{\Omega_1}K_{g_0}d\mathrm{vol}_{g_0}+2\int_{\Omega_2}G_{\Omega_2}K_{g_0}d\mathrm{vol}_{g_0}+4\pi.
	 	\end{align*}
	 	Then there exists conformal maps $f_1:\mathbb{D}\rightarrow \Omega_1$ and $f_2:\mathbb{D}\rightarrow \Omega_2$ such that $f_1(0)=p_1$, $f_2(0)=p_2$ and 
	 	\begin{align*}
	 		I^L(\Gamma)=\frac{1}{\pi}\mathscr{E}(\Gamma)+4\log|\D f_1(0)|+4\log|\D f_2(0)|-12\log(2)=\frac{1}{\pi}\mathscr{E}_0(\Gamma).
	 	\end{align*}
	 \end{theorem}
  \begin{rem}
      If $\mathscr{W}_j$ ($j=1,2$) is the renormalised energy in the sense of Bethuel-Brezis-Hélein associated to the moving frame $(\e_j,\f_j)$—or more precisely, to its boundary data—(see \cite{BBH} and \cite{lauromain}), we have
	 	\begin{align*}
	 		\mathscr{W}_1+\mathscr{W}_2=\mathscr{E}(\Gamma)+2\pi\log|\D f_1(0)|+2\pi\log|\D f_2(0)|,
	 	\end{align*}
	 	since $\omega_1-\ast\,dG_{\Omega_1}=\ast\, d\mu_1$. 
  \end{rem}

	 \begin{proof}[Proof of Theorem~\ref{wp10}]
	    If $\pi^{-1}:\C\rightarrow S^2\setminus\ens{N}$ is the inverse stereographic projection,
	 	\begin{align*}
	 		g_{\mathbb{D}}=\frac{4|dz|^2}{(1+|z|^2)^2}=e^{2\psi(z)}|dz|^2,
	 	\end{align*}
	 	and $S^2_-$ is the southern hemisphere, we deduce that 
	 	\begin{align*}
	 		\dzeta(-\Delta_{S^2_-,g_0})&=\dzeta(-\Delta_{\mathbb{D},\pi^{\ast}g_0})=\dzeta(-\Delta_{\mathbb{D},g_{\mathbb{D}}})
	 		=\dzeta(-\Delta_{\Omega_1,\varphi^{\ast}g_{\mathbb{D}}})=\dzeta(-\Delta_{\Omega_1,g_{1}}),
	 	\end{align*}
	 	and by the Alvarez-Polyakov formula {(see (1.17) of \cite{sarnak})} and \eqref{conformal_factor2}, we have
	 	\begin{align*}
	 		&\dzeta(-\Delta_{S^2_-,g_0})-\dzeta(-\Delta_{\Omega_1,g_0})=\dzeta(-\Delta_{\Omega_1,g_1})-\dzeta(-\Delta_{\Omega_1,g_0})\\
	 		&=-\frac{1}{12\pi}\bigg\{\int_{\Omega_1}|d\alpha_1|_{g_0}^2d\mathrm{vol}_{g_0}+2\int_{\Omega_1}K_{g_0}\alpha_1\,  d\mathrm{vol}_{g_0}+2\int_{\Gamma}k_{g_0}\alpha_1\,  d\mathscr{H}^1_{g_0}+3\int_{\Gamma}\partial_{\nu}\alpha_1\, d\mathscr{H}^1_{g_0}\bigg\}\\
	 		&=-\frac{1}{12\pi}\bigg\{\int_{\Omega_1}|d(-\mu_1+\theta_1)|_{g_0}^2d\mathrm{vol}_{g_0}+2\int_{\Omega_1}K_{g_0}(-\mu_1+\theta_1)d\mathrm{vol}_{g_0}
	 		+2\int_{\Gamma}k_{g_0}(-\mu_1+\theta_1)d\mathscr{H}^1_{g_0}\\
	 		&
	 		+3\int_{\Gamma}\partial_{\nu}(-\mu_1+\theta_1)d\mathscr{H}^1_{g_0}\bigg\},
	 	\end{align*} 
	 	if we choose $\Gamma$ to be 
	 	{oriented as} $\partial\Omega_1$, and
	 	\begin{align*}
	 		\theta_1=\psi\circ f^{-1}_1.
	 	\end{align*}
	 	Therefore, using subscripts with evident notations, we deduce by Theorem \ref{yilindet} with $g=g_0$ that
	 	\begin{align}\label{eq1}
	 		&-\pi\,I^L(\Gamma)=-12\pi\,\log\frac{\det_{\zeta}(-\Delta_{S^2_-,g_0})\det_{\zeta}(-\Delta_{S^2_+,g_0})}{\det_{\zeta}(-\Delta_{\Omega_1,g_0})\det_{\zeta}(-\Delta_{\Omega_2,g_0})}=-12\pi\,\log\frac{\det_{\zeta}(-\Delta_{\Omega_1,g_1})\det_{\zeta}(-\Delta_{\Omega_2,g_2})}{\det_{\zeta}(-\Delta_{\Omega_1,g_0})\det_{\zeta}(-\Delta_{\Omega_2,g_0})}\nonumber\\
	 		&=\int_{\Omega_1}|d(-\mu_1+\theta_1)|_{g_0}^2d\mathrm{vol}_{g_0}+2\int_{\Omega_1}K_{g_0}(-\mu_1+\theta_1)d\mathrm{vol}_{g_0}
	 		+2\int_{\partial\Omega_1}k_{g_0}(-\mu_1+\theta_1)d\mathscr{H}^1_{g_0}\nonumber\\
	 		&
	 		+3\int_{\partial\Omega_1}\partial_{\nu}(-\mu_1+\theta_1)d\mathscr{H}^1_{g_0}\nonumber\\
	 		&+\int_{\Omega_2}|d(-\mu_2+\theta_2)|_{g_0}^2d\mathrm{vol}_{g_0}+2\int_{\Omega_2}K_{g_0}(-\mu_2+\theta_2)d\mathrm{vol}_{g_0}
	 		+2\int_{\partial\Omega_2}k_{g_0}(-\mu_2+\theta_2)d\mathscr{H}^1_{g_0}\nonumber\\
	 		&
	 		+3\int_{\partial\Omega_2}\partial_{\nu}(-\mu_2+\theta_2)d\mathscr{H}^1_{g_0}.
	 	\end{align}
	 	Notice that provided that $\Gamma$ be given with the same orientation of $\partial\Omega_1$, we have
	 	\begin{align}\label{eq2}
	 		&2\int_{\partial\Omega_1}k_{g_0}(-\mu_1+\theta_1)d\mathscr{H}^1_{g_0}+2\int_{\partial\Omega_2}k_{g_0}(-\mu_2+\theta_2)d\mathscr{H}^1_{g_0}=2\int_{\Gamma}k_{g_0}(-\mu_1+\theta_1+\mu_2-\theta_2)\mathscr{H}^1_{g_0}.
	 	\end{align}
	 	Since $K_{g_0}=1=-\Delta_{g_0}\mu_1$ on $\Omega_1$, we have
	 	\begin{align}\label{eq3}
	 		\int_{\Omega_1}K_{g_0}(-\mu_1)d\mathrm{vol}_{g_0}=\int_{\Omega_1}\mu_1\,\Delta_{g_0}\mu_1d\mathrm{vol}_{g_0}&=-\int_{\Omega_1}|d\mu_1|_{g_0}^2d\mathrm{vol}_{g_0}+\int_{\Gamma}\mu_1\partial_{\nu}\mu_1d\mathscr{H}^1_{g_0}\nonumber\\
	 		&=-\int_{\Omega_1}|d\mu_1|_{g_0}^2d\mathrm{vol}_{g_0}+\int_{\Gamma}\left(k_{g_0}-\partial_{\nu}G_{\Omega_1}\right)\mu_1d\mathscr{H}^1_{g_0}
	 	\end{align}
	 	by \eqref{geodesic}. Therefore, we deduce that 
	 	\begin{align}\label{eq4}
	 		&2\int_{\Omega_1}K_{g_0}(-\mu_1+\theta_1)d\mathrm{vol}_{g_0}
	 		+2\int_{\partial\Omega_1}k_{g_0}(-\mu_1+\theta_1)d\mathscr{H}^1_{g_0}\nonumber\\
	 		&=-2\int_{\Omega_1}|d\mu_1|_{g_0}^2d\mathrm{vol}_{g_0}-2\int_{\Gamma}\partial_{\nu}G_{\Omega_1}\mu_1d\mathscr{H}^1_{g_0}+2\int_{\Omega_1}\theta_1 d\mathrm{vol}_{g_0}+2\int_{\Gamma}k_{g_0}\theta_1d\mathscr{H}^1_{g_0}.
	 	\end{align}
	 	Now, we have
	 	\begin{align}\label{eq5}
	 		&\int_{\Omega_1}|d(-\mu_1+\theta_1)|_{g_0}^2d\mathrm{vol}_{g_0}=\int_{\Omega_1}|d(-\mu_1+\psi\circ f_1^{-1})|^2_{g_0}d\mathrm{vol}_{g_0}\nonumber\\
	 		&=\int_{\Omega_1}|d\mu_1|_{g_0}^2d\mathrm{vol}_{g_0}-2\int_{\Omega_1}\s{d\mu_1}{d(\psi\circ f_1^{-1})}_{g_0}d\mathrm{vol}_{g_0}+\int_{\Omega_1}|d(\psi\circ f_1^{-1})|_{g_0}^2d\mathrm{vol}_{g_0}.
	 	\end{align}
	 	Since $-\Delta_{g_0}\mu_1=1$, we deduce that 
	 	\begin{align}\label{eq6}
	 		-2\int_{\Omega_1}\s{d\mu_1}{d(\psi\circ f^{-1})}_{g_0}d\mathrm{vol}_{g_0}&=-2\int_{\Omega_1}\s{d\mu_1}{d\theta_1}_{g_0}d\mathrm{vol}_{g_0}=2\int_{\Omega_1}\theta_1\Delta_{g_0}\mu_1d\mathrm{vol}_{g_0}-2\int_{\Gamma}\theta_1\partial_{\nu}\mu_1d\mathscr{H}^1_{g_0}\nonumber\\
	 		&=-2\int_{\Omega_1}\theta_1d\mathrm{vol}_{g_0}-2\int_{\Gamma}k_{g_0}\theta_1d\mathscr{H}^1_{g_0}+2\int_{\Gamma}\partial_{\nu}G_{\Omega_1}\theta_1d\mathscr{H}^1_{g_0}.
	 	\end{align}
	 	Now, by conformal invariance of the Dirichlet energy, we have
	 	\begin{align*}
	 		\int_{\Omega_1}|d(\psi\circ f^{-1})|^2_{g_0}d\mathrm{vol}_{g_0}=\int_{\mathbb{D}}|\D\psi|^2|dz|^2
	 	\end{align*}
	 	Since $\psi(z)=\log(2)-\log(1+|z|^2)$ and $\psi$ is real, we have
	 	\begin{align}\label{eq7}
	 		\int_{\mathbb{D}}|\D\psi|^2|dz|^2&=4\int_{\mathbb{D}}|\p{z}\psi|^2|dz|^2=4\int_{\mathbb{D}}\frac{|z|^2|dz|^2}{(1+|z|^2)^2}=4\pi\log(2)-2\pi.
	 	\end{align}
	 	Therefore, we get \eqref{eq4}, \eqref{eq5}, \eqref{eq6} and \eqref{eq7}
	 	\begin{align}\label{eq8}
	 		&\int_{\Omega_1}|d(-\mu_1+\theta_1)|_{g_0}^2d\mathrm{vol}_{g_0}+2\int_{\Omega_1}K_{g_0}(-\mu_1+\theta_1)d\mathrm{vol}_{g_0}
	 		+2\int_{\partial\Omega_1}k_{g_0}(-\mu_1+\theta_1)d\mathscr{H}^1_{g_0}\nonumber\\
	 		&=\int_{\Omega_1}|d\mu_1|_{g_0}^2d\mathrm{vol}_{g_0}-2\int_{\Omega_1}\s{d\mu_1}{d\theta_1}_{g_0}d\mathrm{vol}_{g_0}+\int_{\Omega_1}|d\theta_1|_{g_0}^2d\mathrm{vol}_{g_0}\nonumber\\
	 		&-2\int_{\Omega_1}\theta_1d\mathrm{vol}_{g_0}-2\int_{\Gamma}k_{g_0}\theta_1d\mathscr{H}^1_{g_0}+2\int_{\Gamma}\partial_{\nu}G_{\Omega_1}\theta_1d\mathscr{H}^1_{g_0}\nonumber\\
	 		&=\int_{\Omega_1}|d\mu_1|_{g_0}^2d\mathrm{vol}_{g_0}-\colorcancel{2\int_{\Omega_1}\theta_1d\mathrm{vol}_{g_0}}{red}-\colorcancel{2\int_{\Gamma}k_{g_0}\theta_1d\mathscr{H}^1_{g_0}}{blue}+2\int_{\Gamma}\partial_{\nu}G_{\Omega_1}\theta_1d\mathscr{H}^1_{g_0}+\int_{\Omega_1}|d\theta_1|_{g_0}^2d\mathrm{vol}_{g_0}\nonumber\\
	 		&-2\int_{\Omega_1}|d\mu_1|_{g_0}^2d\mathrm{vol}_{g_0}-2\int_{\Gamma}\partial_{\nu}G_{\Omega_1}\mu_1d\mathscr{H}^1_{g_0}+\colorcancel{2\int_{\Omega_1}\theta_1d\mathrm{vol}_{g_0}}{red}+\colorcancel{2\int_{\Gamma}k_{g_0}\theta_1d\mathscr{H}^1_{g_0}}{blue}\nonumber\\
	 		&=-\int_{\Omega_1}|d\mu_1|_{g_0}^2d\mathrm{vol}_{g_0}+\int_{\Omega_1}|d\theta_1|_{g_0}^2d\mathrm{vol}_{g_0}+2\int_{\Gamma}\partial_{\nu}G_{\Omega_1}(-\mu_1+\theta_1)d\mathscr{H}^1_{g_0}\nonumber\\
	 		&=-\int_{\Omega_1}|d\mu_1|_{g_0}^2d\mathrm{vol}_{g_0}+2\int_{\Gamma}\partial_{\nu}G_{\Omega_1}(-\mu_1+\theta_1)d\mathscr{H}^1_{g_0}+4\pi\log(2)-2\pi.
	 	\end{align}
	 	Now, since $\theta_1=\psi\circ f^{-1}$, and $\psi(z)=0$ for all $z\in S^1$, we have $\theta_1=0$ on $\Gamma$. Therefore, we have
	 	\begin{align}\label{eq9}
	 		2\int_{\Gamma}\partial_{\nu}G_{\Omega_1}(-\mu_1+\theta_1)d\mathscr{H}^1_{g_0}&=-2\int_{\Gamma}\partial_{\nu}G_{\Omega_1}\mu_1d\mathscr{H}^1_{g_0}.
	 	\end{align}
	 	Now, since $-\Delta_{g_0}\mu_1=1$ and $K_{g_0}=1$, we have
	 	\begin{align}\label{eq10}
	 		\int_{\Omega_1}G_{\Omega_1}K_{g_0}d\mathrm{vol}_{g_0}&=-\int_{\Omega_1}G_{\Omega_1}\Delta_{g_0}\mu_1d\mathrm{vol}_{g_0}\nonumber\\
	 		&=-\int_{\Omega_1}\mu_1\Delta_{g_0}G_{\Omega_1}d\mathrm{vol}_{g_0}-\int_{\Omega_1}\left(G_{\Omega_1}\partial_{\nu}\mu_1-\mu_1\partial_{\nu}G_{\Omega_1}\right)d\mathscr{H}^1_{g_0}\nonumber\\
	 		&=-2\pi\mu_1(p_1)+\int_{\partial\Omega_1}\mu_1\partial_{\nu}G_{\Omega_1}d\mathscr{H}^1_{g_0},
	 	\end{align}
	 	where we used the Dirichlet condition $G_{\Omega_1}=0$ on $\partial\Omega_1=\Gamma$. Therefore, \eqref{eq9}
	 	and \eqref{eq10} imply that
	 	\begin{align}\label{eq11}
	 		2\int_{\Gamma}\partial_{\nu}G_{\Omega_1}(-\mu_1+\theta_1)d\mathscr{H}^1_{g_0}=-2\int_{\Omega_1}G_{\Omega_1}K_{g_0}d\mathrm{vol}_{g_0}-4\pi\mu_1(p_1).
	 	\end{align}
	 	Gathering \eqref{eq8} and \eqref{eq11} yields
	 	\begin{align}\label{eq12}
	 		&\int_{\Omega_1}|d(-\mu_1+\theta_1)|_{g_0}^2d\mathrm{vol}_{g_0}+2\int_{\Omega_1}K_{g_0}(-\mu_1+\theta_1)d\mathrm{vol}_{g_0}
	 		+2\int_{\partial\Omega_1}k_{g_0}(-\mu_1+\theta_1)d\mathscr{H}^1_{g_0}\nonumber\\
	 		&=-\int_{\Omega_1}|d\mu_1|_{g_0}^2d\mathrm{vol}_{g_0}-2\int_{\Omega_1}G_{\Omega_1}K_{g_0}d\mathrm{vol}_{g_0}-4\pi\mu_1(p_1)+4\pi\log(2)-2\pi.
	 	\end{align}
	 	We also have
	 	\begin{align}\label{eq13}
	 		\int_{\partial\Omega_1}\partial_{\nu}\left(-\mu_1+\theta_1\right)d\mathscr{H}^1_{g_0}+\int_{\partial\Omega_2}\partial_{\nu}\left(-\mu_2+\theta_2\right)d\mathscr{H}^1_{g_0}=0.
	 	\end{align}
	 	Indeed, we have by the boundary conditions \eqref{boundary}
	 	\begin{align}\label{eq13bis}
	 	\int_{\partial \Omega_1}\partial_{\nu}\mu_1d\mathscr{H}^1_{g_0}&=\int_{\Gamma}k_{g_0}d\mathscr{H}^1_{g_0}-\int_{\partial \Omega}\partial_{\nu}G_{\Omega_j}d\mathscr{H}^1_{g_0}=\int_{\Gamma}k_{g_0}d\mathscr{H}^1_{g_0}-\int_{\Omega_1}\Delta_{g_0}G_{\Omega_1}d\mathrm{vol}_{g_0}\nonumber\\
	 	&=\int_{\Gamma}k_{g_0}d\mathrm{vol}_{g_0}-2\pi\nonumber\\
	 	\int_{\partial\Omega_2}\partial_{\nu}\mu_2d\mathscr{H}^1_{g_0}&=-\int_{\Gamma}k_{g_0}d\mathrm{vol}_{g_0}-2\pi.
	 	\end{align}
	 	We also have by the conformal invariance of the Dirichlet energy
	 	\begin{align}\label{eq13ter}
	 	\int_{\partial\Omega_1}\partial_{\nu}\theta_1d\mathscr{H}^1_{g_0}=\int_{\Omega_1}\Delta_{g_0}\theta_1d\mathrm{vol}_{g_0}=\int_{\mathbb{D}}\Delta\psi|dz|^2=\int_{S^1}\partial_{\nu}\psi\,d\mathscr{H}^1=-2\pi.
	 	\end{align}
	 	Therefore, we finally get by \eqref{eq13bis} and \eqref{eq13ter}
	 	\begin{align*}
	 	&\int_{\partial\Omega_1}\partial_{\nu}\left(-\mu_1+\theta_1\right)d\mathscr{H}^1_{g_0}+\int_{\partial\Omega_2}\partial_{\nu}\left(-\mu_2+\theta_2\right)d\mathscr{H}^1_{g_0}\\
	 	&=-\left(\int_{\Gamma}k_{g_0}d\mathrm{vol}_{g_0}-2\pi\right)-2\pi-\left(-\int_{\Gamma}k_{g_0}d\mathrm{vol}_{g_0}-2\pi\right)-2\pi=0
	 	\end{align*}
	 	which proves \eqref{eq13}.
	 	
	 	Finally, we deduce by \eqref{eq1}, \eqref{eq12} and \eqref{eq13} that
	 	\begin{align}\label{eq15}
	 		&-\pi\,I^L(\Gamma)=-\int_{\Omega_1}|d\mu_1|_{g_0}^2d\mathrm{vol}_{g_0}-2\int_{\Omega_1}G_{\Omega_1}K_{g_0}d\mathrm{vol}_{g_0}-4\pi\mu_1(p_1)+4\pi\log(2)-2\pi\nonumber\\
	 		&-\int_{\Omega_2}|d\mu_2|_{g_0}^2d\mathrm{vol}_{g_0}-2\int_{\Omega_1}G_{\Omega_2}K_{g_0}d\mathrm{vol}_{g_0}-4\pi\mu_2(p_2)+4\pi\log(2)-2\pi
	 		\nonumber\\
	 		&=-\int_{\Omega_1}|d\mu_1|_{g_0}^2d\mathrm{vol}_{g_0}-\int_{\Omega_2}|d\mu_2|_{g_0}^2d\mathrm{vol}_{g_0}-2\int_{\Omega_1}G_{\Omega_1}K_{g_0}d\mathrm{vol}_{g_0}-2\int_{\Omega_1}G_{\Omega_2}K_{g_0}d\mathrm{vol}_{g_0}\nonumber\\
	 		&-4\pi\mu_1(p_1)-4\pi\mu_2(p_2)+8\pi\log(2)-4\pi.
	 	\end{align}
	 	Recalling the identity \eqref{rho}, we finally deduce that 
	 	\begin{align}\label{new_loewner1}
	 		\pi\,I^L(\Gamma)&=\int_{\Omega_1}|d\mu_1|_{g_0}^2d\mathrm{vol}_{g_0}+\int_{\Omega_2}|d\mu_2|_{g_0}^2d\mathrm{vol}_{g_0}+2\int_{\Omega_1}G_{\Omega_1}K_{g_0}d\mathrm{vol}_{g_0}+2\int_{\Omega_1}G_{\Omega_2}K_{g_0}d\mathrm{vol}_{g_0}+4\pi\nonumber\\
	 		&+4\pi\log|\D f_1(0)|+4\pi\log|\D f_2(0)|-12\pi\log(2).
	 	\end{align}
	 	Now we introduce the functional  
	 	\begin{align}\label{new_loewner2}
	 		\mathscr{E}(\Gamma)&=\int_{\Omega_1}|d\mu_1|^2_{g_0}d\mathrm{vol}_{g_0}+\int_{\Omega_2}|d\mu_2|^2_{g_0}d\mathrm{vol}_{g_0}+2\int_{\Omega_1}G_{\Omega_1}K_{g_0}d\mathrm{vol}_{g_0}+2\int_{\Omega_2}G_{\Omega_2}K_{g_0}d\mathrm{vol}_{g_0}+4\pi.
	 	\end{align}
   We deduce that 
	 	\begin{align}\label{new_loewner3}
	 		I^L(\Gamma)=\frac{1}{\pi}\mathscr{E}(\Gamma)+4\log|\D f_1(0)|+4\log|\D f_2(0)|-12\log(2).
	 	\end{align}
	 	This concludes the proof of the theorem.
	 \end{proof}
 
	 \begin{rem}\label{hemisphere}
	 	{We check} that {equality \eqref{new_loewner3}} holds for {the equator $S^1$}. {Using the definition \eqref{defs1} with the conformal maps $f, g$ being the identity maps, we see that $I^L(S^1)$} vanishes. 
	 	
	 	Let us first check that
	 	\begin{align*}
	 		\mathscr{E}(S^1)=0
	 	\end{align*}
	 	with the marked points $p_1 = S = (0,0,-1)$ and $p_2 = N = (0,0,1)$.
	 	This identity justifies the term $4\pi$ in the definition of $\mathscr{E}$ as we remarked earlier. 
   
	 For this,	since $K_{g_0}=1$ on {$\Omega_1 = S_-^2$,} 
	 	after making a stereographic projection {$\pi : S^2 \setminus \{N \} \to \mathbb C$ sending $S$ to $0$}, we find
	 		\begin{align}\label{optimal2}
	 		\int_{S^2_{-}} G_{S^2_{-}}K_{g_0}d\mathrm{vol}_{g_0}&=\int_{\mathbb{D}}G_{\mathbb{D}}(z)\frac{4|dz|^2}{(1+|z|^2)^2}=\int_{\mathbb{D}}\frac{4\log|z|}{(1+|z|^2)^2}|dz|^2=8\pi\int_{0}^{1}\frac{r\log r}{(1+r^2)^2}dr\nonumber\\
	 		&=8\pi\left[-\frac{1}{2(1+r^2)}\log(r)+\frac{1}{2}\log(r)-\frac{1}{4}\log(1+r^2)\right]_0^1=-2\pi\log(2).
	 	\end{align}

{We take $f_1 = \pi^{-1}|_{\mathbb D}$ which is consistent with $f_1(0) = p_1 = S$. By \eqref{geodesic} we have $\partial_\nu \mu = -1$ on $S^1$ and $- \Delta_{g_0}\mu=1$ in $S^2_-$ which translates to 
\begin{align*}
	 		-\Delta\mu (z) =\frac{4}{(1+|z|^2)^2} \qquad \text{in}\;\,\mathbb{D}.
	 	\end{align*}
	 	We deduce by a direct verification that
	 	$\mu(z)=-\log(1+|z|^2)$. (This is easy to guess since by \eqref{eq:change_factor_mu}, $\mu$ can be computed from the conformal factor of $f_1$.)
}
Therefore, we have by the conformal invariance of the Dirichlet energy 
	 	\begin{align}\label{optimal3}
	 		\int_{S^2_{-}}\left|\omega-\ast dG_{S^2_{-}}\right|^2_{g_0}d\mathrm{vol}_{g_0}&=\int_{S^2_{-}}|d\mu|_{g_0}^2d\mathrm{vol}_{g_0}=\int_{\mathbb{D}}|\D \mu(x)|^2dx=\int_{\mathbb{D}}\left|\frac{2x}{1+|x|^2}\right|^2dx
	 		=8\pi\int_{0}^1\frac{r^3}{(1+r^2)^2}dr\nonumber\\
	 		&=8\pi\int_{0}^1\left(\frac{r}{1+r^2}-\frac{r}{(1+r^2)^2}\right)dr
	 		=8\pi\left[\frac{1}{2}\log(1+r^2)+\frac{1}{2}\frac{1}{1+r^2}\right]_0^1\nonumber\\
	 		&=8\pi\left(\frac{1}{2}\log(2)-\frac{1}{4}\right)
	 		=4\pi\log(2)-2\pi.
	 	\end{align}
	 	
	 	 	Finally, by \eqref{optimal2} and \eqref{optimal3}, we have
	 	\begin{align*}
	 		\int_{S^2_{-}}|\omega-\ast\, dG_{S^2_{-}}|_{g_0}^2d\mathrm{vol}_{g_0}+2\int_{S^2_{-}}G_{S^2_{-}}\,K_{g_0}d\mathrm{vol}_{g_0}+2\pi
	 		&=\left(4\pi\log(2)-2\pi\right)-4\pi\log(2)+2\pi=0.
	 	\end{align*}
{Applying the same computation to $\Omega_2 = S^2_+$ with $f_2 = - f_1 : \mathbb D \to S^2_+$ which is consistent with the choice $p_2 = N = f_2 (0)$, we obtain the claimed identity $\mathscr{E}(S^1)=0$.

	 	Now we show that 
	 	\begin{align}\label{second_part}
	 		4\log|\D f_1(0)|+4\log|\D f_2(0)|
	 		-12\log(2) =0.
	 	\end{align}
	 	Since the inverse stereographic projection is given by
	 	\begin{align*}
	 		f_1 (z) = \pi^{-1}(z
	 		)=\left(\frac{2\,\Re(z)}{1+|z|^2},\frac{2\,\Im(z)}{1+|z|^2},\frac{-1+|z|^2}{1+|z|^2}\right),
	 	\end{align*}
	 	we compute directly that 
	 	\begin{align*}
	 		|\D f_1(0)| = |\D f_2(0)| =2\sqrt{2}
	 	\end{align*}
	 	which concludes the proof of \eqref{second_part} and shows the identity \eqref{new_loewner3} for the circle $S^1$.}
 \end{rem}

\section{Construction of Harmonic Moving Frames for Weil-Petersson Curves}\label{section7}

  In the previous section, we showed that in the case of smooth curves, the Loewner energy was equal to a renormalised Dirichlet energy of a specific harmonic moving frame. In this section, we will directly  construct harmonic moving frames satisfying appropriate boundary conditions for arbitrary Weil-Petersson quasicircles. In the next section, we will show that Theorem \ref{wp10} holds for non-smooth curves. 
  
	 Before stating the main theorem of this section, recall an easy lemma on harmonic vector fields.
	 
	 \begin{lemme}\label{harmonic}
	 Let $\Sigma\subset \R^3$ be a smooth surface, $\n:\Sigma\rightarrow S^2$ its unit normal, and $g=\iota^{\ast}g_{\R^3}$ be the induced metric. Assume that $\vec{u}:\Sigma\rightarrow S^2$ is a smooth critical point of the Dirichlet energy amongst $S^2$-valued maps such that $\s{{\vec u}}{\n}=0$. 
	 Then $\vec{u}$ satisfies the following Euler-Lagrange equation:
	 \begin{align}\label{el}
	 -\Delta_g\vec{u}=|d\vec{u}|_g^2\vec{u}+\left(\s{d\vec{u}}{d\n}_g+\s{\vec{u}}{\Delta_g\n}\right)\n.
	 \end{align}
	 \end{lemme}
	 \begin{proof}
	 We proceed as in \cite{helein} in Lemme ($1.4.10$), taking variations $\vec{X}$ that also satisfy $\s{\vec{X}}{\n}=0$. 
	 \end{proof}

 The following result is the same as Theorem~\ref{wp10}, but for a general Weil-Petersson quasicircle.
	 \begin{theorem}\label{key2}
	 	Let $\Gamma\subset S^2$ a Weil-Petersson quasicircle and let $\Omega_1,\Omega_2$ {be} the two open connected components of $S^2\setminus\Gamma$. 
	 	{For} $j=1,2$ {and} for all $p_j\in\Omega_j$, there  {exists a} harmonic moving frame $(\e_j,\f_j):\Omega_j\setminus\ens{p_j}\rightarrow U\Omega_j\times U\Omega_j$ such that the Cartan form $\omega_j=\s{\e_j}{d\f_j}$ admits the decomposition
	 	\begin{align}\label{cartan}
	 	  	\omega_j=\ast\,d\left(G_{\Omega_j}+\mu_j\right),
		\end{align}
	 	where $G_{\Omega_j}=G_{\Omega_j,p_j}:\Omega_j\setminus\ens{p_j}\rightarrow \R$ is the Green's function of the Laplacian $\Delta_{g_0}$ on $\Omega_j$ with Dirichlet boundary condition and singularity $p_j\in\Omega_j$, and $\mu_j\in C^{\infty}(\Omega_j)$ satisfies
	 	\begin{align}\label{systemu}
	 	    \left\{\begin{alignedat}{2}
	 	         -\Delta_{g_0}\mu_j&=1\qquad&&\text{in}\;\,\Omega_j\\
	 	         \partial_{\nu}\mu_j&=k_{g_0}-\partial_{\nu}G_{\Omega_j}\qquad&&\text{on}\;\,\partial\Omega_j,
	 	    \end{alignedat} \right.
	 	\end{align}
	 	where $k_{g_0}$ is the geodesic curvature on $\Gamma=\partial\Omega_j$. 
	 \end{theorem}
 \begin{rem}
     The Neumann condition for $\mu_j$ ($1\leq j\leq 2$) is understood in the sense of distributions, since the geodesic curvature is only in $H^{-1/2}(\Gamma)$ in general (see the appendix for more details).
 \end{rem}

	 \begin{proof}
        Rather than using the moving frame that comes from a Ginzburg-Landau type minimisation as in \cite{lauromain}—that would have had to be carried in the geometric setting of domains of $S^2$—we directly use the uniformisation theorem and the geometric formula of \cite{yilinvention} (that does not require any regularity on the curve $\Gamma$)  to construct the relevant moving frame.
	 	We now construct the moving frame on $\Omega_1$. The construction for $\Omega_2$ is similar.
	 	
	 	\textbf{Step 1.} Definition of $(\e_{1},\f_{1})$ and $\mu_{1}$.
	 
	 	Let $\pi:S^2\setminus\ens{N}\rightarrow \C$ be the standard stereographic projection and assume without loss of generality that $N\in \Omega_2$. Let $\Omega=\pi(\Omega_1)\subset \C$ be the image domain and $\gamma=\pi(\Gamma)\subset\C$ be the image curve. Thanks to the Uniformisation Theorem, there exists {a} univalent holomorphic map $f:\mathbb{D}\rightarrow \Omega$ {such that $f(0) = \pi (p_1)$.} 

	 	Now, let $f_1=\pi^{-1}\circ f:\mathbb{D}\rightarrow \Omega_1$. Notice that  $f_1 (0) = p_1$. 
        Explicitly, we have
        \begin{align*}
            f_1(z)=\pi^{-1}(f(z))=\left(\frac{2\,\Re(f(z))}{1+|f(z)|^2},\frac{2\,\Im(f(z))}{1+|f(z)|^2},\frac{-1+|f(z)|^2}{1+|f(z)|^2}\right).
        \end{align*}
        A direct computation show that
	 	\begin{align*}
	 		\p{z}f_1=f'\left(\frac{(1-\bar{f}^2)}{(1+|f|^2)^2},\frac{-i(1+\bar{f}^2)}{(1+|f|^2)^2},\frac{2\bar{f}}{(1+|f|^2)^2}\right).
	 	\end{align*}
	 	Now, by analogy with the construction in Section \ref{zeta} (see also \cite{lauromain}, Proposition $5.1$), define $\mu_1:\Omega_1\rightarrow \R$ and $\e_1:\Omega_1\rightarrow U S^2$ and $\f_1:\Omega_1\rightarrow U S^2$ by
	 	\begin{align*}
	 	\left\{\begin{alignedat}{1}
	 		\p{r}f_1&=e^{\mu_1\circ f_1}\f_1\circ f_1\\
	 		\frac{1}{r}\p{\theta}f_1&=e^{\mu_1\circ f_1}\e_1\circ f_1.
	 		\end{alignedat}\right.
	 	\end{align*}
	 	Then, we have {from direct computations} 
	 	\begin{align*}
	 		e^{2\mu_1\circ f_1}=|\p{r}f_1|^2=\frac{1}{r^2}|\p{\theta}f_1|^2=2|\p{z}f_1|^2=\frac{4|f'(z)|^2}{(1+|f(z)|^2)^2}.
	 	\end{align*}
	 	Therefore, we deduce if $\mu=\mu_1\circ f_1$ that
	 	\begin{align}\label{defmu}
	 	\mu(z)=\log|f'(z)|-\log\left(1+|f(z)|^2\right)+\log(2).
	 	\end{align}
Since $\p{z}=\dfrac{1}{2}(\p{x}-i\,\p{y})$, we have
	 	\begin{align*}
	 	\left\{\begin{alignedat}{1}
	 		\p{r}f_1&=\cos(\theta)\p{x}f_1+\sin(\theta
	 		)\p{y}f_1=\Re\left(\frac{z}{|z|}\right)\Re(\p{z}f_1)-\Im\left(\frac{z}{|z|}\right)\Im\left(\p{z}f_1\right)\\
	 		\frac{1}{r}\p{\theta}f_1&=-\sin(\theta)\p{x}f_1+\cos(\theta)\p{y}f_1=-\Im\left(\frac{z}{|z|}\right)\Re(\p{z}f_1)-\Re\left(\frac{z}{|z|}\right)\Im\left(\p{z}f_1\right).
	 		\end{alignedat}\right.
	 	\end{align*}
	 	By the elementary identities for all $a,b\in \C$
	 	\begin{align*}
       \left\{\begin{alignedat}{1}
	 		\Re(a)\Re(b)+\Im(a)\Im(b)&=\Re(ab)\\
	 		\Re(a)\Im(b)+\Im(a)\Re(b)&=\Im(ab),
            \end{alignedat}\right.
	 	\end{align*}
	 	we deduce that 
	 	\begin{align*}
	 	\left\{\begin{alignedat}{1}
	 		\p{r}f_1&=\Re\left(\frac{z}{|z|}\p{z}f_1\right)
	 		=\Re\left(\frac{\z}{|z|}\bar{f'(z)}\left(\frac{(1-f(z)^2)}{(1+|f(z)|^2)^2},\frac{i(1+f(z)^2)}{(1+|f(z)|^2)^2},\frac{2f(z)}{(1+|f(z)|^2)^2}\right)\right)\\
	 		\frac{1}{r}\p{\theta}f_1&=-\Im\left(\frac{z}{|z|}\p{z}f_1\right)
	 		=\Im\left(\frac{\z}{|z|}\bar{f'(z)}\left(\frac{(1-f(z)^2)}{(1+|f(z)|^2)^2},\frac{i(1+f(z)^2)}{(1+|f(z)|^2)^2},\frac{2f(z)}{(1+|f(z)|^2)^2}\right)\right).
	 		\end{alignedat}\right.
	 	\end{align*}
	 	More generally, if $\varphi:\C\rightarrow \C$ is a smooth complex function, we have 
	   \begin{align}\label{pthete}
	 	 	\partial_{\theta}\varphi=-\Im\left(z\right)\left(\p{z}+\p{\z}\right)\varphi+\Re\left({z}\right)i\left(\p{z}-\p{\z}\right)\varphi=i\left(z\,\p{z}\varphi-\z\,\p{\z}\varphi\right).
	 	\end{align}
	 	Since
	 	\begin{align*}
	 		|\p{z}f_1|=\frac{1}{r}|\p{\theta}f_1|=\frac{2|f'(z)|}{1+|f(z)|^2},
	 	\end{align*}
	 	we deduce that 
	 	\begin{align}\label{frame}
	 	\left\{\begin{alignedat}{1}
	 		\f_1\circ f_1&=\Re\left(\bar{\frac{zf'(z)}{|z f'(z)|}}\left(\frac{(1-f(z)^2)}{(1+|f(z)|^2)},\frac{i(1+f(z)^2)}{(1+|f(z)|^2)},\frac{2f(z)}{(1+|f(z)|^2)}\right)\right)\\
	 		\e_1\circ f_1&=\Im\left(\bar{\frac{zf'(z)}{|z f'(z)|}}\left(\frac{(1-f(z)^2)}{(1+|f(z)|^2)},\frac{i(1+f(z)^2)}{(1+|f(z)|^2)},\frac{2f(z)}{(1+|f(z)|^2)}\right)\right).
	 		\end{alignedat}\right. 
	 	\end{align}
	 	Notice that 
	 	\begin{align*}
	 		F(z)=\left((1-f(z)^2),i(1+f(z)^2),2f(z)\right)
	 	\end{align*}
	 	is a holomorphic null vector, \emph{i.e.} $\s{F(z)}{F(z)}=0$, so we see directly since $|\e_1|=|\f_1|=1$ that 
	 	\begin{align*}
	 		\s{\e_1}{\f_1}=0.
	 	\end{align*}
	 	
	 	\textbf{Step 2.} Verification of the system \eqref{systemu}. 

        \textbf{Part 1.} Equation on $\Omega_1$ for $\mu_1$.
	 	
	 	  Since $\mu=\mu_1\circ f_1$, the equation $-\Delta_{g_0}\mu_1=1$ is equivalent to 
	 	\begin{align}\label{new_liouville}
	 		-\Delta\mu=e^{2\mu}
	 	\end{align}
        Thanks to the explicit expression in \eqref{defmu}, and by harmonicity of $\log|f'|$, we have
	 	\begin{align*}
	 	\Delta\mu=4\,\p{\z}\left(\frac{f'(z)\bar{f'(z)}}{1+|f(z)|^2}\right)=4\left(\frac{|f'(z)|^2}{1+|f(z)|^2}-\frac{|f'(z)|^2|f(z)|^2}{(1+|f(z)|^2)^2}\right)=\frac{4|f'(z)|^2}{(1+|f(z)|^2)^2}=e^{2\mu}. 
	 	\end{align*}
	 	Recalling that 
	 		\begin{align*}
	 		g_{\widehat{\C}} =\frac{4|dz|^2}{(1+|z|^2)^2}=(\pi^{-1})^{\ast}g_{0},
	 	\end{align*}
	 	we deduce that 
	\begin{align*}
	 		\frac{4|f'(z)|^2|dz|^2}{(1+|f(z)|^2)^2}=f^{\ast} g_{\widehat{\C}} = f^{\ast}((\pi^{-1})^{\ast} g_{0})=(\pi^{-1}\circ f)^{\ast} g_{0} = f_1^{\ast} g_{0}.
	 	\end{align*}
	 	Therefore, \eqref{new_liouville} can be rewritten as
	 	\begin{align*}
	 		-\Delta_{f_1^{\ast}g_{0}}(\mu_1\circ f_1)=1
	 	\end{align*}
	 	or by conformal invariance of the Dirichlet energy
	 	\begin{align}\label{liouville1}
	 		-\Delta_{g_{0}}\mu_1=1.
	 	\end{align}

        \textbf{Part 2.} Boundary conditions.
	 	
        If ${h}:\C\rightarrow \R$ is a smooth function, we have
	 	\begin{align*}
	 	    \partial_{\nu}{h}& =\frac{x}{\sqrt{x^2+y^2}}\p{x}{h}+\frac{y}{\sqrt{x^2+y^2}}\p{y}{h}=\frac{\Re(z)}{|z|}\left(\p{z}+\p{\z}\right){h} +\frac{\Im(z)}{|z|}i\left(\p{z}-\p{\z}\right){h}
	 	    =2\,\Re\left(\frac{z}{|z|}\p{z}{h}\right).
	 	\end{align*}
	 	This implies since $\mu(z)=\log|f'(z)|-\log(1+|f(z)|^2)+\log(2)$ by \eqref{defmu} that
	 	\begin{align*}
	 	\p{z}\mu(z)=\frac{1}{2}\frac{f''(z)}{f'(z)}-\frac{f'(z)}{f(z)}\frac{|f(z)|^2}{1+|f(z)|^2},
	 	\end{align*}
	 	and
	 	\begin{align*}
	 	\partial_{\nu}\mu=\Re\left(z\frac{f''(z)}{f'(z)}-2z\frac{f'(z)}{f(z)}\frac{|f(z)|^2}{1+|f(z)|^2}\right)
	 	\end{align*}
	 	{on $\partial \mathbb D$ in the distributional sense. We will comment on it in Remark~\ref{remTheoA}.}
	 	Recall that the geodesic curvature on {$\partial \Omega_1$} 
	 	is given (see \cite{changeodesic}) by 
	 	\begin{align}\label{eq:def_geodesic_curvature}
	 	k_{g_0}=\s{\e_1}{\partial_{\theta}\f_1}.
	 	\end{align}
	 	{From \eqref{frame} it is natural to define}  
	 	\begin{align*}
	 		\varphi(z)=\bar{\frac{zf'(z)}{|z f'(z)|}}\left(\frac{(1-f(z)^2)}{(1+|f(z)|^2)},\frac{i(1+f(z)^2)}{(1+|f(z)|^2)},\frac{2f(z)}{(1+|f(z)|^2)}\right)=\chi(z)\psi(z),
	    \end{align*}
	    where
	    \begin{align}\label{defchipsi}
	    \left\{\begin{alignedat}{1}
	    \chi(z)&=\frac{\bar{zf'(z)}}{|zf'(z)|}=\exp\left(\frac{1}{2}\log\left(\bar{zf'(z)}\right)-\frac{1}{2}\log\left(zf'(z)\right)\right)\\
	    \psi(z)&=\left(\frac{1-f(z)^2}{1+|f(z)|^2},\frac{i(1+f(z)^2)}{1+|f(z)|^2},\frac{2f(z)}{1+|f(z)|^2}\right),
	    \end{alignedat}\right.
	    \end{align}
	    {so that $\f_1 \circ f_1 = \Re (\varphi)$ and $\e_1 \circ f_1 = \Im (\varphi)$.}
	    Then, we compute
	    \begin{align}\label{idchi}
	    \left\{\begin{alignedat}{1}
	    \p{z}\chi&=-\frac{1}{2}\left(\frac{f''(z)}{f'(z)}+\frac{1}{z}\right)\chi\\
	    \p{\z}\chi&=\frac{1}{2}\bar{\left(\frac{f''(z)}{f'(z)}+\frac{1}{z}\right)}\chi.
	    \end{alignedat}\right.
	    \end{align}
	    We also get
	    \begin{align}\label{idpsi0}
	    \left\{\begin{alignedat}{1}
	    \p{z}\psi&=-\frac{f'(z)\bar{f(z)}}{1+|f(z)|^2}\psi+\frac{2f'(z)}{1+|f(z)|^2}\left(-f(z),i\,f(z),1\right)\\
	    \p{\z}\psi&=-\frac{\bar{f'(z)}f(z)}{1+|f(z)|^2}\psi.
	    \end{alignedat}\right.
	    \end{align}
	    Since $\s{\psi}{\psi}=0$, 
	   we have $\s{\p{z}\psi}{\psi}=\s{\p{\z}\psi}{\psi}=0$. In particular, we have
	    \begin{align}\label{idpsi1}
	    \s{(-f(z),i\,f(z),1)}{\psi}=\frac{1}{1+|f(z)|^2}\s{(-f(z),i\,f(z),1)}{(1-f(z)^2,i(1+f(z)^2),2f(z))}=0,
	    \end{align}
	    while
	    \begin{align*}
	    &\bs{(-f(z),i\,f(z),1)}{\left(1-\bar{f(z)}^2,-i\left(1+\bar{f(z)}^2\right),2\bar{f(z)}\right)}\\
	    &=-f(z)+\bar{f(z)}|f(z)|^2+f(z)+\bar{f(z)}|f(z)|^2+2{\bar{f(z)}}
	    =2\bar{f(z)}(1+|f(z)|^2),
	    \end{align*}
	    so that
	    \begin{align}\label{idpsi2}
	    \s{(-f(z),i\,f(z),1)}{{\bar{\psi}}}=2\bar{f(z)}.
	    \end{align}
	    Therefore, we deduce by \eqref{idpsi0}, \eqref{idpsi1} and \eqref{idpsi2} that
	    \begin{align}\label{idpsi}
	    \left\{\begin{alignedat}{1}
	    &\s{\varphi}{\varphi}=\s{\p{z}\varphi}{\varphi}=\s{\p{\z}\varphi}{\varphi}=\s{\psi}{\psi}=\s{\p{z}\psi}{\psi}=\s{\p{\z}\psi}{\psi}=0\\
	    &\s{\p{z}\psi}{\bar{\psi}}=
	    \frac{2f'(z)\bar{f(z)}}{1+|f(z)|^2}\\
	    &\s{\p{\z}\psi}{\bar{\psi}}=-\frac{2\bar{f'(z)}f(z)}{1+|f(z)|^2}.
	    \end{alignedat}\right.
	    \end{align}
	    The identities \eqref{idchi} and \eqref{idpsi} imply that
	    \begin{align*}
	    z\,\p{z}\varphi-\z\,\p{z}\varphi=-\left(\Re\left(z\frac{f''(z)}{f'(z)}\right)+1\right)\varphi+\chi(z)\left(z\,\p{z}\psi-\z\,\p{\z}\psi\right),
	    \end{align*}
	    and since $|\chi|^2=1$ and $|\psi|^2=2$, we have
	    \begin{align*}
	    &\s{z\,\p{z}\varphi-\z\,\p{\z}\varphi}{\varphi+\bar{\varphi}}=-2\,\left(\Re\left(z\frac{f''(z)}{f'(z)}\right)+1\right)\\
	    &+\chi(z)\left(z\left(\chi(z)\s{\p{z}\psi}{\psi}+\bar{\chi(z)}\s{\p{z}\psi}{\bar{\psi}}\right)-\z\left(\chi(z)\s{\p{\z}\varphi}{\varphi}+\bar{\chi(z)}\s{\p{\z}\psi}{\bar{\psi}}\right)\right)\\
	    &=-2\,\left(\Re\left(z\frac{f''(z)}{f'(z)}\right)+1\right)+2z\frac{f'(z)}{f(z)}\frac{|f(z)|^2}{1+|f(z)|^2}+2\z\bar{\frac{f'(z)}{f(z)}}\frac{|f(z)|^2}{1+|f(z)|^2}\\
	    &=-2-2\,\Re\left(z\frac{f''(z)}{f'(z)}-2z\frac{f'(z)}{f(z)}\frac{|f(z)|^2}{1+|f(z)|^2}\right),
	    \end{align*}
	    so that 
	    \begin{align}\label{kg0}
	    k_{g_0}&=\s{\e}{\p{\theta}\f}=-\s{\p{\theta}\e}{\f}
	    =-\s{\partial_{\theta}\Im(\varphi)}{\Re(\varphi)}=-\Im\left(\s{\partial_{\theta}\varphi}{\Re(\varphi)}\right)\nonumber\\
	    &=-\frac{1}{2}\Im\left(\s{i\left(z\p{z}\varphi-\z\p{\z}\right)}{\varphi+\bar{\varphi}}\right)
	    =-\frac{1}{2}\Re\left(\s{z\p{z}\varphi-\z\p{\z}\varphi}{\varphi+\bar{\varphi}}\right)\nonumber
	    \\
	    &=\Re\left(z\frac{f''(z)}{f'(z)}-2z\frac{f'(z)}{f(z)}\frac{|f(z)|^2}{1+|f(z)|^2}\right)+1=\partial_{\nu}\mu+1=\partial_{\nu}\mu+\partial_{\nu}G_{\mathbb{D}},
	    \end{align}
	    which concludes the proof of 
	    the system \eqref{systemu} by the conformal invariance of the Green's function (we denoted for simplicity $G_{\mathbb{D}}=G_{\mathbb{D},0}=\log|\,\cdot\,|$).

	    \textbf{Step 3.} Verification that $(\e_1,\f_1)$ is a harmonic moving frame. 
	    
	 	Now, thanks to Lemma \ref{harmonic} and \eqref{el}, the maps $\e_1$ and $\f_1$ are unit harmonic moving frames if and only if they satisfy in the distributional sense (see Theorem \ref{5.2}) the system (writing $\e=\e_1\circ f_1$ and $\f=\f_1\circ f_1$ for simplicity) 
	 	\begin{align}\label{moving_frame}
	 	\left\{\begin{alignedat}{1}
	 		-\Delta\e&=|\D\e|^2\e+\left(2\s{\D\e}{\D\n}+\s{\e}{\Delta\n}\right)\n\\
	 		-\Delta\f&=|\D\f|^2\f+\left(2\s{\D\f}{\D\n}+\s{\f}{\Delta\n}\right)\n.
	 		\end{alignedat}\right.
	 	\end{align}
	 	where $\n:\mathbb{D}\rightarrow S^2$ is the same map as $f_1$ but viewed as the Gauss map associated to the branched minimal immersion of the disk from $\mathbb{D}$ into $\R^3$ with Weierstrass data $(f,dz)$. 
	 	It is given by
	 	\begin{align*}
	 		\n(z)=\left(\frac{2\,\Re(f(z))}{1+|f(z)|^2},\frac{2\,\Im(f(z))}{1+|f(z)|^2},\frac{-1+|f(z)|^2}{1+|f(z)|^2}\right).
	 	\end{align*}
	 	By a direct computation, we see that the Gauss map satisfies the following equations 
        \begin{align*}
	 		|\D \n(z)|^2&=\frac{8|f'(z)|^2}{(1+|f(z)|^2)^2}\\
	 		-\Delta\n&=|\D\n|^2\n.
	 	\end{align*}
 In particular, the previous equation \eqref{moving_frame} must reduce to
	 	\begin{align}\label{moving_frame_2}
	 		\left\{\begin{alignedat}{1}
	 			-\Delta\e&=|\D\e|^2\e+2\s{\D\e}{\D\n}\n\\
	 			-\Delta\f&=|\D\f|^2\f+2\s{\D\f}{\D\n}\n.
	 		\end{alignedat}\right.
	 	\end{align}
	 	However, since $\s{\e}{\Delta\n}=|\D\n|^2\s{\e}{\n}=0$, we deduce that $-\s{\Delta\e}{\n}=2\s{\D\e}{\D\n}$, and since $|\D\e|^2=1$, we also get (by taking the Laplacian of $|\e|^2=1$) that $-\s{\Delta\e}{\e}=|\D\e|^2\e$. Therefore, we need only check that 
	 	\begin{align}\label{null_condition}
	 		\s{\Delta\e}{\f\,}=\s{\Delta\f}{\e\,}=0
	 	\end{align}
	 	to show that $\e$ and $\f$ satisfy the equations \eqref{moving_frame_2}.
	 {Recall from \eqref{defchipsi} that }	$\e=\Re(\varphi)$ and $\f=\Im(\varphi)$, we deduce that 
	 	\begin{align*}
	 		\Delta\e=\Re(\Delta\varphi),\qquad \Delta\f=\Im(\Delta\varphi) 
	 	\end{align*}
	 	and we have
	 	\begin{align}\label{step0}
	 		\left\{\begin{alignedat}{1}
	 			\s{\Re(\Delta\varphi)}{\Im(\varphi)}&=\frac{1}{2}\Im\left(\s{\Delta\varphi}{\varphi}\right)-\frac{1}{2}\Im(\s{\Delta\varphi}{\bar{\varphi}})\\
	 			\s{\Im(\Delta\varphi)}{\Re(\varphi)}&=\frac{1}{2}\Im\left(\s{\Delta\varphi}{\varphi}\right)+\frac{1}{2}\Im(\s{\Delta\varphi}{\bar{\varphi}}).
	 		\end{alignedat}\right.
	 	\end{align}
	 	Therefore, the equations \eqref{null_condition} are equivalent to
	 	\begin{align}\label{target}
	 		\Im(\s{\Delta\varphi
	 		}{\varphi})=\Im(\s{\Delta\varphi}{\bar{\varphi}})=0.
	 	\end{align}
	 	Using \eqref{idpsi0}, \eqref{idpsi1} and \eqref{idpsi}, we get
	 	\begin{align}\label{np1}
	 	&\s{\Delta \varphi}{\varphi}
	 	=-4\s{\p{z}\varphi}{\p{\z}\varphi}\\
	 	&=-4\bs{-\frac{1}{2}\left(\frac{f''(z)}{f'(z)}+\frac{1}{z}\right)\varphi-\frac{f'(z)\bar{f(z)}}{1+|f(z)|^2}\left(-f(z),i\,f(z),1\right)}{\left(\frac{1}{2}\bar{\left(\frac{f''(z)}{f'(z)}+\frac{1}{z}\right)}-\frac{\bar{f'(z)}f(z)}{1+|f(z)|^2}\right)\varphi}=0,\nonumber
	 	\end{align}
	 	which implies in particular that $\Im\left(\s{\Delta\varphi}{\varphi}\right)=0$.
	 	Then, we compute
	 	\begin{align}\label{laststep}
	 	\s{\p{\z}\varphi}{\bar{\varphi}}=\bs{\left(\frac{1}{2}\bar{\left(\frac{f''(z)}{f'(z)}+\frac{1}{z}\right)}-\frac{\bar{f'(z)}f(z)}{1+|f(z)|^2}\right)\varphi}{\bar{\varphi}}=\left(\bar{\left(\frac{f''(z)}{f'(z)}+\frac{1}{z}\right)}-2\frac{\bar{f'(z)}f(z)}{1+|f(z)|^2}\right).
	 	\end{align}
	 	Therefore, we have
	 	\begin{align*}
	 	\frac{1}{4}\s{\Delta\varphi}{\bar{\varphi}}=\p{z}\s{\p{\z}\varphi}{\bar{\varphi}}-\s{\p{\z}\varphi}{\p{z}\bar{\varphi}}=\p{z}\s{\p{\z}\varphi}{\bar{\varphi}}-|\p{\z}\varphi|^2,
	 	\end{align*}
	 	where we used $\p{z}\bar{\varphi}=\bar{\p{\z}\varphi}$. By \eqref{laststep}, we deduce that
	 	\begin{align*}
	 	\p{z}\s{\p{\z}\varphi}{\bar{\varphi}}=-2\frac{|f'(z)|^2}{1+|f(z)|^2}+2\frac{|f'(z)|^2|f(z)|^2}{(1+|f(z)|^2)^2}=-\frac{2|f'(z)|^2}{(1+|f(z)|^2)^2},
	 	\end{align*}
	 	so that
	 	\begin{align*}
	 	\s{\Delta\varphi}{\bar{\varphi}}=-\frac{8|f'(z)|^2}{(1+|f(z)|^2)^2}-4|\p{\z}\varphi|^2\in \R,
	 	\end{align*}
	 	which implies that
	 	\begin{align}\label{step2}
	 		\Im\left(\s{\Delta\varphi}{\bar{\varphi}}\right)=0.
	 	\end{align}
	 	Therefore, we deduce that \eqref{target} holds, which implies that $\e$ and $\f$ solve the equations \eqref{moving_frame_2}. 
	 	
	 	\textbf{Step 4.} 
	 {Proof} of the decomposition $\omega_1=\ast\,d(G_{\Omega_1}+\mu_1)$. 
	 	
	 	Recall that $\e=\e_1$ and $\f=\f_1$, and let 
	 	\begin{align*}
	 	\omega=\s{\e}{d\f\,}=\s{\e}{\partial\f\,}+\s{\e}{\bar{\partial}\f\,}.
	 	\end{align*}
	 	Recall that since $\ast\, dx=dy$ and $\ast\, dy=-dx$, we have
	 	\begin{align*}
	    \ast\, dz&=\ast\left(dx+i\,dy\right)=dy-i\,dx=-i(dx+i\,dy)=-i\,dz\\
	    \ast\, d\z&=i\,d\z.
	 	\end{align*}
	 	Therefore, $\omega=\ast\,d\left(\mu+G\right)$ (where we write for simplicity $G=G_{\mathbb{D}}=\log|\,\cdot\,|$) if and only if
	 	\begin{align*}
	 	\s{\e}{\partial\f\,}+\s{\e}{\bar{\partial}\f\,}=\ast\left(\partial\left(\mu+G\right)+\bar{\partial}\left(\mu+G\right)\right)=-i\,\partial(\mu+G)+i\,\bar{\partial}(\mu+G),
	 	\end{align*}
	 	which is equivalent to the identity
	 	\begin{align}\label{cartan2}
	 	\s{\e}{\partial\f\,}=-i\,\partial\left(\mu+G\right).
	 	\end{align}
	 	We have by \eqref{idchi} and \eqref{idpsi0}
	 	\begin{align*}
	 	&\partial_{z}\f=\partial_{z}\Re(\varphi)=\frac{1}{2}\left(\p{z}\varphi+\bar{\p{\z}\varphi}\right)
	 	=\frac{1}{2}\left(-\frac{1}{2}\left(\frac{f''(z)}{f'(z)}+\frac{1}{z}\right)\varphi-\frac{f'(z)\bar{f(z)}}{1+|f(z)|^2}\varphi\right.\\
        &\left.+\frac{2f'(z)}{1+|f(z)|^2}\chi(z)(-f(z),i\,f(z),1)
	 	+\frac{1}{2}{\left(\frac{f''(z)}{f'(z)}+\frac{1}{z}\right)}\bar{\varphi}-\frac{f'(z)\bar{f(z)}}{1+|f(z)|^2}\bar{\varphi}\right)
	 	=-\frac{i}{2}\left(\frac{f''(z)}{f'(z)}-\frac{2f'(z)\bar{f(z)}}{1+|f(z)|^2}\right).
	 	\end{align*}
	 	Therefore, using \eqref{idpsi1}, \eqref{idpsi2}, \eqref{idpsi}, and $\s{\e}{\f}=0$, we deduce that
	 	\begin{align*}
	 	\s{\e}{\p{z}\f\,}&=\s{\Im(\varphi)}{\p{z}\Re(\varphi)}=-\frac{i}{2}\left(\frac{f''(z)}{f'(z)}+\frac{1}{z}\right)+\frac{i}{2}\frac{f'(z)}{1+|f(z)|^2}\s{(-f(z),i\,f(z),1)}{\bar{\psi}}\\
	 	&=-\frac{i}{2}\left(\frac{f''(z)}{f'(z)}-\frac{2f'(z)\bar{f(z)}}{1+|f(z)|^2}+\frac{1}{z}\right),
	 	\end{align*}
	 	and this concludes the proof of \eqref{cartan2} since by \eqref{defmu}
	 	\begin{align*}
	 	\partial_{z}(\mu(z)+\log|z|)&=\p{z}\left(\log|f'(z)|-\log(1+|f(z)|^2)-\frac{1}{2}\log(2)+\log|z|\right)\\
	 	&=\frac{1}{2}\left(\frac{f''(z)}{f'(z)}-\frac{2f'(z)\bar{f(z)}}{1+|f(z)|^2}+\frac{1}{z}\right).  
	 	\end{align*}
        This last identity concludes the proof of the theorem.
	 \end{proof}

	 Finally, we will establish the uniqueness of distributional solutions of the system \eqref{key2} with appropriate boundary conditions {\eqref{systemu}}. This is the exact analogous of Remark I.$1$ of \cite{BBH}. First, we need to define explicit maps that yield trivialisations of vector fields on simply connected domains of the sphere. Let $\Omega_1\subset S^2$ be as Theorem \ref{key2}.
  Using the stereographic projection $\pi : S^2\setminus\ens{N} \to \C$, we have one holomorphic chart $z$ on $S^2\setminus\ens{N}$, and for a domain $\Omega_1\subset S^2\setminus\ens{N}$, it yields a trivialisation 
  $T \Omega_1 \rightarrow \Omega\times \C$
	 	where $\Omega = \pi (\Omega_1)$.

	 	More explicitly, let $X:\Omega\rightarrow \R^3$ {be a vector field} such that $\s{X}{\n}=0$, where $\n 
	 	:\Omega\rightarrow S^2$ is the unit normal {given by}
	 	\begin{align*}
	 		\pi^{-1}(z)=\n(z)=\left(\frac{2\,\Re(z)}{1+|z|^2},\frac{2\,\Im(z)}{1+|z|^2},\frac{-1+|z|^2}{1+|z|^2}\right).
	 	\end{align*}
	 	Now, we introduce the function $\psi:\C\rightarrow \C^3$, given by 
	 	\begin{align*}
	 		\psi(z)=\left(\frac{1-z^2}{1+|z|^2},\frac{i(1+z^2)}{1+|z|^2},\frac{2z}{1+|z|^2}\right),
	 	\end{align*}
	 	and we easily check that 
	 	\begin{align}\label{ntrivial0}
	 		\s{\psi}{\psi}=0\qquad |\psi|^2=\s{\psi}{\bar{\psi}}=2.
	 	\end{align}
	 	Therefore, we deduce that $(\e_1,\e_2)$ defined as follows is a tangent unit moving frame (orthogonal to $\n$
	 	) 
	 	\begin{align*}
	 		\e_1(z)=\Re(\psi(z)),\quad \e_2(z)=\Im(\psi(z)).
	 	\end{align*}
	 	The trivialisation map on $\Omega_1\subset S^2\setminus\ens{N}$ is then given by 
	 	\begin{align}\label{trivialisation}
	 		T\Omega_{1} &\rightarrow \Omega\times \C\nonumber\\
	 		(z,v)&\mapsto (z,\s{v}{\e_1(z)}+i\,\s{v}{\e_2(z)}),
	 	\end{align}
	 	while the trivialisation map of sections is given by 
	 	\begin{align}\label{trivialisation_sections}
	 		\Psi_{\Omega_1}:\Gamma(T\Omega_1)&\rightarrow C^{\infty}(\Omega_1,\C)\nonumber\\
	 		X&\mapsto \s{X}{\e_1}+i\,\s{X}{\e_2}.
	 	\end{align}
	 	Notice that for all tangent vector field $X$, we have  $\s{X}{\n}=0$, which implies that there exists real functions $\lambda_1,\lambda_2 :\Omega_1\rightarrow \R$ such that 
	 	\begin{align*}
	 		X=\lambda_1\e_1+\lambda_2\e_2.
	 	\end{align*}
	 	\begin{rem}
	 	    	Using the next Theorem \ref{5.2}, it is easy to check that $(\e_j,\f_j)$ ($j=1,2$) are harmonic vector fields since by \eqref{defchipsi} and \eqref{idchi}, we have
	 \begin{align*}
	 -\Delta\chi=\left|\frac{f''(z)}{f'(z)}+\frac{1}{z}\right|^2\chi=|\D\chi|^2\chi,
	 \end{align*}
	 \emph{i.e.} $\chi:\mathbb{D}\rightarrow S^1$ is a harmonic map with values into $S^1$.
	 	\end{rem}
	 
	 \begin{theorem}\label{5.2}
	 	Under the conditions of Theorem \ref{key2}, let $\Omega_1^{\ast}=\Omega_1\setminus\ens{p_1}$ and $\vec{u}\in W^{1,2}_{\mathrm{loc}}(\Omega_1^{\ast},U\Omega_1)\cap W^{1,1}(\Omega_1,U\Omega_1)$ be a unit vector field in $\Omega_1$, and let $\vec{u}_0=\Psi_{\Omega_1}(\vec{u}):\Omega_1\rightarrow S^1$. Then $\vec{u}$ is a harmonic vector field on $\Omega_1$, \emph{i.e.} 
	 	\begin{align*}
	 		-\Delta_{g_0}\vec{u}=|d\vec{u}|_{g_0}^2{\vec{u}}+\left(2\s{d\vec{u}}{d\n}_{g_0}+\s{\vec{u}}{\Delta_{g_0}\n}\right)\n
	 	\end{align*}
	 	if and only if $\vec{u}_0$ is a harmonic map with values into $S^1$, \emph{i.e.}
	 	\begin{align*}
	 		-\Delta_{g_0}\vec{u}_0=|d\vec{u}_0|^2_{g_0}\vec{u}_0.
	 	\end{align*}
	 	In particular, for all degree $1$ boundary data $h\in H^{1/2}(\partial\Omega_1, S^1)$ and $p\in \Omega_1$, there exists a unique unit vector-field $\vec{u}\in W^{1,2}_{\mathrm{loc}}(\Omega_1^{\ast},U\Omega_1)\cap W^{1,1}(\Omega_1,U\Omega_1)$ such that $\vec{u}=\Psi_{\Omega_1}^{-1}(h)$ on $\partial \Omega_1$ 
   and such that $\vec{u}_0=\Psi_{\Omega_1}(\vec{u})$ satisfies in the distributional sense
	 	\begin{align*}
	 	\dive\left(\vec{u}_0\times \D\vec{u}_0\right)=0\qquad\text{in}\;\,\mathscr{D}'(\Omega_1).
	 	\end{align*}	
	 \end{theorem}
	 \begin{rem}
	 If $\vec{u}_0:\Omega \rightarrow S^1$, writing locally $\vec{u}_0=e^{i\varphi}$ for some real-valued function $\varphi$, we deduce that $\vec{u}_0$ is harmonic if and only if
	 \begin{align*}
	 -\Delta \vec{u}_0=
	 \left(|\D\varphi|^2-i\,(\Delta\varphi)\right)\vec{u}_0=|\D\vec{u}_0|^2\vec{u}_0.
	 \end{align*}
	 Therefore, $\vec{u}_0$ is harmonic as a map with values into $S^1$ if and only if $\varphi$ is harmonic, \emph{i.e.} $\Delta\varphi=0$. 
	 \end{rem}
	 \begin{proof}[{Proof of Theorem~\ref{5.2}}]
	 	By making a stereographic projection, thanks to the conformal invariance of the harmonic equation, we deduce that for all unit vector-field $\vec{u}\in \Gamma(T\Omega_1^{\ast})$ is given in $\Omega=\pi_N(\Omega_1)$ as
	 	\begin{align}\label{expansion0}
	 		\vec{u}=\lambda_1\,\Re(\psi)+\lambda_2\,\Im(\psi),
	 	\end{align}
	 	where 
	 	\begin{align*}
	 		\psi(z)=\left(\frac{1-z^2}{1+|z|^2},\frac{i(1+z^2)}{1+|z|^2},\frac{2z}{1+|z|^2}\right).
	 	\end{align*}
	 	Furthermore, we have $\lambda_1^2+\lambda_2^2=1$, which implies that there exists a measurable function $\varphi$ such that $\lambda_1+i\lambda_2=e^{-i\varphi}$. In particular, we can rewrite \eqref{expansion0} as
	 	\begin{align*}
	 		\vec{u}=\cos(\varphi)\,\Re(\psi)-\sin(\varphi)\,\Im(\psi)=\Re(e^{-i\varphi})\,\Re(\psi)+\Im(e^{-i\varphi})\,\Im(\psi)=\Re\left(e^{i\varphi}\psi\right),
	 	\end{align*}
	 	where we used the identity
	 	$
	 		\Re(a)\,\Re(b)+\Im(a)\,\Im(b)=\Re(\bar{a}\,b)
        $
	 	valid for all $a,b\in\C$. If 
	 	\begin{align*}
	 		\vec{v}=\sin(\varphi)\Re(\psi)+\cos(\varphi)\Im(\psi)=\Im(e^{i\varphi}\psi),
	 	\end{align*}
	 	we immediately have $\s{\vec{u}}{\vec{v}\,}=0$, and since $|\vec{u}\,|^2=|\vec{v}\,|^2 = 1$, while $-\Delta\n=|\D\n|^2\n$, we get
	 	\begin{align*}
	 		&\s{\Delta\vec{u}}{\vec{u}\,}=-|\D\vec{u}\,|^2\\
	 		&\s{\Delta\vec{u}}{\vec{n}\,}=-2\s{\D\vec{u}}{\D\n\,}-\s{\vec{u}}{\Delta\n\,}=-2\s{\D\vec{u}}{\D\n\,},
	 	\end{align*}
	 	and similar formulae for $\vec{v}$. Therefore, we deduce that $(\vec{u},\vec{v})$ solves the system
	 	\begin{align*}
	 		\left\{\begin{alignedat}{2}
	 			-\Delta\vec{u}&=|\D\vec{u}\,|^2\vec{u}+2\s{\D\vec{u}}{\D\n\,}\n\qquad&&\text{in}\;\,\Omega\\
	 			-\Delta\vec{v}&=|\D\vec{v}\,|^2\vec{v}+2\s{\D\vec{v}}{\D\n\,}\n\qquad&&\text{in}\;\,\Omega,
	 		\end{alignedat}\right.
	 	\end{align*}
	 	if and only if
	 	\begin{align*}
	 		\s{\Delta\vec{u}}{\vec{v}\,}=\s{\Delta\vec{v}}{\vec{u}\,}=0.
	 	\end{align*}
	 	Now, we compute
	 	\begin{align*}
	 		\Delta\vec{u}&=\Re\left((i\,\Delta\varphi-|\D \varphi|^2)e^{i\varphi}\psi\right)+2\,\Re\left(ie^{i\varphi}\D\varphi\cdot\D\psi\right)+\Re\left(e^{i\varphi}\Delta\psi\right)\\
	 		&=-(\Delta\varphi)\Im(e^{i\varphi}\psi)-|\D\varphi|^2\Re(e^{i\varphi}\psi)+\Re\left(e^{i\varphi}\Delta\psi\right)\\
	 		&=-(\Delta\varphi)\,\vec{v}-|\D\varphi|^2\vec{u}+2\,\Re\left(ie^{i\varphi}\D\varphi\cdot\D\psi\right)+\Re(e^{i\varphi}\Delta\psi)\\
	 		\Delta\vec{v}&=(\Delta\varphi)\,\vec{u}-|\D\varphi|^2\vec{u}+2\,\Im\left(ie^{i\varphi}\D\varphi\cdot\D\psi\right)+\Im(e^{i\varphi}\Delta\psi).
	 	\end{align*}
	 	We
	 	have since $\s{\D\varphi}{\varphi}=0$ the identity
	 	\begin{align*}
	 		\bs{\Re(ie^{i\varphi}\D\varphi\cdot\D\psi)}{\vec{v}}&=\Re\bs{ie^{i\varphi}\D\varphi\cdot\D\psi}{\frac{e^{i\varphi}\psi-e^{-i\varphi}\bar{\psi}}{2i}}=-\frac{1}{2}\Re\left(\D\varphi\cdot\s{\D\psi}{\bar{\psi}}\right)\\
	 		\bs{\Im(ie^{i\varphi}\D\varphi\cdot \D\psi)}{\vec{u}}&=\Im\bs{ie^{i\varphi}\D\varphi\cdot\D\psi}{\frac{e^{i\varphi}\psi+e^{-i\varphi}\bar{\psi}}{2}}=\frac{1}{2}\Re\left(\D\varphi\cdot\s{\D\psi}{\bar{\psi}}\right)\\
	 		\s{\Re(e^{i\varphi}\Delta\psi)}{\vec{v}}&=\Re\bs{e^{i\varphi}\Delta\psi}{\frac{e^{i\varphi}\psi-e^{-i\varphi}\psi}{2i}}=\frac{1}{2}\Im\left(e^{2i\varphi}\s{\Delta\psi}{\psi}\right)-\frac{1}{2}\Im\s{\Delta\psi}{\psi}\\
	 		\s{\Im(e^{i\varphi}\Delta\psi)}{\vec{u}}&=\frac{1}{2}\Im\left(e^{2i\varphi}\s{\Delta\psi}{\psi}\right)+\frac{1}{2}\Im\s{\Delta\psi}{\bar{\psi}}.
	 	\end{align*}
	 	In particular, we have
	 	\begin{align*}
	 		\s{\Delta\vec{u}}{\vec{v}}=-(\Delta\varphi)-\Re\left(\D\varphi\cdot\s{\D\psi}{\bar{\psi}}\right)+\frac{1}{2}\Im\left(e^{2i\varphi}\s{\Delta\psi}{\psi}\right)-\frac{1}{2}\Im\s{\Delta\psi}{\psi}\\
	 		\s{\Delta\vec{v}}{\vec{u}}=(\Delta\varphi)+\Re\left(\D\varphi\cdot\s{\D\psi}{\bar{\psi}}\right)+\frac{1}{2}\Im\left(e^{2i\varphi}\s{\Delta\psi}{\psi}\right)+\frac{1}{2}\Im\s{\Delta\psi}{\bar{\psi}}.
	 	\end{align*}
	 	Summing those equations and substracting the first one to the second one yields the system
	 	\begin{align}\label{initial_system}
	 		\left\{\begin{alignedat}{1}
	 			&\Im\left(e^{2i\varphi}\s{\Delta\psi}{\psi}\right)=0\\
	 			&2(\Delta\varphi)+2\,\Re\left(\D\varphi\cdot\s{\D\psi}{\bar{\psi}}\right)+\Im\s{\Delta\psi}{\bar{\psi}}=0.
	 		\end{alignedat} \right.
	 	\end{align}
	 	We will show that for all smooth \emph{real-valued} function $\varphi:\Omega\rightarrow\R$
	 	\begin{align}\label{three_diagram}
	 		&\Re\left(\D\varphi \cdot \s{\D\psi}{\bar{\psi}}\right)=\s{\Delta\psi}{\psi}=\Im\s{\Delta\psi}{\bar{\psi}}=0,
	 	\end{align}
	 	which will imply that $(\vec{u},\vec{v})$ solves the system \eqref{initial_system} if and only if $\Delta\varphi=0$, or $\varphi$ is harmonic.

	 	Now, we compute
	 	\begin{align*}
	 		\p{z}\psi&=-\frac{\z}{1+|z|^2}\psi+\frac{2}{1+|z|^2}\left(-z,i\,z,1\right)\\
	 		\p{\z}\psi&=-\frac{z}{1+|z|^2}\psi
	 	\end{align*}
	 	We have
	 	\begin{align*}
	 		\D\varphi\cdot \D\psi&=2\,\p{z}\varphi\cdot \p{\z}\psi+2\,\p{\z}\varphi\cdot\p{z}\psi\\
	 		&=-2\frac{\z\p{z}\varphi}{1+|z|^2}\psi+\frac{4}{1+|z|^2}(-z\p{z}\varphi,iz\p{z}\varphi,\p{z}\varphi)
	 		-2\frac{z\p{\z}\varphi}{1+|z|^2}\psi\\
	 		&=-4\,\Re\left(\frac{\z\p{z}\varphi}{1+|z|^2}\right)\psi+\frac{4}{1+|z|^2}(-z\p{z}\varphi,iz\p{z}\varphi,\p{z}\varphi).
	 	\end{align*}
	 	Then we have
	 	\begin{align}\label{null_condition1}
	 		\frac{1}{4}\Delta\psi=\p{z\z}^2\psi=\frac{-1+|z|^2}{(1+|z|^2)^2}\psi-\frac{2}{(1+|z|^2)^2}(-z^2,iz^2,z)
	 	\end{align}
	 	Now, notice that 
	 	\begin{align*}
	 		\s{(-z^2,iz^2,z)}{\psi}&=\frac{1}{1+|z|^2}\s{(-z^2,iz^2,z)}{\left(1-z^2,i(1+z^2),2z\right)}\\
	 		&=\frac{1}{1+|z|^2}\left(-z^2(1-\z^2)+z^2(1+\z^2)+2z^2\right)=0,
	 	\end{align*}
	 	which implies as $\s{\psi}{\psi}=0$ and by \eqref{null_condition1} that 
	 	\begin{align}\label{eq_part1}
	 		\s{\Delta\psi}{\psi}=0.
	 	\end{align}
	 	Now, we have
	 	\begin{align*}
	 		\s{(-z^2,iz^2,z)}{\bar{\psi}}=\frac{1}{1+|z|^2}\left(-z^2(1-\z^2)+z^2(1+\z^2)+2|z|^2\right)=2|z|^2.
	 	\end{align*}
	 	Since $|\psi|^2=\s{\psi}{\bar{\psi}}=2$, we deduce that 
	 	\begin{align}\label{eq_part2}
	 		\s{\Delta\psi}{\bar{\psi}}=4\left(\frac{2(-1+|z|^2)}{(1+|z|^2)^2}-\frac{4|z|^2}{(1+|z|^2)^2}\right)=-\frac{8}{1+|z|^2}\in\R.
	 	\end{align}
	 	We now compute
	 	\begin{align*}
	 		\s{(-z,iz,1)}{\bar{\psi}}=\frac{1}{1+|z|^2}\left(-z(1-\z^2)+z(1+\z^2)+2\z\right)=2\z.
	 	\end{align*}
	 	which shows since $|\psi|^2=2$ that
	 	\begin{align*}
	 		\s{\p{z}\psi}{\psi}&=-\frac{2\z}{1+|z|^2}+\frac{4\z}{1+|z|^2}=\frac{2\z}{1+|z|^2}\\
	 		\s{\p{\z}\psi}{\psi}&=-\frac{2z}{1+|z|^2}.
	 	\end{align*}
	 	Therefore, we have
	 	\begin{align*}
	 		\D\varphi\cdot\s{\D\psi}{\psi}=2\,\p{\z}\varphi\cdot\s{\p{z}\psi}{\psi}+2\,\p{z}\varphi\cdot\s{\p{\z}\psi}{\psi}=\frac{4\,\z\,\p{\z}\varphi}{1+|z|^2}-\frac{4\,z\,\p{z}\varphi}{1+|z|^2}=8i\,\Im\left(\frac{\z\,\p{\z}\varphi}{1+|z|^2}\right)\in i\R,
	 	\end{align*}
	 	and this immediately implies that 
	 	\begin{align}\label{eq_part3}
	 		\Re\left(\D\varphi\cdot \s{\D\psi}{\psi}\right)=0.
	 	\end{align}
	 	Finally, we deduce by \eqref{eq_part1}, \eqref{eq_part2} and \eqref{eq_part3} that \eqref{three_diagram} holds and that the system \eqref{initial_system} holds if and only if $\Delta\varphi=0$. If $\vec{u}=g_{\Omega_1}=\Psi_{\Omega_1}(g)$ for some $g:\partial\Omega_1=\Gamma\rightarrow S^1$, then we have
	 	\begin{align*}
	 		\lambda_1+i\lambda_2=g,
	 	\end{align*}
	 	or
	 	\begin{align*}
	 		e^{-i\varphi}=g\qquad\text{on}\;\, \Gamma. 
	 	\end{align*}
	 	In particular, the function $\vec{u}_0=e^{-i\varphi}:\Omega_1\setminus\ens{p_1}\rightarrow S^1$ is a harmonic map on $\Omega_1\setminus\ens{p_1}$ satisfying $\vec{u}_0=h$ on $\Gamma$. Now, notice that provided $\vec{u}\in W^{1,1}(\Omega_1)$, one can rewrite the equation distributionally as
	 	\begin{align*}
	 		\dive\left(\vec{u}\times\D\vec{u}\right)=2\s{\D\vec{u}}{\D\n}\n\times \vec{u}.
	 	\end{align*}
	 	In particular, we deduce as $u_0$ is harmonic that 
	 	\begin{align*}
	 		\dive\left(\vec{u}_0\times \D \vec{u}_0\right)=\frac{\partial}{\partial x_1}\left(\vec{u}_0\times \frac{\partial \vec{u}_0}{\partial x_1}\right)+\frac{\partial}{\partial x_2}\left(\vec{u}_0\times \frac{\partial \vec{u}_0}{\partial x_2}\right)=0.
	 	\end{align*}
	 	By Theorem I.$5$ and Remark I.$1$ of \cite{BBH}, we deduce that $\vec{u}_0$ is the unique harmonic function with a singularity at $p_1$ such that $\vec{u}_0=h$ on $\partial \Omega_1$. This concludes the proof of the theorem. 
	 \end{proof}
    
	 \section{Proof of the Main Theorems for Non-Smooth Curves}

	 In order to extend Theorem \ref{wp10} to the non-smooth setting, 
	 we will obtain another formula for $\mathscr{E}_0$ in terms of conformal maps and that holds true for any closed simple curve of finite Loewner energy. Using this additional formula, the convergence result will be easily obtained.  
	 
	 Under the preceding notations, if $\Gamma\subset S^2$ {Weil-Petersson quasicircle}, 
	  {from Remark~\ref{rem:existence_frame},} thanks to Theorem \ref{key2}, there exists {harmonic} moving frames $(\e_1,\f_1)$ and $(\e_2,\f_2)$ on $\Omega_1$ and $\Omega_2$ with arbitrary
   singularities $p_1$ and $p_2$ respectively, such that
	 \begin{align*}
	 	\mathscr{E}(\Gamma)=\int_{\Omega_1}|d\mu_1|^2_{g_0}d\mathrm{vol}_{g_0}+\int_{\Omega_2}|d\mu_2|^2_{g_0}d\mathrm{vol}_{g_0}+2\int_{\Omega_1}G_{\Omega_1}K_{g_0}d\mathrm{vol}_{g_0}+2\int_{\Omega_2}G_{\Omega_2}K_{g_0}d\mathrm{vol}_{g_0}+4\pi.
	 \end{align*}
	 where $\omega_j=\s{\e_j}{d\f_j}=\ast\,d(G_{\Omega_j}+\mu_j)$ in $\mathscr{D}'(\Omega_j)$ for $j=1,2$,  and $\mu_j$ satisfies \eqref{boundary}. We saw in Theorem \ref{wp10} that in the case of smooth curves, there exists conformal maps $f_1:\mathbb{D}\rightarrow \Omega_1$ and $f_2:\mathbb{D}\rightarrow \Omega_2$ such that
  \begin{align*}
      I^L(\Gamma)=\frac{1}{\pi}\mathscr{E}(\Gamma)+4\log|\D f_1(0)|+4\log|\D f_2(0)|-12\log(2)=\frac{1}{\pi}\mathscr{E}_0(\Gamma).
  \end{align*}
  In this section, we generalise this result for curves of finite Loewner energy.
	 Now, if $\pi:S^2\setminus\ens{p_2}\rightarrow \C$ is 
	 {a}
	 stereographic projection, since $f_j:\mathbb{D}\rightarrow \Omega_j$ is conformal and $\pi$ is also conformal, we deduce that 
	 $\pi\circ f_j:\mathbb{D}\rightarrow \pi(\Omega_j)\subset \R$ is also conformal. Therefore, these maps are {bi}holomorphic or anti-{bi}holomorphic, so up to a complex conjugate (which is an isometry), we can assume that 
	 {they} are holomorphic. 
	 Notice that $\Omega=\pi(\Omega_1)$ is bounded, while $\pi(\Omega_2)=\C\setminus\bar{\Omega}$ is unbounded. Therefore, if ${\ii}:\C\setminus\ens{0}\rightarrow \C\setminus\ens{0}$ is the inversion, {we let}
	 $g=\pi\circ f_{2} \circ {\ii}:\C\setminus\bar{\mathbb{D}}\rightarrow \C\setminus\bar{\Omega}$ 
	 {and} $f ={\pi \circ f_1}:\mathbb{D}\rightarrow \Omega$. 
	 {From \eqref{defs1} and \eqref{universal}}, if $\gamma=\pi(\Gamma)$, we have
	 \begin{align}\label{loewner1}
	 	I^L(\Gamma)=I^L(\gamma)=\int_{\mathbb{D}}\left|\frac{f''(z)}{f'(z)}\right|^2|dz|^2+\int_{\mathbb{C}\setminus\bar{\mathbb{D}}}\left|\frac{g''(z)}{g'(z)}\right|^2|dz|^2+4\pi\log|f'(0)|-4\pi \log|g'(\infty)|.
	 \end{align}
	 Indeed, since $f_2(0)=p_2$, we have $g(\infty)=\infty$, so that the functions $f$, $g$ satisfy the needed conditions {to apply Theorem~\ref{thm:equi_WP}.}

	 Now, with the previous notations, define the functional
	 \begin{align*}
	 	\mathscr{E}_0({\Gamma})=\mathscr{E}({\Gamma})+4\pi\log|\D f_1(0)|+4\pi\log|\D f_2(0)|-12\pi\log(2).
	 \end{align*}
	 
	 \begin{defi}\label{def:s3} 
	 Let $\gamma$ be a Jordan curve with finite Loewner energy. Let
	 $f:\mathbb{D}\rightarrow \Omega$, $g:\C\setminus\bar{\mathbb{D}}\rightarrow \C\setminus\bar{\Omega}$ be biholomorphic maps such that $g(\infty)=\infty$, we define the third universal Liouville action $S_3$ by 
	 	\begin{align*}
	 		S_3(\gamma)&=\int_{\mathbb{D}}\left|\frac{f''(z)}{f'(z)}-2\frac{f'(z)}{f(z)}\frac{|f(z)|^2}{1+|f(z)|^2}\right|^2|dz|^2+\int_{\C\setminus\bar{\mathbb{\mathbb{D}}}}\left|\frac{g''(z)}{g'(z)}-2\frac{g'(z)}{g(z)}\frac{|g(z)|^2}{1+|g(z)|^2}+\frac{2}{z}\right|^2|dz|^2\\
	 		&+2\int_{\mathbb{D}}\log|z|\frac{4|f'(z)|^2|dz|^2}{(1+|f(z)|^2)^2}-2\int_{\C\setminus\bar{\mathbb{D}}}\log|z|\frac{4|g'(z)|^2|dz|^2}{(1+|g(z)|^2)^2}+4\pi\\
	 		&+4\pi\log|f'(0)|-4\pi\log|g'(\infty)|-4\pi\log(1+|f(0)|^2).
	 	\end{align*}
	 \end{defi}

  \begin{rem}
      \begin{enumerate}
          \item[($1$)] One may wonder from where this definition comes from. It will be made clear in the proof of the next theorem where we explicitly rewrite $\mathscr{E}_0$ with the help of the conformal maps $f$ and $g$ defined above.
          \item[($2$)] We call this quantity $S_3$ since a functional called $S_2$ was defined in \cite{takteo} as
          the log-determinant of the Grunsky operator associated with the curve $\gamma$  (up to a factor $-\frac{1}{12}$).
      \end{enumerate}
  \end{rem}
	 
	 {The goal of this section is to show  the identity}
	 	\begin{equation}\label{eq:chain_eq}
	 	    {\pi} I^L = S_1=S_3=\mathscr{E}_0.
	 	\end{equation}
	 	
	 The third equality is straightforward {and is proved in Theorem~\ref{third_loewner}, and the proof of the whole identity is completed in Theorem~\ref{s3}. }  
	 \begin{theorem}\label{third_loewner}
	 	Let $\Gamma\subset S^2$ be a simple curve of finite Loewner energy. Then we have
	 	\begin{align*}
	 		\mathscr{E}_0(\Gamma)=S_3(\Gamma)
	 	\end{align*}
	 \end{theorem}
	 \begin{proof}
	 	If $\Gamma\subset S^2$ is a curve of finite Loewner energy $\Omega_1$ and $\Omega_2$ the two connected components of $S^2\setminus\Gamma$, and $f_1:\mathbb{D}\rightarrow \Omega_1$, $f_2:\mathbb{D}\rightarrow \Omega_2$ are the 
	 	conformal maps associated to $\Gamma$ in the definition of $\mathscr{E}$ {with $f_j (0) = p_j$ for $j= 1,2$}.
	 	Now, recall {from \eqref{eq:change_factor_mu}} that 
	 	\begin{align*}
	 		\log|\D f_1|=\frac{1}{2}\log(2)+\mu_1\circ f_1.
	 	\end{align*}
	 	We have by conformal invariance of the Dirichlet energy
	 	\begin{align}\label{n_loewner}
	 		\int_{\mathbb{D}}\left|\D\log|\D f_1|\right|^2|dz|^2=\int_{\mathbb{D}}\left|\D (\mu_1\circ f_1)\right|^2|dz|^2=\int_{\Omega_1}|d\mu_1|_{g_0}^2d\mathrm{vol}_{g_0}.
	 	\end{align}
Since $f_1$ is conformal and $f_1(0) = p_1$, we have 
\begin{align*}
G_{\Omega_1} \circ f_1(z) = G_{\Omega_1,p_1} \circ f_1 (z) = G_{\mathbb D,0} (z) = \log |z|.
\end{align*}
	 {A change of variable gives}
	 	\begin{align}\label{n_loewner2}
	 		2\int_{\Omega_1}G_{\Omega_1}d\mathrm{vol}_{g_0}=\int_{\mathbb{D}}\log|z||\D f_1|^2|dz|^2.
	 	\end{align}
	 	Finally, we deduce by \eqref{n_loewner} and \eqref{n_loewner2} that 
	 	\begin{align}\label{n_loewner3}
	 		\mathscr{E}({\Gamma})
	 		&=\int_{\Omega_1}|d\mu_1|^2_{g_0}d\mathrm{vol}_{g_0}+\int_{\Omega_2}|d\mu_2|^2_{g_0}d\mathrm{vol}_{g_0}+2\int_{\Omega_1}G_{\Omega_1}K_{g_0}d\mathrm{vol}_{g_0}+2\int_{\Omega_2}G_{\Omega_2}K_{g_0}d\mathrm{vol}_{g_0}+4\pi\nonumber\\
	 		&+4\pi\log|\D f_1(0)|+4\pi\log|\D f_2(0)|-12\pi\log(2)\nonumber\\
	 		&=\int_{\mathbb{D}}|\D\log|\D f_1||^2|dz|^2+\int_{\mathbb{D}}|\D\log|\D f_2||^2|dz|^2+\int_{\mathbb{D}}\log|z||\D f_1|^2|dz|^2+\int_{\mathbb{D}}\log|z||\D f_2|^2|dz|^2+4\pi\nonumber\\
	 		&+4\pi\log|\D f_1(0)|+4\pi\log|\D f_2(0)|-12\pi\log(2).
	 	\end{align}
	 	
	 	Up to a rotation of $S^2$, we can assume that $p_2=N$ and if $\pi:S^2\setminus\ens{N}\rightarrow \C$ is the standard stereographic projection, let 
	 	\begin{align*}
	 		f&=\pi\circ f_1:\mathbb{D}\rightarrow \Omega\\
	 		\tilde{g}&=\pi\circ f_2:\mathbb{D}\rightarrow \C\setminus\bar{\Omega}
	 	\end{align*}
	 {which we assume without loss of generality to be biholomorphic (up to a complex conjugation)}.
	 	Now, since
	 	\begin{align*}
	 		f_1(z)=\pi^{-1}(f(z))=\left(\frac{2\,\Re(f(z))}{1+|f(z)|^2},\frac{2\,\Im(f(z))}{1+|f(z)|^2},\frac{-1+|f(z)|^2}{1+|f(z)|^2}\right),
	 	\end{align*}
	 	a computation shows that
	 	\begin{align*}
	 		\p{z}f_1=f'\left(\frac{(1-\bar{f}^2)}{(1+|f|^2)^2},\frac{-i(1+\bar{f}^2)}{(1+|f|^2)^2},\frac{2\bar{f}}{(1+|f|^2)^2}\right),
	 	\end{align*}
	 	which implies that 
	 	\begin{align*}
	 		|\p{z}f_1|^2=\frac{|f'|^2}{(1+|f|^2)^4}\left(|1-f^2|^2+|1+f^2|^2+4|f|^2\right)=\frac{2|f'|^2}{(1+|f|^2)^2}.
	 	\end{align*}
	 	{W}e deduce that
	 	\begin{align}\label{twice_conf_factor}
	 		|\D f_1|^2=4|\p{z}f_1|^2=\frac{8|f'|^2}{(1+|f|^2)^2}.
	 	\end{align}
	 	Therefore, we have
	 	\begin{align*}
	 		\log|\D f_1|=\log|f'|-\log(1+|f|^2)+\frac{3}{2}\log(2),
	 	\end{align*}
	 	so that 
	 	\begin{align}\label{constantes1}
	 		&4\pi \log|\D f_1(0)|=4\pi\log|f'(0)|-4\pi\log(1+|f(0)|^2)+6\pi\log(2)
	 	\end{align}
	 	Since $\Omega=f(\mathbb{D})$ is compact,  
         we have
	 	\begin{align}\label{l2}
	 	\int_{\mathbb{D}}|f'(z)|^2\frac{|f(z)|^2|dz|^2}{(1+|f(z)|^2)^2}\leq \frac{1}{4}\int_{\mathbb{D}}|f'(z)|^2|dz|^2=\frac{1}{4}\mathrm{Area}(\Omega)<\infty.
        \end{align}
	 	Therefore, \eqref{l2} implies that $\D\log|\D f_1|\in  L^{2}(\mathbb{D})$ and 
	 	\begin{align}\label{n_loewner4}
	 		\int_{\mathbb{D}}|\D\log|\D f_1||^2|dz|^2=\int_{\mathbb{D}}\left|\frac{f''(z)}{f'(z)}-2\frac{f'(z)}{f(z)}\frac{|f(z)|^2}{1+|f(z)|^2}\right|^2|dz|^2<\infty,
	 	\end{align}
	 	while \eqref{twice_conf_factor} implies that 
	 	\begin{align}\label{n_loewner5}
	 		\int_{\mathbb{D}}\log|z||\D f_1|^2|dz|^2=\int_{\mathbb{D}}\log|z|\frac{8|f'(z)|^2|dz|^2}{(1+|f(z)|^2)^2}<\infty,
	 	\end{align}
        which is finite by \eqref{l2} and the smoothness of $f$ in $\mathbb{D}$.
	 	Since the function $\tilde{g}:\mathbb{D}\rightarrow \C\setminus\bar{\Omega}$ is unbounded at $0$, we do not see trivially that 
	 	\begin{align*}
	 	\int_{\mathbb{D}}\left|\frac{\tilde{g}''(z)}{\tilde{g}'(z)}-2\frac{\tilde{g}'(z)\bar{\tilde{g}(z)}}{1+|\tilde{g}(z)|^2}\right|^2|dz|^2=\int_{\mathbb{D}}\left|\frac{\tilde{g}''(z)}{\tilde{g}'(z)}-2\frac{\tilde{g}'(z)}{g(z)}\frac{|\tilde{g}(z)|^2}{1+|\tilde{g}(z)|^2}\right|^2|dz|^2<\infty.
	 	\end{align*}
	 	{For this,} as 
	 	$\tilde{g}$ is univalent and $\tilde{g}(0)=\infty$, we deduce that $\tilde{g}$ admits the following meromorphic expansion at $z=0$ for some $a\in \C\setminus\ens{0}$ and $a_0,a_1\in\C$ 
	 	\begin{align}\label{eq:g_tilde_expansion}
	 		\tilde{g}(z)=\frac{a}{z}+a_0+a_1z+O(|z|^2).
	 	\end{align}
	 	Therefore, we have by a direct computation
	 	\begin{align}\label{est01}
	 		\frac{\tilde{g}''(z)}{\tilde{g}'(z)}-2\frac{\tilde{g}'(z)}{\tilde{g}(z)}\frac{|\tilde{g}(z)|^2}{1+|\tilde{g}(z)|^2}
	 		&=-\frac{a_0}{a}+\left(\frac{a_0^2}{a^2}-\frac{4a_1}{a}\right)z-\frac{\z}{|a|^2}+O(|z|^2)\in L^{\infty}_{\mathrm{loc}}(\mathbb{D}).
	 	\end{align}
	 	Since $\Gamma$ is a Weil-Petersson quasicircle, we deduce by estimates similar to \eqref{l2} and \eqref{n_loewner4} that $\D\log|\tilde{g}'|\in L^2(\mathbb{D}\setminus\bar{\mathbb{D}}(0,\epsilon))$ and $\tilde g' \in L^2(\mathbb{D}\setminus\bar{\mathbb{D}}(0,\epsilon))$  for all $\epsilon>0$ and we finally deduce that 
	 	\begin{align*}
	 		\int_{\mathbb{D}}\left|\frac{\tilde{g}''(z)}{\tilde{g}'(z)}-2\frac{\tilde{g}'(z)}{\tilde{g}(z)}\frac{|\tilde{g}(z)|^2}{1+|\tilde{g}(z)|^2}\right|^2|dz|^2<\infty.
	 	\end{align*} 	
	 	Now, if ${g}=\tilde{g}\circ{\ii}:\C\setminus\bar{\mathbb{D}}\rightarrow \C\setminus\bar{\Omega}$, we compute
	 	and
	 	\begin{align*}
	 		\frac{\tilde{g}''(z)}{\tilde{g}'(z)}-2\frac{\tilde{g}'(z)}{\tilde{g}(z)}\frac{|\tilde{g}(z)|^2}{1+|\tilde{g}(z)|^2}
	 		&=-\frac{1}{z^2}\left(\frac{{g}''(1/z)}{{g}'(1/z)}-2\frac{{g}'(1/z)}{{g}(1/z)}\frac{|{g}(1/z)|^2}{1+|{g}(1/z)|^2}+2z\right).
	 	\end{align*}
	 	A change of variable shows that
	 	\begin{align}\label{n_loewner6}
	 		\int_{\mathbb{D}}\left|\frac{\tilde{g}''(z)}{\tilde{g}'(z)}-2\frac{\tilde{g}'(z)}{\tilde{g}(z)}\frac{|\tilde{g}(z)|^2}{1+|\tilde{g}(z)|^2}\right|^2|dz|^2=\int_{\C\setminus\bar{\mathbb{D}}}\left|\frac{{g}''(z)}{{g}'(z)}-2\frac{{g}'(z)}{{g}(z)}\frac{|{g}(z)|^2}{1+|{g}(z)|^2}+\frac{2}{z}\right|^2|dz|^2
	 	\end{align}
	 	Furthermore, we directly get
	 	\begin{align}\label{n_loewner7}
	 		\int_{\mathbb{D}}\log|z|\frac{|\tilde{g}'(z)|^2|dz|^2}{(1+|\tilde{g}(z)|^2)^2}&=\int_{\mathbb{D}}\log|z|\frac{|{g}'(1/z)|^2}{(1+|{g}(1/z)|^2)^2}\frac{|dz|^2}{|z|^4}
	 		=\int_{\mathbb{C}\setminus\bar{\mathbb{D}}}\log\left(\frac{1}{|z|}\right)\frac{|{g}'(z)|^2|dz|^2}{(1+|{g}(z)|^2)^2}\nonumber\\
	 		&=-\int_{\C\setminus\bar{\mathbb{D}}}\log|z|\frac{|g'(z)|^2|dz|^2}{(1+|g(z)|^2)^2}.
	 	\end{align}
	 	Now, notice that 
	 	\begin{align*}
	 		|\D f_2|^2=\frac{8|g'(z)|^2}{(1+|g(z)|^2)^2}=\frac{8}{|a|^2}+O(|z|),
	 	\end{align*}
	 	which implies that 
	 	\begin{align*}
	 		\log|\D f_2|=\frac{3}{2}\log(2)-\log|a|.
	 	\end{align*}
	 	Furthermore, the {expansion \eqref{eq:g_tilde_expansion}} shows that as $|z|\rightarrow \infty$, we have 
	 	\begin{align*}
	 	{g}(z)=az+O(1),
	 	\end{align*}
	 	so that $|a|=|{g}'(\infty)|$,
	 	and
	 	\begin{align}\label{constantes2}
	 		4\pi\log|\D f_2(0)|=-4\pi\log|{g}'(\infty)|+6\pi\log(2).
	 	\end{align}
	 	Finally, we deduce by \eqref{constantes1} and \eqref{constantes2} that 
	 	\begin{align}\label{constante}
	 		4\pi\log|\D f_1(0)|+4\pi\log|\D f_2(0)|-12\pi\log(2)=4\pi\log|f'(0)|-4\pi\log{g}'(\infty)|-4\pi\log(1+|f(0)|^2).
	 	\end{align}
	 	Gathering \eqref{n_loewner3}, \eqref{n_loewner4}, \eqref{n_loewner5}, \eqref{n_loewner6}, \eqref{n_loewner7}, and \eqref{constante},
	 	we finally deduce that
	 	\begin{align*}
	 		\mathscr{E}(\Gamma)&=\int_{\mathbb{D}}|\D\log|\D f_1||^2|dz|^2+\int_{\mathbb{D}}|\D\log|\D f_2||^2|dz|^2+\int_{\mathbb{D}}\log|z||\D f_1|^2|dz|^2+\int_{\mathbb{D}}\log|z||\D f_2|^2|dz|^2+4\pi\nonumber\\
	 		&+4\pi\log|\D f_1(0)|+4\pi\log|\D f_2(0)|-12\pi\log(2)\\
	 		&=\int_{\mathbb{D}}\left|\frac{f''(z)}{f'(z)}-2\frac{f'(z)}{f(z)}\frac{|f(z)|^2}{1+|f(z)|^2}\right|^2|dz|^2+\int_{\C\setminus\bar{\mathbb{\mathbb{D}}}}\left|\frac{{g}''(z)}{{g}'(z)}-2\frac{{g}'(z)}{{g}(z)}\frac{|{g}(z)|^2}{1+|{g}(z)|^2}+\frac{2}{z}\right|^2|dz|^2\\
	 		&+2\int_{\mathbb{D}}\log|z|\frac{4|f'(z)|^2|dz|^2}{(1+|f(z)|^2)^2}-2\int_{\C\setminus\bar{\mathbb{D}}}\log|z|\frac{4|{g}'(z)|^2|dz|^2}{(1+|{g}(z)|^2)^2}+4\pi\\
	 		&+4\pi\log|f'(0)|-4\pi\log|\tilde{g}'(\infty)|-4\pi\log(1+|f(0)|^2)\\
	 		&=S_3(\Gamma)
	 	\end{align*}
	 	which concludes the proof of the theorem. 
	 \end{proof}
	 \begin{rem}
	 	If $\Gamma=S^1$, then we can take $f=\mathrm{Id}_{\mathbb{D}}$ and $g=\mathrm{Id}_{\C\setminus\bar{\mathbb{D}}}$, 
	 	and we compute
	 	\begin{align*}
	 		S_3(\Gamma)&=\int_{\mathbb{D}}\left|\frac{f''(z)}{f'(z)}-2\frac{f'(z)}{f(z)}\frac{|f(z)|^2}{1+|f(z)|^2}\right|^2|dz|^2+\int_{\C\setminus\bar{\mathbb{\mathbb{D}}}}\left|\frac{g''(z)}{g'(z)}-2\frac{g'(z)}{g(z)}\frac{|g(z)|^2}{1+|g(z)|^2}+\frac{2}{z}\right|^2|dz|^2\\
	 		&+2\int_{\mathbb{D}}\log|z|\frac{4|f'(z)|^2|dz|^2}{(1+|f(z)|^2)^2}-2\int_{\C\setminus\bar{\mathbb{D}}}\log|z|\frac{4|g'(z)|^2|dz|^2}{(1+|g(z)|^2)^2}+4\pi\\
	 		&+4\pi\log|f'(0)|-4\pi\log|g'(\infty)|-4\pi\log(1+|f(0)|^2)\\
	 		&=8\int_{\mathbb{D}}\frac{|z|^2|dz|^2}{(1+|z|^2)^2}+16\int_{\mathbb{D}}\log|z|\frac{|z|^2|dz|^2}{(1+|z|^2)^2}+4\pi\\
	 		&=16\pi\int_{0}^1\frac{r^3dr}{(1+r^2)^2}+32\pi\int_{0}^1\frac{r\log(r)dr}{(1+r^2)^2}+4\pi\\
	 		&=16\pi\left(\frac{1}{4}(2\log(2)-1)\right)+32\pi\left(-\frac{1}{4}\log(2)\right)+4\pi\\
	 		&=0
	 	\end{align*}
	 	as expected.
	 \end{rem}
	 
	 In the next theorem, we finally complete the proof of \eqref{eq:chain_eq} by showing that $\pi I^L(\Gamma)=S_3(\Gamma)$.
	 
	 \begin{theorem}\label{s3}
	 	Let $\Gamma\subset S^2$ be a closed simple curve of finite Loewner energy. Then
	 	we have
	 	\begin{align*}
	 		 I^L(\Gamma)=\frac{1}{\pi}\mathscr{E}_0(\Gamma),
	 	\end{align*}
	 	where $\mathscr{E}_0$ is defined in 
	 	\eqref{thA4}. 
	 	Furthermore, if $\Omega_1,\Omega_2\subset S^2\setminus\Gamma$ are the two connected components of $S^2\setminus\Gamma$, for all conformal maps $f_1:\mathbb{D}\rightarrow \Omega_1$ and $f_2:\mathbb{D}\rightarrow \Omega_2$, we have
	 	\begin{align*}
	 	I^L(\Gamma)&=\frac{1}{\pi}\sum_{j=1}^{2}\left(\int_{\mathbb{D}}|\nabla\log|\nabla f_j||^2|dz|^2+\int_{\mathbb{D}}\log|z||\nabla f_j|^2|dz|^2+\mathrm{Area}(\Omega_j)
	 		+4\pi\log|\nabla f_j(0)|\right)-12\log(2)
	 	\end{align*}
	 	\normalsize
	 \end{theorem}
	 \vspace{-2.8em}
	 	          	 \begin{figure}[H]
	 	          	 	\centering
	 	 	\includegraphics[width=0.5\textwidth]{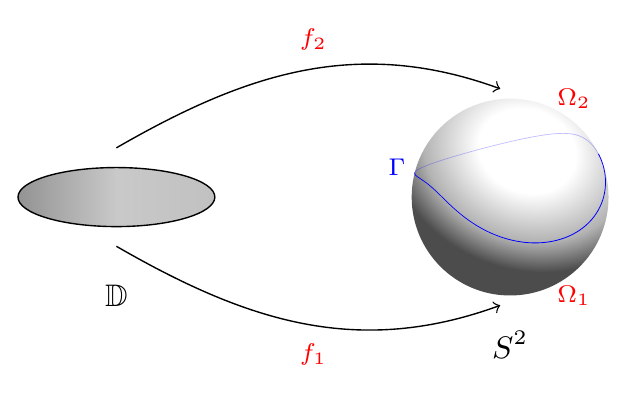}
	 	          	 	\vspace{-1em}
	 	          	 	\caption{Spherical formula for the Loewner energy with respect to conformal maps. 
	 	          	 	} 
	 	          	 \end{figure}

	 \begin{proof}
	    By Theorem \ref{wp10}, we have the identity $I^L(\Gamma)=\frac{1}{\pi}\mathscr{E}_0(\Gamma)$ for all smooth $\Gamma$, and by the preceding Theorem \ref{third_loewner}, we have $\mathscr{E}_0(\Gamma)=S_3(\Gamma)$ for any Jordan curve $\Gamma$ of finite Loewner energy. Therefore, we will prove that $I^L=\frac{1}{\pi }   S_3$ which will imply our result.

	 	We now let $\Omega_1,\Omega_2\subset S^2\setminus {\Gamma}$ be the two connected components of $S^2\setminus\Gamma$, and $f_1:\mathbb{D}\rightarrow \Omega_1$, $f_2:\mathbb{D}\rightarrow \Omega_2$ be the two 
	 	 conformal maps associated to $\Omega_1$ and $\Omega_2$, 
	 	 and let 
	 	 $p_1=f_1(0)$ and $p_2=f_2(0)$. Up to a rotation on $S^2$ (which does not change any of the energies considered), we can assume that $p_2=N$. If $\pi:S^2\setminus\ens{N}\rightarrow \C$ is the standard stereographic projection, let $\gamma=\pi(\Gamma)$, and $\Omega$ the bounded component of $\C\setminus\gamma$ and define $f=\pi\circ f_1:\mathbb{D}\rightarrow \pi(\Omega_1)=\Omega$ and $g=\pi\circ f_2\circ{\ii}:\C\setminus\bar{\mathbb{D}}\rightarrow \C\setminus\bar{\Omega}$ such that (using Theorem \ref{third_loewner})
	 	\begin{align}\label{s31}
	 		\mathscr{E}_0(\Gamma)&=S_3(\gamma)=\int_{\mathbb{D}}\left|\frac{f''(z)}{f'(z)}-2\frac{f'(z)}{f(z)}\frac{|f(z)|^2}{1+|f(z)|^2}\right|^2|dz|^2+\int_{\C\setminus\bar{\mathbb{\mathbb{D}}}}\left|\frac{g''(z)}{g'(z)}-2\frac{g'(z)}{g(z)}\frac{|g(z)|^2}{1+|g(z)|^2}+\frac{2}{z}\right|^2|dz|^2\nonumber\\
	 		&+2\int_{\mathbb{D}}\log|z|\frac{4|f'(z)|^2|dz|^2}{(1+|f(z)|^2)^2}-2\int_{\C\setminus\bar{\mathbb{D}}}\log|z|\frac{4|g'(z)|^2|dz|^2}{(1+|g(z)|^2)^2}+4\pi\nonumber\\
	 		&+4\pi\log|f'(0)|-4\pi\log|g'(\infty)|-4\pi\log(1+|f(0)|^2).
	 	\end{align}
	 	Now, by Corollary A.$4$ of \cite{takteo} and Theorem $8.1$ \cite{yilinvention}, if $\ens{\gamma_n}_{n\in\N}$ is a sequence of smooth curves converging uniformly to a simple curve $\gamma$ and such that for a sequence of maps $f_n:\mathbb{D}\rightarrow \C$ such that $f_n(0)=0$, $f_n'(0)=1$ and $f_n(\mathbb{D})=\Omega_n$, where $\Omega_n$ is the bounded component of $\C\setminus \gamma_n$, and satisfies
	 	\begin{align*}
	 		\lim\limits_{n\rightarrow \infty}\int_{\mathbb{D}}\left|\frac{f_n''(z)}{f_n'(z)}-\frac{f''(z)}{f'(z)}\right|^2|dz|^2=0,
	 	\end{align*}
	 	where $f:\mathbb{D}\rightarrow \Omega$ is a univalent function such that $f(0)=0$ and $f'(0)=1$, we have
	 	\begin{align}\label{convergence}
	 		I^L(\Gamma_n)\conv{n\rightarrow \infty}I^L(\Gamma).
	 	\end{align}
	 	In particular, for any sequence of holomorphic maps $g_n:\C\setminus\bar{\mathbb{D}}\rightarrow \C\setminus\bar{\Omega_n}$ such that $g_n(\infty)=\infty$, since
	 	\begin{align*}
	 		S_1(\Gamma_n)=\pi\, I^L(\Gamma_n)&=\int_{\mathbb{D}}\left|\frac{f''_n(z)}{f_n'(z)}\right|^2|dz|^2+\int_{\C\setminus\bar{\mathbb{D}}}\left|\frac{g_n''(z)}{g_n'(z)}\right|^2|dz|^2+4\pi\log|f_n'(0)|-4\pi\log|g_n'(\infty)|\\
	 		&=\int_{\mathbb{D}}\left|\frac{f''_n(z)}{f_n'(z)}\right|^2|dz|^2+\int_{\C\setminus\bar{\mathbb{D}}}\left|\frac{g_n''(z)}{g_n'(z)}\right|^2|dz|^2-4\pi\log|g_n'(\infty)|
	 	\end{align*}
	 	we deduce that 
	 	\begin{align}\label{additional}
	 		\int_{\C\setminus\bar{\mathbb{D}}}\left|\frac{g_n''(z)}{g_n'(z)}\right|^2|dz|^2-4\pi\log|g_n'(\infty)|\conv{n\rightarrow \infty}\int_{\C\setminus\bar{\mathbb{D}}}\left|\frac{g''(z)}{g'(z)}\right|^2|dz|^2-4\pi\log|g'(\infty)|
	 	\end{align}
	 	for all univalent function $g:\C\setminus\bar{\mathbb{D}}\rightarrow\C\setminus\bar{\Omega}$ such that $g(\infty)=\infty$.
	 	Now, if $\gamma=\pi(\Gamma)\subset \C$, let $\ens{\epsilon_n}_{n\in\N}\subset (0,\infty)$ such that $\epsilon_n\conv{n\rightarrow \infty}0$, and define 
	 	\begin{align*}
	 		f_n:\mathbb{D}&\rightarrow \C\\
	 		z&\mapsto f((1-\epsilon_n)z)/{(1-\epsilon_n)}.
	 	\end{align*}
	 	Then $\gamma_n=f_n(S^1)$ is smooth and uniformly converges to $\gamma$. Furthermore, we have
	 	\begin{align}\label{first_convergence}
	 		\int_{\mathbb{D}}\left|\frac{f''_n(z)}{f'_n(z)}-\frac{f''(z)}{f'(z)}\right|^2|dz|^2\conv{n\rightarrow \infty}0.
	 	\end{align}
	 	which implies that 
	 	\begin{align*}
	 		I^L(\gamma_n)=\frac{1}{\pi}S_1(\gamma_n)\conv{n\rightarrow \infty}\frac{1}{\pi}S_1(\gamma)=I^L(\gamma).
	 	\end{align*}
        Now, we need to show the result $f_n'\conv{n\rightarrow{ \infty}}f'$ in $L^2(\mathbb{D})$ strongly. 
        Notice that since $f'$ is smooth in $\mathbb{D}$, we have by construction $f_n'\conv{n\rightarrow\infty}f'$ almost everywhere. Furthermore, a linear change of variable shows that
        \begin{align*}
            \int_{\mathbb{D}}|f_n'(z)|^2|dz|^2=\int_{\mathbb{D}}|f'((1-\epsilon_n)z))|^2|dz|^2&=\frac{1}{(1-\epsilon_n)^2}\int_{\mathbb{D}(0,1-\epsilon_n)}|f'(w)|^2|dw|^2
            \conv{n\rightarrow \infty}\int_{\mathbb{D}}|f'(w)|^2|dw|^2.
        \end{align*}
        By Brezis-Lieb lemma (\cite{brezis_lieb}), since $f_n'\conv{n\rightarrow \infty}f'$ almost everywhere and $\np{f_n'}{2}{\mathbb{D}}\conv{n\rightarrow\infty}\np{f'}{2}{\mathbb{D}}$, we deduce that 
        \begin{align}\label{strong_convergence}
            f_n'\conv{n\rightarrow \infty} f'\quad \text{strongly in}\;\, L^2(\mathbb{D}).
        \end{align}
        Therefore, we also get the convergence
	 	\begin{align*}
	 		\begin{alignedat}{2}
	 			\frac{f_n'}{f_n}\frac{|f_n|^2}{1+|f_n|^2}&\conv{n\rightarrow \infty} \frac{f'}{f}\frac{|f|^2}{1+|f|^2}\quad&& \text{in}\;\, L^2(\mathbb{D})
	 		\end{alignedat}
	 	\end{align*}
	 	which finally shows by \eqref{first_convergence} that
	 	\begin{align}\label{s32}
	 		\int_{\mathbb{D}}\left|\frac{f_n''(z)}{f_n'(z)}-2\frac{f_n'(z)}{f_n(z)}\frac{|f_n(z)|^2}{1+|f_n(z)|^2}\right|^2|dz|^2&\conv{n\rightarrow \infty}\int_{\mathbb{D}}\left|\frac{f''(z)}{f'(z)}-2\frac{f'(z)}{f(z)}\frac{|f(z)|^2}{1+|f(z)|^2}\right|^2|dz|^2\nonumber\\
	 		\int_{\mathbb{D}}\log|z|\frac{4|f_n'(z)|^2|dz|^2}{(1+|f_n(z)|^2)^2}&\conv{n\rightarrow \infty}\int_{\mathbb{D}}\log|z|\frac{4|f'(z)|^2|dz|^2}{(1+|f(z)|^2)^2}.
	 	\end{align}
	 	Finally, we also have $f_n(0)=f(0)$ and
	 	\begin{align}\label{s33}
	 		4\pi\log|f_n'(0)|=4\pi\log|f'(0)|+4\pi\log(1-\epsilon_n)\conv{n\rightarrow \infty}4\pi\log|f'(0)|.
	 	\end{align}
	 	Therefore, if $\Omega_n=f_n(\mathbb{D})$, and $g_n:\mathbb{C}\setminus\bar{\mathbb{D}}\rightarrow \C\setminus\bar{\Omega_n}$ is any univalent map such that $g_n(\infty)=\infty$, since $\gamma_n\conv{n\rightarrow \infty}\gamma$ uniformly, we can assume without loss of generality that $g_n'(\infty)\conv{n\rightarrow \infty}g'(\infty)$. Furthermore, by Corollary A.$4$ of \cite{takteo}, we also get
	 	\begin{align}\label{second_conv}
	 		&\lim\limits_{n\rightarrow \infty}\int_{\C\setminus\bar{\mathbb{D}}}\left|\frac{g_n''(z)}{g_n'(z)}-\frac{g''(z)}{g(z)}\right|^2=0\\
	 		&\lim\limits_{n\rightarrow \infty}\int_{\C\setminus\bar{\mathbb{D}}}\left|\left(\frac{g_n''(z)}{g_n'(z)}-2\frac{g_n'(z)}{g_n(z)}+\frac{2}{z}\right)-\left(\frac{g''(z)}{g'(z)}-2\frac{g'(z)}{g(z)}+\frac{2}{z}\right)\right|^2|dz|^2=0.
	 	\end{align}
	 	As previously, we have
	 	\begin{align}\label{third_conv}
	 		\frac{g_n'}{g_n}\frac{1}{1+|g_n|^2}&\conv{n\rightarrow \infty}\frac{g'}{g}\frac{1}{1+|g|^2}\quad \text{in}\;\, L^2(\C\setminus \bar{\mathbb{D}})\nonumber\\
            {g_n'}\frac{1}{1+|g_n|^2}&\conv{n\rightarrow \infty}{g'}\frac{1}{1+|g|^2}\quad \text{in}\;\, L^2(\C\setminus \bar{\mathbb{D}})
	 	\end{align}
	 	Therefore, \eqref{second_conv} and \eqref{third_conv} imply that 
	 	\begin{align}\label{s34}
	 		\int_{\C\setminus\bar{\mathbb{\mathbb{D}}}}\left|\frac{g_n''(z)}{g_n'(z)}-2\frac{g_n'(z)}{g_n(z)}\frac{|g_n(z)|^2}{1+|g_n(z)|^2}+\frac{2}{z}\right|^2|dz|^2&\conv{n\rightarrow \infty}\int_{\C\setminus\bar{\mathbb{\mathbb{D}}}}\left|\frac{g''(z)}{g'(z)}-2\frac{g'(z)}{g(z)}\frac{|g(z)|^2}{1+|g(z)|^2}+\frac{2}{z}\right|^2|dz|^2\nonumber\\
	 		\int_{\C\setminus\bar{\mathbb{\mathbb{D}}}}\log|z|\frac{4|g_n'(z)|^2|dz|^2}{(1+|g_n(z)|^2)^2}&\conv{n\rightarrow \infty}\int_{\C\setminus\bar{\mathbb{\mathbb{D}}}}\log|z|\frac{4|g'(z)|^2|dz|^2}{(1+|g(z)|^2)^2}.
	 	\end{align}
	 	Finally, we deduce by \eqref{s31}, \eqref{s32}, \eqref{s33} and \eqref{s34} that 
	 	\begin{align*}
	 		S_3(\gamma_n)\conv{n\rightarrow\infty}S_3(\gamma)
	 	\end{align*}
	 	which concludes the proof of the theorem by \eqref{convergence}, Theorem \ref{wp10} and Theorem \ref{third_loewner}.
	 \end{proof}
  	 \begin{rems}\label{remTheoA}	 

	 Notice that we can also directly express the Loewner energy using moving frames. First, we trivially have
	 \begin{align*}
	 I^L(\Gamma)&=\frac{1}{\pi}\left\{\sum_{i=1}^2\int_{\Omega_i}|\omega_i-\ast\,dG_{\Omega_i}|_{g_0}^2+2\int_{\Omega_i}G_{\Omega_i}K_{g_0}d\mathrm{vol}_{g_0}+\mathrm{Area}(\Omega_i)\right\}\\
	 & +4\log|\D f_1(0)|+4\log|\D f_2(0)|-12\,\log(2).
	 \end{align*}
	 Alternatively, we have
	 \begin{align*}
	 I^L(\Gamma)&=\frac{1}{2 { \pi}}\int_{\Omega_1}\left(|d\e_1|^2_{g_0}+|d\f_1|^2_{g_0}-2|dG_{\Omega_1}|^2_{g_0}\right)d\mathrm{vol}_{g_0}+\frac{1}{2 {\pi}}\int_{\Omega_2}\left(|d\e_2|^2_{g_0}+|d\f_2|^2_{g_0}-2|dG_{\Omega_2}|^2_{g_0}\right)d\mathrm{vol}_{g_0}\nonumber\\
	 	 		&+4\log|\D f_1(0)|+4\log|\D f_2(0)|-12\log(2),
	 \end{align*}
  which is (up to the second line involving the conformal maps $f_1$ and $f_2$) very reminiscent of the Ginzburg-Landau renormalised energy (\cite[Chapter VIII]{BBH}).
  
	 To see this equality, since $\e_1$, $\f_1$ and $\n$ are unitary, we have
	 \begin{align*}
	 |d\e_1|^2_{g_0}&=|\s{d\e_1}{\f_1}|_{g_0}^2+|\s{d\e_1}{\n}|_{g_0}^2=|\omega_1|^2_{g_0}+|\s{d\n}{\e_1}|_{g_0}^2\\
	 |d\f_1|^2_{g_0}&=|\omega_1|^2_{g_0}+|\s{d\n}{\f_1}|_{g_0}^2\\
	 |d\e_1|^2_{g_0}+|d\f_1|_{g_0}^2&=2|\omega_1|^{2}_{g_0}+|d\n|_{g_0}^2=2|\omega_1|_{g_0}^2+2.
	 \end{align*}
	 Then, integrating by parts and using that $G_{\Omega_1}=0$ on $\partial\Omega_1$, we deduce by Stokes theorem—and the equation (that follows from \eqref{liouville1})
	 \begin{align*}
	 d\left(\omega_1-\ast\,dG_{\Omega_1}\right)=-K_{g_0}d\mathrm{Area}_{g_0},
	 \end{align*}
	 where $K_{g_0}=1$ is the Gauss curvature of the sphere—that 
	 \begin{align*}
	 &\frac{1}{2}\int_{\Omega_1}\left(|d\e_1|^2_{g_0}+|d\f_1|_{g_0}^2-2|dG_{\Omega_1}|_{g_0}^2\right)d\mathrm{vol}_{g_0}=\int_{\Omega_1}\left(|\omega_1|_{g_0}^2-|dG_{\Omega_1}|^2_{g_0}\right)d\mathrm{vol}_{g_0}+\mathrm{Area}_{g_0}(\Omega_1)\\
	 &=\int_{\Omega_1}\s{\omega_1-\ast\, dG_{\Omega_1}}{\omega_1+\ast\, dG_{\Omega_1}}_{g_0}d\mathrm{vol}_{g_0}+\mathrm{Area}_{g_0}(\Omega_1)\\
	 &=\int_{\Omega_1}|\omega_1-\ast\, dG_{\Omega_1}|_{g_0}^2d\mathrm{vol}_{g_0}+2\int_{\Omega_1}(\omega_1-\ast\, dG_{\Omega_1})\wedge dG_{\Omega_1}+\mathrm{Area}_{g_0}(\Omega_1)\\
	 &=\int_{\Omega_1}|\omega_1-\ast\,dG_{\Omega_1}|_{g_0}^2 {d\mathrm{vol}_{g_0}}-2\int_{\Omega_1}G_{\Omega_1}d(\omega_1-\ast\,dG_{\Omega_1})+\mathrm{Area}_{g_0}(\Omega_1)\\
	 &=\int_{\Omega_1}|d\mu_1|_{g_0}^2d\mathrm{vol}_{g_0}+2\int_{\Omega_1}G_{\Omega_1}K_{g_0}d\mathrm{vol}_{g_0}+\mathrm{Area}_{g_0}(\Omega_1)
	 \end{align*}
	 which implies since $\mathrm{Area}_{g_0}(S^2)=4\pi$ that
	 \begin{align*}
	 &\frac{1}{2}\int_{\Omega_1}\left(|d\e_1|^2_{g_0}+|d\f_1|^2_{g_0}-2|dG_{\Omega_1}|^2_{g_0}\right)d\mathrm{vol}_{g_0}+\frac{1}{2}\int_{\Omega_1}\left(|d\e_2|^2_{g_0}+|d\f_2|^2_{g_0}-2|dG_{\Omega_2}|^2_{g_0}\right)d\mathrm{vol}_{g_0}\\
	 &=\int_{\Omega_1}|d\mu_1|_{g_0}^2d\mathrm{vol}_{g_0}+\int_{\Omega_2}|d\mu_2|_{g_0}^2+2\int_{\Omega_1}G_{\Omega_1}K_{g_0}d\mathrm{vol}_{g_0}+2\int_{\Omega_2}G_{\Omega_2}K_{g_0}d\mathrm{vol}_{g_0}+4\pi.
	 \end{align*}
	 Notice that it gives another explanation for the factor $4\pi$ in the definition of $\mathscr{E}$.
	 \end{rems}

	 \section{Appendix}

In this appendix, we provide more details on the geodesic curvature for Weil-Petersson quasicircles and show a consequence of Theorem~\ref{s3} which is an identity on univalent functions associated to a Weil-Petersson quasicircle.
	
\subsection{Properties of the Geodesic Curvature for Weil-Petersson Quasicircles}\label{geodesicprop}

      	 \begin{lemme}\label{poinca}
	 Let $\mathbb{H}=\C\cap\ens{z:\Im(z)>0}$ be the Poincaré half-plane, and $f:\mathbb{H}\rightarrow \C$ a univalent holomorphic map, $\Omega=f(\mathbb{H})$, and assume that $\gamma=\partial \Omega$ is a simple curve of finite Loewner energy. Then the geodesic curvature $k_{g_0}$ of $\gamma$ is given in the distributional sense by
	 \begin{align}\label{geodesic_flat}
	 k_{g_0}=\Im\left(\frac{f''(z)}{f'(z)}\right) \qquad \text{for all}\;\, z\in\partial_{\infty}\mathbb{H}=\R.
	 \end{align}
	 \vspace{-1em}
	 \end{lemme}\label{geodesic_curvature_poinca}
	 \begin{proof}
	 The geodesic curvature is given by
	    \begin{align*}
	    k_{g_0}=\s{\p{x}\e}{\f},
	    \end{align*}
	    if $(\e,\f)$ is the Cartesian frame given by (in the following formulae, $f$ is seen as a $\R^2$-valued function)
	    \begin{align*}
	    \left\{\begin{alignedat}{1}
	    \e&=\frac{\p{x}f}{|\p{x}f|}=\left(\Re\left(\frac{f'(z)}{|f'(z)|}\right),\Im\left(\frac{f'(z)}{|f'(z)|}\right)\right)=\frac{f'(z)}{|f'(z)|}\\
	    \f&=\frac{\p{y}f}{|\p{y}f|}=\left(-\Im\left(\frac{f'(z)}{|f'(z)|}\right),\Re\left(\frac{f'(z)}{|f'(z)|}\right)\right)=i\,\frac{f'(z)}{|f'(z)|}.
	    \end{alignedat}\right.
	    \end{align*}
	    Define $\e_z=\dfrac{f'(z)}{|f'(z)|}$. Then we have
        \begin{align*}
	    \p{z}\e_z&=\p{z}\left(\frac{f'(z)}{|f'(z)|}\right)=\frac{f''(z)}{|f'(z)|}-\frac{1}{2}\frac{f''(z)}{|f'(z)|}=\frac{1}{2}\frac{f''(z)}{|f'(z)|}=\frac{1}{2}\frac{f''(z)}{f'(z)}\e_z\\
	    \p{\z}\e_z&=\p{\z}\left(\frac{f'(z)}{|f'(z)|}\right)=-\frac{1}{2}\frac{f'(z)^2\bar{f''(z)}}{|f'(z)|^3}=-\frac{1}{2}\bar{\left(\frac{f''(z)}{f'(z)}\right)}\e_z.
	    \end{align*}	    
	    Therefore, we deduce that
	    \begin{align*}
	    \p{z}\Re\left(\frac{f'(z)}{|f'(z)|}\right)&=\frac{1}{2}\left(\p{z}\left(\frac{f'(z)}{|f'(z)|}\right)+\bar{\p{\z}\left(\frac{{f'(z)}}{|f'(z)|}\right)}\right)=\frac{1}{4}\frac{f''(z)}{f'(z)}\left(\e_z-\e_{\z}\right)=\frac{i}{2}\frac{f''(z)}{f'(z)}\Im\left(\frac{f'(z)}{|f'(z)|}\right)\\
	    \p{z}\Im\left(\frac{f'(z)}{|f'(z)|}\right)&=-\frac{i}{2}\left(\frac{1}{2}\frac{f''(z)}{|f'(z)|}+\frac{1}{2} \bar{\frac{f''(z)}{|f'(z)|}}\right)=-\frac{i}{4}\frac{f''(z)}{f'(z)}\left(\e_z+\e_{\z}\right)=-\frac{i}{2}\frac{f''(z)}{f'(z)}\Re\left(\frac{f'(z)}{|f'(z)|}\right)
	    \end{align*}
	    where we used
	    
	    Therefore, we have
	    \begin{align*}
	    \p{x}\Re\left(\frac{f'(z)}{|f'(z)|}\right)&=2\,\Re\left(\p{z}\Re\left(\frac{f'(z)}{|f'(z)|}\right)\right)=-\Im\left(\frac{f''(z)}{f'(z)}\right)\Im\left(\frac{f'(z)}{|f'(z)|}\right)\\
	    \p{x}\Im\left(\frac{f'(z)}{|f'(z)|}\right)&=\Im\left(\frac{f''(z)}{f'(z)}\right)\Re\left(\frac{f'(z)}{f'(z)}\right),
	    \end{align*}
	    so that
	    $
	    \p{x}\e=\Im\left(\dfrac{f''(z)}{f'(z)}\right)\f,
	    $
	    and
	    \begin{align*}
	    k_{g_0}=\s{\p{x}\e}{\f}=\Im\left(\frac{f''(z)}{f'(z)}\right),
	    \end{align*}	  
	    which concludes the proof of the lemma.
	    \end{proof}
     
 	 \begin{lemme}\label{new_geodesic}
 	  	 	Let $\gamma\subset\C$ be a Weil-Petersson quasicircle. Then the geodesic curvature $k_{g_0}:S^1\rightarrow \R$ is a tempered distribution of order at most $2$. More precisely, we have $k_{g_0}\in H^{-1/2}(S^1)$. 
 	  	 	\end{lemme}
        \begin{proof}
  	 	Either using the Poincaré half-plane $\mathbb{H}$ and the formula \eqref{geodesic_flat} or \eqref{kg0}, we 
  	 	get
  	 	\begin{align}\label{partial}
  	 	k_{g_0}=\Re\left(z\frac{f''(z)}{f'(z)}\right)+1.
  	 	\end{align}
  	 	Now, if $0<\epsilon<1$ and $f_{\epsilon}:\mathbb{D}\rightarrow \Omega$ is defined by 
  	 	\begin{align*}
  	 	f_{\epsilon}(z)=\frac{1}{1-\epsilon}f((1-\epsilon)z)\qquad z\in\mathbb{D},
  	 	\end{align*}
  	 	we have (see \cite{yilinvention}, Lemma $8.2$)
  	 	\begin{align*}
  	 	\lim\limits_{\epsilon\rightarrow 0}\int_{\mathbb{D}}\left|\frac{f_{\epsilon}''(z)}{f_{\epsilon}'(z)}-\frac{f''(z)}{f'(z)}\right|^2|dz|^2=0,
  	 	\end{align*}
  	 	which is equivalent by trace theory to
  	 	\begin{align*}
  	 	\lim\limits_{\epsilon\rightarrow 0}\hs{\log|f'_{\epsilon}|-\log|f'|}{{1}/{2}}{S^1}=0,\\
  	 	\lim\limits_{\epsilon\rightarrow 0}\hs{\arg(f'_{\epsilon})-\arg(f')}{{1}/{2}}{S^1}=0.
  	 	\end{align*}
  	 	and using the equivalent norm for $H^s$ spaces given by
  	 	\begin{align*}
  	 	\hs{u}{s}{S^1}=\sum_{n\in\Z}^{}|n|^{2s}|a_n|^2,\qquad \text{if}\;\, u(z)=\sum_{n\in\Z}a_nz^n.
  	 	\end{align*}
  	 	we deduce 
  	 	that
  	 	 \begin{align*}
  	 	  	 	\lim\limits_{\epsilon\rightarrow 0}\hs{\partial_{\theta}\log(f_{\epsilon}')-\partial_{\theta}\log(f')}{-{1}/{2}}{S^1}=0.
  	 	\end{align*}
  	 	Since
  	 	\begin{align*}
  	 	\frac{f''(z)}{f'(z)}=\frac{1}{ie^{i\theta}}\partial_{\theta}\log(f'(z)),
  	 	\end{align*}
  	 	we deduce that $\dfrac{f''}{f'}\in H^{-1/2}(S^1)$ and that concludes the proof of the lemma by \eqref{partial}.
  	 	\end{proof}
 	  	 	
 	  	\begin{rem}
 	  	For other considerations related to trace spaces, see the Definition $5$ of \cite{bishop2bis} and \cite{bishop_ft}.
 	  	\end{rem} 	
 	 	 \setcounter{footnote}{0}
 	 	\begin{prop}
 	 	For all $\epsilon>0$, let $\mathbb{D}_+(0,\epsilon)=\mathbb{H}\cap\ens{z:|z|<\epsilon}$, where $\mathbb{H}=\C\cap\ens{z:\Im(z)>0}$ is the Poincaré half-plane. For $\epsilon>0$ small enough, the map $f:\mathbb{D}_+(0,\epsilon)\rightarrow \mathbb{C}$ defined by
 	 	\begin{align*}
 	 	f(z)=z\,e^{i\log\log(z)},
 	 	\end{align*}
 	 	where $\log(z)$ is the principal value of the logarithm on $\mathbb{H}$, is an immersion and $\log|f'|\in W^{1,2}(\mathbb{D}_+(0,\epsilon))$. In particular, the curve $\gamma:(-\epsilon,\epsilon)\rightarrow \C$ such that $\gamma(t)=te^{i\log\log(t)}$ for all $t\in (-\epsilon,\epsilon)$ is a part of a Weil-Petersson quasicircle.\footnote{Beware that the $\log$ function here is defined as the trace of our continuous determination of the logarithm on the upper-half plane and is not the standard $\log$ function on $(0,\infty)$.} Furthermore, its geodesic curvature $k_{g_0}$ is given by
 	 	\begin{align*}
 	 	k_{g_0}=-\frac{1}{t(1+\log^2(t))}+\mathrm{p.v.}\int_{-\epsilon}^{\epsilon}\frac{dt}{t\log(t)}.
 	 	\end{align*}
 	 	\end{prop}
 	 	\begin{proof}
 	 	We compute
 	 	\begin{align*}
 	 	f'(z)&=e^{i\log\log(z)}+\frac{i}{\log(z)}e^{i\log\log(z)}=\left(1+\frac{i}{\log(z)}\right)e^{i\log\log(z)}\\
 	 	f''(z)&=-\frac{i}{z\log^2(z)}e^{i\log\log(z)}+\frac{i}{z\log(z)}\left(1+\frac{i}{\log(z)}\right)e^{i\log\log(z)}.
 	 	\end{align*}
 	 	Therefore, we have 
 	 	\begin{align*}
 	 	\frac{f''(z)}{f'(z)}=-\frac{i}{z\log(z)(i+\log(z))}+\frac{i}{z\log(z)}=i\frac{-1+i+\log(z)}{z\log(z)(i+\log(z))}.
 	 	\end{align*}
 	 	Notice that we have
 	 	\begin{align*}
 	 	|i+\log(z)|^2=|\log|z|+i(1+\arg(z))|^2=\log^2|z|+(1+\arg(z))^2
 	 	\end{align*}
 	 	Therefore, we have
 	 	\begin{align*}
 	 	\left|\frac{f''(z)}{f'(z)}\right|^2=\frac{(\log|z|-1)^2+(1+\arg(z))^2}{|z|^2\log^2|z|(\log^2|z|+(1+\arg(z))^2}\leq \frac{1}{\log^2|z|}+\frac{-2\log|z|+1}{\log^4|z|}
    \leq\frac{2}{|z|^2\log^4|z|}+\frac{2}{|z|^2\log^2|z|},
 	 	\end{align*}
 	 	and 
 	 	\begin{align*}
 	 	\int_{\mathbb{D}(0,\frac{1}{2})}\left|\frac{f''(z)}{f'(z)}\right|^2|dz|^2\leq 4\pi\int_{0}^{\frac{1}{2}}\frac{dr}{r\log^4(r)}+4\pi\int_{0}^{\frac{1}{2}}\frac{dr}{r\log^2(r)}=\frac{4\pi}{\log(2)}+\frac{4\pi}{3\log^2(2)}<\infty,
 	 	\end{align*}
 	 	which shows that $\gamma$ is a Weil-Petersson quasicircle. Then, we have by Lemma \ref{geodesic_flat} for $z\in\R$
 	 	\begin{align*}
 	 	k_{g_0}=\Im\left(\frac{f''(z)}{f'(z)}\right)=\frac{1-\log|z|+\log^2|z|}{|z|\log|z|(1+\log^2|z|)}=-\frac{1}{|z|(1+\log^2|z|)}+\frac{1}{|z|\log|z|},
 	 	\end{align*}
 	 	which concludes the proof of the proposition.
 	 	\end{proof}
 	 	\begin{rem}
 	 	In particular, we see that there exists curves whose geodesic curvature is a distribution of order $1$. This curve is an example of spiral mentioned earlier in the introduction. 
 	 	\end{rem}

\subsection{A consequence of  Theorem \ref{s3}}
	 
	 The new identity of $\pi\,I^L=S_3$ from Theorem \ref{third_loewner} and Theorem \ref{s3} provides a new identity about holomorphic univalent maps of the plane. 
 
	 \begin{lemme}\label{grunsky}
		Let $\gamma\subset\C$ be a closed simple curve with finite Loewner energy. 
	 	We have
	 	\begin{align*}
	 		&4\,\Re\int_{\mathbb{D}}\left(\frac{f''(z)}{f'(z)}-2\frac{f'(z)}{f(z)}+\frac{2}{z}\right)\bar{\left(\frac{f'(z)}{f(z)}\frac{1}{1+|f(z)|^2}-\frac{2}{z}\right)}|dz|^2+4\int_{\mathbb{D}}\left|\frac{f'(z)}{f(z)}\frac{1}{1+|f(z)|^2}-\frac{1}{z}\right|^2|dz|^2\nonumber\\
	 		&+4\,\Re\int_{\C\setminus\bar{\mathbb{D}}}\left(\frac{g''(z)}{g'(z)}-2\frac{g'(z)}{g(z)}+\frac{2}{z}\right)\bar{\left(\frac{g'(z)}{g(z)}\frac{1}{1+|g(z)|^2}\right)}|dz|^2+4\int_{\C\setminus\bar{\mathbb{D}}}\left|\frac{g'(z)}{g(z)}\frac{1}{1+|g(z)|^2}\right|^2|dz|^2\nonumber\\
	 		&+2\int_{\mathbb{D}}\log|z|\frac{4|f'(z)|^2|dz|^2}{(1+|f(z)|^2)^2}-2\int_{\C\setminus\bar{\mathbb{D}}}\log|z|\frac{4|g'(z)|^2|dz|^2}{(1+|g(z)|^2)^2}+4\pi=0,
	 	\end{align*}
        where $f$ and $g$ are univalent maps as in Definition \ref{def:s3}.
	 \end{lemme}
	 \begin{proof}
	 Recall the definition of $S_3(\gamma)$ in Definition~\ref{def:s3}:  
	\begin{align}\label{new_grunsky1}
	 		S_3(\gamma)   	&=\int_{\mathbb{D}}\left|\frac{f''(z)}{f'(z)}-2\frac{f'(z)}{f(z)}\frac{|f(z)|^2}{1+|f(z)|^2}\right|^2|dz|^2+\int_{\C\setminus\bar{\mathbb{\mathbb{D}}}}\left|\frac{g''(z)}{g'(z)}-2\frac{g'(z)}{g(z)}\frac{|g(z)|^2}{1+|g(z)|^2}+\frac{2}{z}\right|^2|dz|^2\nonumber\\
	 		&+2\int_{\mathbb{D}}\log|z|\frac{4|f'(z)|^2|dz|^2}{(1+|f(z)|^2)^2}-2\int_{\C\setminus\bar{\mathbb{D}}}\log|z|\frac{4|g'(z)|^2|dz|^2}{(1+|g(z)|^2)^2}+4\pi\nonumber\\
	 		&+4\pi\log|f'(0)|-4\pi\log|g'(\infty)|
	 	\end{align}
	 	and that $\pi I^L(\gamma)=S_1(\gamma)=S_3(\gamma)$ from Theorem~\ref{s3}.
	 Using the identity  \eqref{s2}, we obtain
	 \begin{align}\label{new_grunsky2}
	     S_3(\gamma)  & = S_3 (\ii (\gamma)) \\
	         & =\int_{\mathbb{D}}\left|\frac{f''(z)}{f'(z)}-2\frac{f'(z)}{f(z)}+\frac{2}{z}+2\frac{f'(z)}{f(z)}\frac{1}{1+|f(z)|^2}-\frac{2}{z}\right|^2|dz|^2\nonumber\\
	 		&
	 		\quad +\int_{\C\setminus\bar{\mathbb{D}}}\left|\frac{g''(z)}{g'(z)}-2\frac{g'(z)}{g(z)}+\frac{2}{z}+2\frac{g'(z)}{g(z)}\frac{1}{1+|g(z)|^2}\right|^2|dz|^2+4\pi\log\abs{\frac{f'(0)}{g'(\infty)}}\nonumber\\
	 		&\quad +2\int_{\mathbb{D}}\log|z|\frac{4|f'(z)|^2|dz|^2}{(1+|f(z)|^2)^2}-2\int_{\C\setminus\bar{\mathbb{D}}}\log|z|\frac{4|g'(z)|^2|dz|^2}{(1+|g(z)|^2)^2}+4\pi\nonumber\\
	 		&=\int_{\mathbb{D}}\left|\frac{f''(z)}{f'(z)}-2\frac{f'(z)}{f(z)}+\frac{2}{z}\right|^2|dz|^2+\int_{\C\setminus\bar{\mathbb{D}}}\left|\frac{g''(z)}{g'(z)}-2\frac{g'(z)}{g(z)}+\frac{2}{z}\right|^2    |dz|^2	+4\pi\log\abs{\frac{f'(0)}{g'(\infty)}}\nonumber\\
	 		&\quad +4\,\Re\int_{\mathbb{D}}\left(\frac{f''(z)}{f'(z)}-2\frac{f'(z)}{f(z)}+\frac{2}{z}\right)\bar{\left(\frac{f'(z)}{f(z)}\frac{1}{1+|f(z)|^2}-\frac{2}{z}\right)}|dz|^2\nonumber\\
	 		&\quad +4\,\Re\int_{\C\setminus\bar{\mathbb{D}}}\left(\frac{g''(z)}{g'(z)}-2\frac{g'(z)}{g(z)}+\frac{2}{z}\right)\bar{\left(\frac{g'(z)}{g(z)}\frac{1}{1+|g(z)|^2}\right)}|dz|^2\nonumber\\
	 		&\quad +4\int_{\mathbb{D}}\left|\frac{f'(z)}{f(z)}\frac{1}{1+|f(z)|^2}-\frac{1}{z}\right|^2 |dz|^2 +4\int_{\C\setminus\bar{\mathbb{D}}}\left|\frac{g'(z)}{g(z)}\frac{1}{1+|g(z)|^2}\right|^2|dz|^2\nonumber\\
	 		&\quad +2\int_{\mathbb{D}}\log|z|\frac{4|f'(z)|^2|dz|^2}{(1+|f(z)|^2)^2}|dz|^2-2\int_{\C\setminus\bar{\mathbb{D}}}\log|z|\frac{4|g'(z)|^2|dz|^2}{(1+|g(z)|^2)^2}+4\pi.\nonumber
	 \end{align}
	 	Comparing \eqref{new_grunsky1} and \eqref{new_grunsky2}, we get the claimed identity.
	 \end{proof}
	 Let us check the formula in the case $\gamma=S^1$. 
	 In this case, we have $f(z)=z$, and $g(z)=z$, and the sum in Lemma \ref{grunsky} simplifies to
	 \begin{align}\label{sum}
	 	4\int_{\mathbb{D}}\left|\frac{1}{z}\frac{1}{1+|z|^2}-\frac{1}{z}\right|^2|dz|^2+4\int_{\C\setminus\bar{\mathbb{D}}}\left|\frac{1}{z}\frac{1}{1+|z|^2}\right|^2|dz|^2+2\int_{\mathbb{D}}\frac{4\log|z||dz|^2}{(1+|z|^2)^2}-2\int_{\mathbb{C}\setminus\bar{\mathbb{D}}}\frac{4\log|z||dz|^2}{(1+|z|^2)^2}+4\pi.
	 \end{align}
	 First, we have
	 \begin{align}\label{sum0}
	 	\int_{\mathbb{D}}\left|\frac{1}{z}\frac{1}{1+|z|^2}-\frac{1}{z}\right|^2|dz|^2=\int_{\mathbb{D}}\left|\frac{\z}{1+|z|^2}\right|^2|dz|^2=\int_{\mathbb{D}}\frac{|z|^2|dz|^2}{(1+|z|^2)^2}
	 \end{align}
	 and an immediate change of variable $z\mapsto\dfrac{1}{z}$ shows that 
	 \begin{align}\label{sum1}
	 \left\{\begin{alignedat}{1}
	 	&\int_{\mathbb{C}\setminus\bar{\mathbb{D}}}\left|\frac{1}{z}\frac{1}{1+|z|^2}\right|^2|dz|^2=\int_{\mathbb{D}}\frac{|z|^2|dz|^2}{(1+|z|^2)^2}\\
	 	&\int_{\mathbb{C}\setminus\bar{\mathbb{D}}}\frac{\log|z||dz|^2}{(1+|z|^2)^2}=-\int_{\mathbb{D}}\frac{\log|z||dz|^2}{(1+|z|^2)^2}.
	 	 \end{alignedat}\right.
	 \end{align}
	 By previous computations (Remark \ref{hemisphere}), we have
	 \begin{align}\label{sum2}
	 \left\{\begin{alignedat}{1}
	 	\int_{\mathbb{D}}\frac{4|z|^2|dz|^2}{(1+|z|^2)^2}&=4\pi\log(2)-2\pi\\
	 	\int_{\mathbb{D}}\frac{4\log|z||dz|^2}{(1+|z|^2)^2}&=-2\pi\log(2),
	 	\end{alignedat}\right.
	 \end{align}
	 which shows by \eqref{sum0}, \eqref{sum1} and \eqref{sum2} that the sum \eqref{sum} equals to
	 \begin{align*}
	 	2\int_{\mathbb{D}}\frac{		4|z|^2|dz|^2}{(1+|z|^2)^2}+4\int_{\mathbb{D}}\frac{4\log|z||dz|^2}{(1+|z|^2)^2}+4\pi=2\left(4\pi\log(2)-2\pi\right)+4\left(-2\pi\log(2)\right)+4\pi=0,
	 \end{align*}
	 as expected.

	 \nocite{}
	 \bibliographystyle{plain}
	 \bibliography{biblio}
	 
 \end{document}